\documentclass[]{article}
\usepackage[margin=1in]{geometry}

\RequirePackage{amsthm,amsmath,amssymb,bm}
\RequirePackage[colorlinks,citecolor=blue,urlcolor=blue]{hyperref}

\usepackage{mathtools}

\numberwithin{equation}{section}

\renewcommand{\P}{\mathbb{P}}
\newcommand{\E}{\mathbb{E}}
\newcommand{\Var}{\text{Var}}

\newcommand{\cN}{\mathcal{N}}

\newcommand{\Z}{\mathbb{Z}}
\newcommand{\R}{\mathbb{R}}
\newcommand{\C}{\mathbb{C}}

\def\ball{{\mathsf B}}

\newcommand{\eps}{\varepsilon} 
\newcommand{\vphi}{\varphi}
\def\id{{\mathbf I}}

\renewcommand{\d}{\textup{d}}
\renewcommand{\l}{\vert}

\newcommand{\<}{\langle}
\renewcommand{\>}{\rangle}
\newcommand{\sign}{\text{sign}}
\newcommand{\diag}{\text{diag}}

\newcommand{\op}{{\rm op}}
\newcommand{\ones}{\bm{1}}

\def\sT{{\mathsf T}}
\def\bzero{{\boldsymbol 0}}

\DeclareMathOperator*{\argmin}{arg\,min}
\DeclareMathOperator*{\argmax}{arg\,max}

\def\simiid{{\stackrel{i.i.d.}{\sim}}}

\newtheorem{theorem}{Theorem}[section]
\newtheorem{lemma}[theorem]{Lemma}
\newtheorem{remark}{Remark}[section]
\newtheorem{proposition}[theorem]{Proposition}
\newtheorem{corollary}[theorem]{Corollary}

\def\s1RSB{\mbox{\rm\tiny 1RSB}}
\def\sSK{\mbox{\rm\tiny SK}}
\def\sTAP{\mbox{\rm\tiny TAP}}
\def\sCLPR{\mbox{\rm\tiny CLPR}}
\def\cC{{\mathcal C}}
\def\cF{{\mathcal F}}
\def\cE{{\mathcal E}}

\def\cP{{\mathcal P}}
\def\bsigma{{\boldsymbol \sigma}}

\def\hbx{\widehat{\boldsymbol x}}

\def\be{{\boldsymbol e}}
\def\bg{{\boldsymbol g}}
\def\bbm{{\boldsymbol m}}
\def\bu{{\boldsymbol u}}
\def\bv{{\boldsymbol v}}
\def\bw{{\boldsymbol w}}
\def\bx{{\boldsymbol x}}
\def\by{{\boldsymbol y}}

\def\bA{{\boldsymbol A}}
\def\bB{{\boldsymbol B}}
\def\bD{{\boldsymbol D}}

\def\bG{{\boldsymbol G}}
\def\bH{{\boldsymbol H}}

\def\bO{{\boldsymbol O}}
\def\bU{{\boldsymbol U}}
\def\bW{{\boldsymbol W}}
\def\bX{{\boldsymbol X}}
\def\hbX{\widehat{\boldsymbol X}}
\def\bY{{\boldsymbol Y}}
\def\bZ{{\boldsymbol Z}}

\def\bLambda{{\boldsymbol \Lambda}}
\newcommand{\hg}{\hat{g}}

\def\de{{\rm d}}
\def\magn{{M}}
\def\xathx{{A}}

\def\Const{{C}}
\def\const{{c}}

\def\integers{{\mathbb Z}}
\def\reals{{\mathbb R}}
\def\supp{{\rm supp}}
\def\GOE{{\rm GOE}}
\def\MMSE{\mathrm{MMSE}}
\def\Crit{{\rm Crit}}
\def\sD{{\mathsf D}}
\def\kl{{\rm kl}}
\def\id{{\mathbf I}}
\def\Proj{{\mathsf P}}
\def\diag{{\rm diag}}

\def\cT{{\mathcal T}}
\def\cS{{\mathcal S}}

\newcommand{\arctanh}{\operatorname{arctanh}}
\renewcommand{\Im}{\operatorname{Im}}
\renewcommand{\Re}{\operatorname{Re}}
\newcommand{\ii}{\mathbf{i}}
\newcommand{\Tr}{\operatorname{Tr}}

\renewcommand{\l}{\vert}
\renewcommand{\d}{{\rm d}}

\begin{document}

\title{TAP free energy, spin glasses, and variational inference}

\author{Zhou Fan\footnote{Department of Statistics and Data Science, Yale University} \and Song Mei\footnote{Institute for Computational and Mathematical Engineering, Stanford University} \and Andrea Montanari\footnote{Department of Electrical
Engineering and Department of Statistics, Stanford University}}

\date{}

\maketitle

\begin{abstract}
We consider the Sherrington-Kirkpatrick model of spin glasses with
ferromagnetically biased couplings. For a specific choice of the couplings mean, the resulting Gibbs measure
is equivalent to the Bayesian posterior for a high-dimensional estimation problem known as `$\integers_2$ synchronization'.
Statistical physics suggests to compute the expectation with respect to this Gibbs measure (the posterior mean in the synchronization problem),
by minimizing the so-called Thouless-Anderson-Palmer  (TAP) free energy, instead of the mean field (MF) free energy.
We prove that this identification is correct, provided  the ferromagnetic bias is larger than a constant (i.e. the noise level is small enough in synchronization). Namely, we prove that the scaled $\ell_2$ distance between any low energy local minimizers of the TAP free energy and the mean of the Gibbs measure vanishes in the large size limit. Our proof technique is based on upper bounding the expected number of critical points of the TAP free energy using the Kac-Rice formula. 
\end{abstract}

\section{Introduction and main results}

Computing expectations of a high-dimensional probability distribution is a central problem in computer science, statistics, and statistical physics.
While a number of mathematical and computational techniques have been developed for this task,
fundamental questions remain unanswered even for seemingly simple models.

Consider the problem of estimating a vector $\bx \in \{ +1, -1 \}^n$ from the observation $\bY \in \R^{n \times n}$ given by 
\begin{align}\label{eq:Z2model}
\bY = & \frac{\lambda}{n} \bx \bx^\sT + \bW\, .
\end{align}
Here $\bW$ is an unknown noise matrix, which is distributed according to the Gaussian orthogonal ensemble (GOE).
Namely,  $\bW=\bW^{\sT}\in\reals^{n\times n}$ where $(W_{ij})_{1 \le i\le j\le n}$ are independent centered Gaussian random variables with off-diagonal entries having variance $1/n$ and diagonal entries having variance $2/n$. (In what follows, we will write $\bW\sim\GOE(n)$.) 
The parameter $\lambda \in (0, \infty)$ is assumed to be known and corresponds to a signal-to-noise ratio. 

This problem is known as $\Z_2$ synchronization \cite{singer2011angular} and  is closely related to correlation clustering \cite{bansal2004correlation}.
It  has applications to determining the orientation of a manifold \cite{singer2011orientability}, and  as a simplified model for the two-groups 
stochastic block model \cite{holland1983stochastic} and for  topic models \cite{blei2012probabilistic,ghorbani2018instability}. 
$\Z_2$ synchronization is known to undergo a phase transition at $\lambda=1$: for $\lambda<1$, no estimator
can achieve a correlation with the true signal $\bx$  which is bounded away from zero. Vice versa, for $\lambda>1$,  there exist estimators that achieve
a strictly positive correlation, see e.g.  \cite{montanari2016semidefinite,deshpande2016asymptotic}.

In this paper, we consider the regime $\lambda>1$ and study the optimal estimator. 
Notice that---even in absence of noise---the vector $\bx$ can be determined from the observation $\bY$  only up to a sign. 
In order to resolve this ambiguity, we set as our goal to estimate the  $n\times n$ matrix  of relative signs $\bX=\bx\bx^\sT$.
An estimator will be a map $\hbX:\reals^{n\times n}\to \reals^{n\times n}$, $\bY\mapsto\hbX(\bY)$. We will see that the optimal $\hbX$
is approximately of rank one, i.e. $\hbX(\bY)\approx \bbm_\star \bbm_\star^{\sT}$, and therefore $\bbm_\star\in \reals^n$ can be viewed as
 an estimator of $\pm \bx$. 

By the symmetry of the problem, it is reasonable to consider Bayes estimation with 
respect to the uniform prior\footnote{The Bayes estimator under
uniform prior is minimax optimal over $\bx \in
\{+1,-1\}^n$. }.
Namely, we assume  $\bx\sim  \text{Unif}(\{ +1, -1 \}^n)$, and try to minimize the matrix mean square error, defined by
\begin{align}
\MMSE_n(\lambda)= \inf_{\widehat{\bX}} \frac{1}{n^2}\E\Big[\big\|\widehat \bX(\bY)
-\bx\bx^\sT\big\|_F^2\Big]\, .
\end{align}
The minimum of $\MMSE_n(\lambda)$ is achieved by the posterior
mean\footnote{This follows from Pythagoras's theorem: $\E\big[\big\|\widehat \bX(\bY)
-\bx\bx^\sT\big\|_F^2\big] = \E\big[\big\|\widehat \bX(\bY)-\E[\bx\bx^{\sT}|\bY]\big\|_F^2\big]+\E\big[\big\|\bx\bx^{\sT}
-\E[\bx\bx^\sT|\bY]\big\|_F^2\big]$.} 
\begin{align}
\widehat{\bX}_{\mathrm{Bayes}}= \widehat{\bX}_{\mathrm{Bayes}}(\bY)= \E[\bx\bx^\sT \mid \bY]. 
\end{align}
The asymptotics of the Bayes risk $\lim_{n \to \infty}\MMSE_n(\lambda)$ was calculated in \cite{deshpande2016asymptotic}. 

Computing the posterior expectation requires summing functions over $\bx\in \{+1,-1\}^n$, an example of  the high-dimensional integration
problem mentioned above.  While exact integration is expected to be intractable,
a number of methods have been developed that attempt to
compute the posterior expectation approximately.
Two main strategies are  Markov Chain Monte Carlo (see, e.g., the surveys in  \cite{diaconis2009markov,sinclair2012algorithms,andrieu2003introduction}) and variational inference 
(e.g., \cite{wainwright2008graphical,mezard2009information,blei2017variational}). Markov Chain Monte Carlo usually requires a large number of steps to get an accurate approximation.
Further, despite remarkable mathematical progress, proving Markov chain mixing remains extremely
challenging \cite{levin2017markov}. For instance,  we have near-linear-time algorithms
that estimate the signal $\bx\in\{\pm 1\}^n$ in the model \eqref{eq:Z2model} (and indeed more generally),
with Bayes optimal error \cite{montanari2020estimation}, for any $\lambda>1$. On
the other hand, polynomial mixing
for the posterior $p_{\bx|\bY}$ is---to the best of our knowledge---open for any $\lambda=O(1)$.

Variational inference attempts to compute the marginals of a high dimensional distribution by minimizing a suitable
`free energy' function. This approach is usually faster than MCMC, leading to a broad range of applications, from topic modeling \cite{blei2012probabilistic} to
computer vision, and inference in graphical models \cite{koller2009probabilistic}.
Of course, its accuracy relies in a crucial way on the accuracy of the free energy construction.
For instance, several applications make use of the so-called `naive mean field' free energy  \cite{blei2012probabilistic}. However, it was shown in \cite{ghorbani2018instability} that, 
for topic modeling (and even the simpler $\Z_2$ synchronization problem), naive mean field can return wrong information about the posterior distribution. 
Theorem \ref{THM:NAIVE_FAIL} of this paper confirms this conclusion, by showing that naive mean field does not compute the correct posterior mean for $\Z_2$ synchronization
when $\lambda$ is a large enough constant.

Methods from spin glass theory \cite{mezard1987spin} can overcome these limitations. 
Consider the Gibbs measure of the Sherrington-Kirkpatrick (SK) model of spin glasses \cite{kirkpatrick1978infinite}, with ferromagnetic bias aligned with $\bx$, 
\begin{align}\label{eqn:general_gibbs_measure}
G_{\beta, \lambda}(\bsigma) &=\frac{1}{Z_n(\beta,\lambda)}\,  \exp\{ \beta \<
\bsigma, \bY\bsigma\>/2 \} \nonumber\\
&= \frac{1}{Z_n(\beta,\lambda)} \exp\{\beta \lambda \<\bsigma, \bx\>^2/2 + \beta \<\bsigma, \bW \bsigma\>/2 \}. 
\end{align}
It is easy  to check that the posterior distribution of $\bx$ given observation $\bY$ is given by the biased SK measure with\footnote{In physics language,
the condition $\beta=\lambda$ is referred to as the `Nishimori line' \cite{nishimori2001statistical}.} $\beta = \lambda$
\begin{align}
\label{eq:posterior}
p(\bsigma \l \bY) = G_{\lambda, \lambda}(\bsigma) \propto \exp\big\{ \lambda \<\bsigma, \bY \bsigma\> /2 \big\}. 
\end{align}
Following the statistical physics terminology, we shall refer to the variables $\sigma_1,\dots,\sigma_n$ as `spins'.

For the SK model, naive mean field is known to be a poor approximation of  the actual free energy, which explains the failures mentioned above.
In 1977, Thouless, Anderson and Palmer \cite{thouless1977solution} proposed a variational formula for the SK free energy,
whose decision variables $\bbm = (m_1,\dots, m_n)\in [-1,1]^n$ encode the expectation of the spins $\bsigma$. This variational formula is known as the TAP free
 energy, and its first-order stationarity conditions are known as the TAP equations. 
The relationship between the TAP equations and the Gibbs measure has been
studied in the physics and mathematics literature for 
the last 40 years, and still presents a number of outstanding challenges. A brief overview is presented in Section \ref{sec:literature}. 

Explicitly, the TAP free energy for the SK model (\ref{eqn:general_gibbs_measure}) is the function $\cF_{\beta, \lambda}: [-1, 1]^n \to \R$, 
defined by
\begin{align}\label{eqn:TAP_free_energy}
\cF_{\beta, \lambda}(\bbm) = -\frac{1}{n} \sum_{i=1}^n h(m_i) - \frac{\beta}{2n}\< \bY, \bbm
\bbm^\sT\>- \frac{\beta^2}{4} [1-Q(\bbm)]^2,
\end{align}
where $Q(\bbm) = \Vert \bbm \Vert_2^2/n$, and $h: [-1, 1] \to \R$ is the binary entropy function
\[
h(m) = - \frac{1 + m}{2} \log\Big( \frac{1 + m}{2} \Big) - \frac{1 - m}{2} \log\Big( \frac{1 - m}{2} \Big).
\]
The corresponding TAP equations (first order stationary condition for $\cF_{\beta, \lambda}$) are given by
\begin{align}\label{eqn:TAP_equations}
\bbm=\tanh\big(\beta \bY \bbm -\beta^2[1-Q(\bbm)]\bbm\big). 
\end{align}

The first two terms in Eq.~(\ref{eqn:TAP_free_energy}) correspond to the naive mean field free energy, while the term 
$-\beta^2 (1-Q(\bbm))^2/4$ is known as `Onsager's correction,' and is the main innovation introduced in \cite{thouless1977solution}.
Indeed, heuristically, the TAP free energy can be understood using Plefka's expansion \cite{plefka1982convergence}, which provides 
a series expansion of the log partition function. The first three terms of the expansion correspond to $\cF_{\beta, \lambda}(\bbm)$,
and the others are expected to be negligible as $n\to\infty$.

Statistical physicists suggest an ambitious general conjecture on the role of the TAP free energy
\cite{mezard1987spin}.
Namely, the Gibbs measure (\ref{eqn:general_gibbs_measure}) is a convex combination
of pure states, which are well-concentrated probability
measures\footnote{A pure state $G_{\beta, \lambda}^{\alpha}$ can be assumed to be
  a probability measure of the form $G_{\beta, \lambda}^{\alpha} (\bsigma)\propto
G_{\beta, \lambda}(\bsigma) \, {\bf 1}_{\bsigma\in\Omega_{\alpha}}$.}
with each pure state $\alpha$ corresponding  to a local minimum $\bbm^{\alpha}$ of the TAP free energy.
We refer to Section \ref{sec:literature} for further pointers to this line of work.

Understanding the geometry of the critical points of the TAP free energy $\cF_{\beta, \lambda}$ would help us elucidate the structure of the 
Gibbs measure $G_{\beta, \lambda}$. 
A remarkable sequence of papers within the statistical mechanics literature \cite{bray1980metastable,parisi1995number,cavagna2003formal,crisanti2003complexity,crisanti2004spin,crisanti2005complexity} computed the expected number of critical  points of the TAP free energy for the SK model ($\lambda = 0$) using non-rigorous but sophisticated tools from physics. 
More precisely, these authors obtain the exponential
growth rate of this expectation, as a function of the free energy (the value of $\cF_{\beta, \lambda}(\bbm)$), a quantity that is known as the `complexity' of the spin model. 
In our first main result, we prove rigorously that the formula given by \cite{crisanti2005complexity} is indeed an upper bound for the expected number of critical points, 
in the general setting $\lambda \ge 0$. To the best of our knowledge,  similar rigorous calculations of the complexity were
only obtained so far for (mixed) $p$-spin spherical models, and for models on the torus. These models are simpler because of the
underlying symmetry \cite{auffinger2013random,auffinger2013complexity,auffinger2011random,arous2019landscape}.

In the context of the Bayesian model (\ref{eq:posterior}), the above picture simplifies.
Adopting the jargon of statistical physics, the model is known to be  replica symmetric 
\cite{deshpande2016asymptotic}, which amounts to say that the Gibbs measure $G_{\lambda,\lambda}(\bsigma)$
is well approximated by a single pure state.
As a consequence, we expect the global minimum of $\cF_{\lambda,\lambda}$ to correspond to the 
posterior expectation. Our second main result proves this conjecture for all $\lambda$ larger than a big enough constant. 
This implies that minimizing the TAP free energy is  a viable approach to optimal estimation in the present model. 

While the  relevance of the TAP free energy for statistical inference was pointed out several times in the past (see Section \ref{sec:literature}),
this is the first rigorous result confirming this connection.

\subsection{The complexity of the TAP free energy}

For $U \subseteq\R^4$ and $V \subseteq (-1, 1)$, define the number of critical points of $\cF_{\beta, \lambda}$ with $\bbm$ in the region defined by $U$ and $V$ by
\begin{equation}\label{eq:Critn}
\Crit_n(U, V) = \sum_{\bbm: \nabla \cF_{\beta, \lambda}(\bbm) = \bzero} \ones\{ (Q(\bbm), \magn(\bbm),
\xathx(\bbm), E(\bbm)) \in U,\;\bbm \in V^n \},
\end{equation}
where $Q, \magn, \xathx, E$ are the functions 
\begin{equation}\label{eqn:statistics_of_spins}
\begin{aligned}
Q(\bbm) =& \frac{1}{n} \Vert \bbm \Vert_2^2,\\
\magn(\bbm) =& \frac{1}{n}\<\bx, \bbm\>, \\
\xathx(\bbm) =&\frac{1}{n} \sum_{i=1}^n m_i \arctanh(m_i),\\
E(\bbm) =& - \frac{1}{n}\sum_{i=1}^n \Big[h(m_i) + \frac{1}{2}\, m_i\, \arctanh(m_i) \Big] - \frac{\beta^2}{4} (1 - Q(\bbm)^2).
\end{aligned}
\end{equation}
Note that at any point $\bbm$ where $\nabla \cF_{\beta,
  \lambda}(\bbm)=\bzero$, we have  from Eq.\ (\ref{eqn:TAP_equations}) that $\cF_{\beta, \lambda}(\bbm)=E(\bbm)$. 

Define the function $S_\star: (0, 1] \times [-1, 1] \times \R^2 \to \R$ by
\begin{equation}
S_\star(q, \vphi, a, e) = \inf_{(\mu, \nu, \tau, \gamma) \in \R^4}
S(q, \vphi, a, e; \mu, \nu, \tau, \gamma),
\end{equation}
where
\begin{align}
S(q, \vphi, a, e;\mu, \nu, \tau, \gamma)
&=\frac{1}{4\beta^2}\Big[\frac{a}{q} - \frac{\beta \lambda \vphi^2}{q} -
\beta^2(1-q)\Big]^2\nonumber\\
&- q \mu - \vphi \nu - a \tau - [ u(q, a) - e] \gamma +
\log I(q,\vphi,\mu,\nu,\tau,\gamma)\,,
\end{align}
\begin{equation}\label{eq:uqa}
u(q, a) = - \frac{\beta^2}{4} (1 - q^2) + \frac{a}{2},
\end{equation}
and
\begin{align}
I &(q,\vphi,\mu,\nu,\tau,\gamma)= \int_{-\infty}^\infty \frac{1}{(2 \pi \beta^2 q)^{1/2}}
\exp\Big\{-\frac{(x - \beta \lambda \vphi)^2}{2 \beta^2 q}\nonumber\\
&\hspace{0.3in}+ \mu \tanh^2(x) + \nu
\tanh(x) + \tau x \tanh(x) + \gamma \log[2\cosh(x)] \Big\} \d x.\label{eq:I}
\end{align}

\begin{theorem}\label{thm:Kac_rice_complexity}
Fix any $\beta>0$, $\lambda \ge 0$, $\eta, b \in (0, 1)$, and closed
set $U \subseteq  [\eta, 1]
\times \R^3$. Let $V_n = [-1 + e^{-n^b}, 1 - e^{-n^b}]$. Then
\[\limsup_{n \to \infty} n^{-1} \log \E[\Crit_n(U, V_n)] \le \sup_{(q, \vphi, a,
e) \in U} S_\star(q, \vphi, a, e).\]
\end{theorem}

The main proof strategy of Theorem \ref{thm:Kac_rice_complexity} is to calculate $\E[\Crit_n(U, V_n)]$ using the Kac-Rice formula \cite[Theorem 11.2.1]{adlertaylor}. 
This is the same strategy as pioneered by Auffinger, Ben Arous and \v{C}ern\`y in \cite{auffinger2013random}
(building on early insights by Fyodorov \cite{fyodorov2004complexity}) in the context of the spherical $p$-spin model. However, several new challenges arise. 
First of all, we cannot adopt  \cite[Theorem 11.2.1]{adlertaylor} directly
because of a degeneracy of the conditional Hessian. We present a proof for the
Kac-Rice upper bound with degeneracy in Appendix~\ref{sec:Proof_Basic_Kac_Rice}. 

We then evaluate the Kac-Rice formula to leading exponential order. As usual,
the  most difficult step requires to evaluate the  expected determinant of the Hessian.
In the spherical model, this is distributed as the determinant of a Wigner matrix, shifted by a term proportional to the identity, and 
an exact formula was given in  \cite{auffinger2013random}. A slightly more
complicated calculation arises for the $p$-spin spherical model with biased coupling, see
\cite{arous2019landscape}, where the Hessian is a low-rank deformation of a Wigner matrix, shifted by a term proportional to the identity. 
In the present case, the Hessian is distributed as a low-rank deformation of a Wigner matrix \emph{plus a diagonal matrix},
which depends on the point $\bbm\in[-1,+1]^n$. Unlike in earlier work, the bulk of the spectral distribution of the Hessian depends now on the point $\bbm$ (instead of $\| \bbm \|_2$).
 We give an upper bound of the expected determinant of the Hessian using free
probability theory in Section \ref{sec:Hessian}. Finally, we 
use this expression to upper bound the exponential growth rate of $\E[\Crit_n(U, V_n)]$ using Sanov's Theorem and Varadhan's Lemma in Appendix
\ref{sec:Proof_Kac_Rice_Abstract}. We refer to Section \ref{sec:Describe_Prop_Kac_Rice} for details of the proof. 

\begin{remark}
 We impose the technical conditions that $U \subseteq  [\eta, 1]
 \times \R^3$ and $V_n =  [-1 + e^{-n^b}, 1 - e^{-n^b}]$ 
because of the singular behaviors of the TAP free energy at zero and at the boundary, which are subtle to deal with using the Kac-Rice formula. Instead, the behaviors of $\cF_{\beta, \lambda}$ near zero and near the boundary can be easily analyzed
by direct arguments. 

Indeed, $\bbm = \bzero$ is always a critical point of $\cF_{\beta, \lambda}$. In
one interesting case $\beta > 1$ and $\lambda = 0$, it can be easily shown that
$\cF_{\beta, \lambda}$ is strongly convex over $D_\eta \equiv \{ \bbm \in [-1,
1]^n : Q(\bbm) \le \eta\}$ with high probability, so that there are no other critical points of $\cF_{\beta, \lambda}$ inside $D_\eta$. In another interesting case,
namely when $\beta = \lambda$ are large enough, the critical points inside $D_\eta$ have high function value with high probability
 (see Lemma \ref{lem:initial_global_bound}), so that they are not relevant. 

Near the boundary, $\nabla \cF_{\beta, \lambda}$ diverges due to the entropy
term. It is easy to show that there are no critical points of $\cF_{\beta,
\lambda}$ within the region $\{ \bbm: \vert \Vert \bbm \Vert_\infty - 1\l \le
\exp(- n^{2/3}) \}$ with high probability for sufficiently large $n$. (See Lemma \ref{lem:infinity_norm_high_probability_bound} for the proof of the case $\beta = \lambda$). 
\end{remark}
\subsection{Consequences for SK model with zero-mean couplings ($\lambda=0$)}

It is interesting to specialize Theorem \ref{thm:Kac_rice_complexity} to
$\lambda=0$ (the SK model with zero-mean couplings), and analytically maximize over $\varphi$ (the
value of the magnetization), which is equivalent to setting $\nu=\varphi=0$. Finally we replace the variable $a$ by $\Delta$ via $a=\beta^2q(1-q)+2q\Delta$.
This results in the reduced expression
\begin{equation}
S_{0,\star,\beta}(q, \Delta, e) = \inf_{(\mu,\tau,\gamma)\in\reals^3} S_0(q, \Delta, e;\mu, \tau, \gamma)\, ,\label{eq:S0star}
\end{equation}
\begin{align}
S_0(q, \Delta, e;\mu, \tau, \gamma) &=
\frac{\Delta^2}{\beta^2}-q\mu+e\gamma\nonumber-\Big(\tau+\frac{\gamma}{2}\Big)
\big(\beta^2q(1-q)+2q\Delta\big)\\
&\hspace{0.2in}+\frac{\beta^2}{4}(1-q^2)\gamma+\log I_0(q,\mu,\tau,\gamma)\, ,\label{eq:S0}
\end{align}
\begin{align}
I_0&(q,\mu,\tau,\gamma)=\int_{-\infty}^\infty \frac{1}{(2 \pi \beta^2 q)^{1/2}}
\exp\Big\{-\frac{x^2}{2 \beta^2 q}\nonumber\\
&\hspace{0.3in}+ \mu \tanh^2(x) +\tau x \tanh(x) + \gamma \log[2\cosh(x)] \Big\} \d x.
\end{align}
This coincides with the expression in the statistical physics literature,  
cf. \cite{bray1980metastable,parisi1995number,cavagna2003formal,crisanti2003complexity,crisanti2004spin,crisanti2005complexity}. We refer in particular to 
\cite{crisanti2003complexity} which compares different theoretical physics
approaches. Equations (29)-(31) of \cite{crisanti2003complexity}  can be recovered from the above expression by setting $\mu = \lambda_{\sCLPR}-(\Delta^2/(2\beta^2 q))$, $\gamma = -u_{\sCLPR}$, $\tau =  (u_{\sCLPR}/2)+(\Delta /(\beta^2 q))$.
 Minimization over $(\mu,\tau,\gamma)$ is then replaced by minimization over $(\lambda_{\sCLPR},u_{\sCLPR})$.

Notice that the expression of \cite{crisanti2003complexity} involves one extra parameter (denoted by $B$). However, the authors set it to $0$ on the basis 
of physical considerations motivated by `Plefka's criterion' \cite{plefka1982convergence}.

Denote by $H_{\sSK}(\bsigma) = -\<\bsigma,\bW\bsigma\>/2$ (where $\bsigma\in\{+1,-1\}^n$) 
the Hamiltonian of 
the SK model.  Substituting $\bbm=\bsigma$ in Eq.~(\ref{eqn:TAP_free_energy}), we obtain the following lower bound on the ground state energy:
\begin{align}
 \frac{1}{n}\min_{\bsigma\in \{-1,+1\}^n} H_{\sSK}(\bsigma) \ge \frac{1}{\beta}\min_{\bbm\in [-1,+1]^n}\cF_{\beta,0}(\bbm)\, . 
\end{align}
%
Note that $(1/n)\min_{\bsigma\in \{-1,+1\}^n} H_{\sSK}(\bsigma)$ concentrates exponentially around its expectation by Gaussian concentration \cite{BLM}.
Using Markov's inequality and the fact that local minima of $\cF_{\beta,0}(\bbm)$ occur in the interior of $[-1,+1]^n$ (see Lemma 
\ref{lem:infinity_norm_high_probability_bound}), we get the following lower bound on the expected ground state energy
\begin{align}
&\lim\inf_{n\to\infty} \frac{1}{n}\E\min_{\bsigma\in \{-1,+1\}^n}
H_{\sSK}(\bsigma)\nonumber\\
&\quad\ge  F_{\s1RSB}(\beta)\equiv \frac{1}{\beta}\inf\big\{ e:\;\; \sup_{q,\Delta} S_{0,\star,\beta}(q, \Delta, e)\ge 0\big\}\, .
\end{align}
(Here $S_{0,\star,\beta}(q, \Delta, e)$ is given by Eq.~(\ref{eq:S0star}).)
In  \cite{cavagna2003formal}, Cavagna, Giardina, Parisi, and  M{\'e}zard identify $F_{\s1RSB}(\beta)$ with the `one step replica symmetry breaking' (1RSB)
formula for the free energy of the SK model. Hence, our result provides an alternative route to prove the celebrated 1RSB 
lower bound---first established by Guerra in \cite{guerra2003broken}---in the zero temperature ($\beta=\infty$) case. 

\subsection{Bayes estimation in $\mathbb Z_2$ synchronization}
We now return to the $\integers_2$ synchronization model of Eq.~(\ref{eq:Z2model}). 
As mentioned above, the posterior distribution of $\bx$ given observation $\bY$ is given (under the uniform prior) by the biased SK Gibbs measure 
(\ref{eqn:general_gibbs_measure}) with $\beta = \lambda$, cf. Eq.~(\ref{eq:posterior}). Accordingly, we fix $\beta = \lambda$ throughout this section. 

Given the relation between the TAP free energy and the Bayes posterior, it is natural to consider the estimator
\begin{align}
\hbx^{\sTAP}(\bY) \equiv\argmin_{\bbm\in [-1,+1]^n}\cF_{\lambda,\lambda}(\bbm)\, .\label{eq:TAP_estimator}
\end{align}
Our next theorem provides a characterization of the landscape of the TAP free energy $\cF_{\lambda, \lambda}$  for $\lambda$ a sufficiently large constant,
implying in particular that $\hbx^{\sTAP}(\bY)$ is close to the Bayes optimal estimator. 
More precisely, we show that any critical point $\bbm$ for which $\cF_{\lambda, \lambda}(\bbm)$ is below a certain constant energy value
is such that $\bbm\bbm^\sT\approx \widehat{\bX}_{\mathrm{Bayes}}$. This implies that any such $\bbm\bbm^\sT$ also asymptotically achieves the Bayes risk $\MMSE(\lambda)$.
\begin{theorem}\label{thm:TAPBayes}
Denote
\[\cC_{\lambda, n} = \{ \bbm \in [-1, 1]^n: \nabla \cF_{\lambda, \lambda}(\bbm)
= \bzero, \cF_{\lambda, \lambda}(\bbm) \le -\lambda^2/3\}.\]
There exists a constant $\lambda_0>0$, such that for any constant
$\lambda>\lambda_0$, we have $\cC_{\lambda, n} \neq \emptyset$ with high probability, and
\begin{align}\label{eq:nearBayes}
\lim_{n \to \infty}  \E\Big[\Big(\sup_{\bbm \in \cC_{\lambda, n}} \frac{1}{n^2}
\|\bbm \bbm^\sT-\widehat{\bX}_{\mathrm{Bayes}}\|_F^2 \Big)\ones\{\cC_{\lambda,n}
\neq \emptyset\}\Big] =0.
\end{align}
\end{theorem}
We refer to Section \ref{sec:proof_thm_TAPBayes} for the proof of this theorem. 

This theorem suggests
that a good minimizer of $\cF_{\lambda,\lambda}$ can be computed by applying a gradient based
optimization algorithm, provided the initialization has a value of $\cF_{\lambda,\lambda}$ below
a constant.
It is easy to construct an initialization satisfying the condition required here. For instance, if
 $\bv_1(\bY)$ is the principal eigenvector of $\bY$ with unit norm, we can set $\bbm^0=\sign(\bv_1(\bY))$.
 A simple calculation yields $\cF_{\lambda,\lambda}(\bbm^0) = -(\lambda/2n)\<\bY,\bbm^0(\bbm^0)^{\sT}\> \le
 -(\lambda^2/2n^2)\<\bbm^0,\bx\>^2 +\lambda$ (where the upper bound holds with high probability,
 from $\|\bW\|_{\op}\le 2$ \cite{furedi1981eigenvalues}). Using standard results on the eigenvectors of
 deformed GOE matrices \cite{baik2005phase}, we get $|\<\bbm^0,\bx\>|/n\ge 1-\eps$ for all $\lambda>\lambda_0(\eps)$
 whence $\cF_{\lambda,\lambda}(\bbm^0) <-\lambda^2/3$ with high probability for all $\lambda$ large enough.
 
\begin{remark}
  Since the TAP free energy is non-convex, the above does not prove that it can be optimized efficiently:
  further information about its Hessian would be required to establish this.
However, the papers \cite{deshpande2016asymptotic,montanari2020estimation} develop 
an approximate message passing (AMP) algorithm that attempts to construct the solution of the TAP equation (\ref{eqn:TAP_equations}). 
The algorithm is iterative, and proceeds as follows, with $\bbm^{-1} = \bzero$:
\begin{equation}\label{eqn:AMP}
\begin{aligned}
\bbm^0 = & \tanh\big( c_0(\lambda)\sqrt{n}\, \bv_1(\bY))\, ,\\
\bbm^{k+1}  =& \tanh(\lambda \bY \bbm^k - \lambda^2 [1 - Q(\bbm^k)] \bbm^{k-1})\, .\\
\end{aligned}
\end{equation}
Here $\bv_1(\bY)$ is the principal eigenvector of $\bY$ with unit norm, and $c_0(\lambda)$ is a suitable normalization (see \cite{montanari2020estimation}).
Building on earlier work by Bolthausen \cite{bolthausen2014iterative}, as well as on
\cite{deshpande2016asymptotic},  it is shown in \cite{montanari2020estimation} that the iterates $\bbm^k$ asymptotically achieve the Bayes risk, for any $\lambda>1$,
\begin{align}
\lim_{k \to \infty} \lim_{n \to \infty} \frac{1}{n^2} \| \bbm^k (\bbm^k)^\sT - \bx \bx^\sT \|_F^2 = \lim_{n \to \infty} \MMSE_n(\lambda). 
\end{align}
Together with  Theorem \ref{thm:TAPBayes}, this result implies that the algorithm is constructing an approximation of $\hbx^{\sTAP}(\bY)$: 
\begin{align}
 \lim_{k \to \infty} \lim_{n \to \infty} \E\Big[ \frac{1}{n^2} \| \bbm^k (\bbm^k)^\sT - \hbx^{\sTAP}( \hbx^{\sTAP})^\sT \|_F^2 \Big] = 0\, .
\end{align}
This leaves open two interesting questions: $(i)$~Can we construct an algorithm that converges to $\hbx^{\sTAP}(\bY)$, for fixed $n$
(i.e. can we invert the order of limits in the last equation)? $(ii)$~Do generic optimization algorithms (e.g., gradient descent) converge to $\hbx^{\sTAP}(\bY)$, when
applied to $\cF_{\lambda,\lambda}$?
\end{remark}

\begin{remark}\label{rmk:OptVsMessagePassing}
While---as discussed in the last remark---the AMP algorithm of  \cite{montanari2020estimation} achieves the Bayes risk,
the optimization-based formulation of Eq.~(\ref{eq:TAP_estimator}) has important practical advantages. 
In particular, 
we do not know  how robust is AMP with respect to deviations from the model (\ref{eq:Z2model}).
On the other hand, we  expect $\hbx^{\sTAP}(\bY)$ to be robust and meaningful even if the model for the data $\bY$ is incorrect.
\end{remark}

The TAP free energy (\ref{eqn:TAP_free_energy}) may be contrasted with the
``mean field free energy"
\[\cF_\mathrm{MF}(\bbm)=-\frac{1}{n}\sum_{i=1}^n h(m_i)-\frac{\lambda}{2n}
\< \bbm,\bY \bbm\>,\]
whose stationary points satisfy the mean field equations
\begin{align}
\bbm=\tanh(\lambda \bY \bbm)\, .
\end{align}
Notice that the mean field equations omit the Onsager correction term $-\lambda^2[1-Q(\bbm)]\bbm$. To clarify
the importance of this correction term, we establish the following negative
result for critical points $\bbm$ of $\cF_\mathrm{MF}$.
\begin{theorem}\label{THM:NAIVE_FAIL}
Denote $\cS_{\lambda, n} = \{ \bbm \in (-1, 1)^n: \nabla \cF_{\rm MF}(\bbm) = \bzero\}$. There exists a constant $\lambda_0>0$ and a constant $\eps(\lambda)>0$ for every
$\lambda>\lambda_0$, such that
\begin{align}\label{eqn:naive_fail}
\lim_{n\to \infty}\P\Big(\inf_{\bbm \in \cS_{\lambda, n}}\frac{1}{n^2}\|\bbm\bbm^\sT-\widehat{\bX}_\mathrm{Bayes}\|_F^2\ge \eps(\lambda)\Big) = 1.
\end{align}
\end{theorem}
The proof of this theorem is presented in Appendix \ref{sec:proof_thm_naive_fail}.

\subsection{Notations}

We use boldface for vectors and matrices, e.g. $\bX$ for a matrix and $\bx$ for
a vector. For a univariate function $f: \R \to \R$ and a vector $\bx = (x_1, x_2, \ldots, x_n)^\sT \in \R^n$, we define the function
$f: \R^n \to \R^n$ by $f(\bx) = (f(x_1), f(x_2), \ldots, f(x_n))^\sT$. For example, we write $\bx^2 = (x_1^2, x_2^2, \ldots, x_n^2)^\sT$. 
We denote $\ball^n(\bx, r)$ to be the Euclidean ball in $\R^n$ centered at $\bx \in \R^n$ with radius $r$.

\section{Related literature}\label{sec:literature}

The Sherrington-Kirkpatrick (SK) model \cite{kirkpatrick1978infinite} is a mean field model of spin glasses. 
The asymptotics of its log-partition function $\lim_{n\to\infty}n^{-1}\log Z_n(\beta,\lambda)$ (cf. Eq.~(\ref{eqn:general_gibbs_measure})) was first computed using the replica method 
\cite{parisi1979infinite,parisi1980sequence,parisi1983order}, and is expressed using the so-called Parisi's formula. 
Thouless, Anderson, and Palmer \cite{thouless1977solution} proposed a variational principle in terms of the
TAP free energy of Eq.~(\ref{eqn:TAP_free_energy}). 

The relationship between the Gibbs measure (\ref{eqn:general_gibbs_measure})
(and the partition function) and the TAP free energy  has been studied in the physics literature 
\cite{bray1984weighted, cavagna2003formal, de1983weighted}. The local minimizers of the TAP free energy are  interpreted as the local magnetizations of  pure states.
As a partial confirmation of this prediction, Talagrand \cite{talagrand2010mean} and Chatterjee \cite{chatterjee2010spin} proved that the TAP equation holds for local 
magnetizations at high temperature (up to a small error term). 
Bolthausen \cite{bolthausen2014iterative} derived an iterative algorithm for
solving the TAP equations at high temperature. 
This is an instance of a broader class of iterative algorithms that are now known as approximate message passing (AMP), see e.g. \cite{bayati2011dynamics}.
At the low temperature regime, Auffinger and Jagannath \cite{auffinger2016thouless} proved that the TAP equations are
satisfied---approximately---by local magnetizations of pure states. Recently, Chen and Panchenko
\cite{chen2017tap} proved that the minimum of the TAP free energy  (under a suitable constraint on $\bbm$)
coincides indeed with the normalized log partition function $n^{-1}\log Z_n(\beta,\lambda)$ up to an error vanishing as $n\to\infty$. 

The exponential growth rate of the expected number of critical points of the TAP free energy is also known as the `annealed complexity.'
It was first computed by Bray and Moore in \cite{bray1980metastable} using the replica method, while an alternative `supersymmetric' approach 
was developed in \cite{cavagna2003formal}. 
The two methods give equivalent formal expressions, which correspond to the function $S_0$ of Eq.~(\ref{eq:S0}). 
These two papers however report different results for the overall complexity as a function of the free energy, which follow from two different
ways to set the Lagrange parameters in $S_0$.  This discrepancy was further investigated in \cite{crisanti2003complexity} and 
\cite{crisanti2005complexity} which suggest that the specific choice should depend on the value of the free energy. This is
consistent with the prescription of our Theorem \ref{thm:Kac_rice_complexity}.

So far, the annealed complexity function has been computed rigorously  for the spherical $p$-spin model
\cite{auffinger2013random}, for the spherical mixed $p$-spin model \cite{auffinger2013complexity},
and for certain Gaussian energy functions over the torus \cite{auffinger2011random}. For the spherical
$p$-spin  model, 
Subag \cite{subag2017complexity} used the second moment method to show that the typical number
of critical points coincides with its expected value for energies in a certain interval.
A version of the spherical $p$-spin model with non-vanishing mean
was studied in \cite{arous2019landscape}. This model has  a statistical interpretation
analogous to the one of model \eqref{eq:Z2model}.
As mentioned above, the main technical challenge in extending these results to the SK model is to compute
the conditional expectation of the absolute value of the determinant of the Hessian. 
(Let us mention that ---apart from using non-rigorous methods---all of the physics derivations drop this absolute value, without justification.)
Our calculation of this quantity relies on tools from free probability theory \cite{voiculescu1991limit,biane1997free,capitaine2011free}. 

The $\Z_2$ synchronization problem has proven to be a  useful testing ground for comparing different approaches.
The Bayes risk was computed in \cite{deshpande2016asymptotic}. In particular, a non-trivial correlation with the signal can only be achieved for $\lambda>1$.
The ideal threshold $\lambda=1$ is achieved both by spectral methods \cite{baik2005phase} and semidefinite programming  \cite{montanari2016semidefinite,javanmard2016phase}.
However, neither of these approaches achieves the Bayes optimal error. 

Over the last few years, a substantial amount of work has been devoted to models that generalize
$\integers_2$ synchronization \eqref{eq:Z2model} in significant ways. We refer to
\cite{dia2016mutual,lelarge2019fundamental,barbier2019adaptive,miolane2017fundamental,perry2018optimality,montanari2020estimation,alaoui2017finite}
for a few pointers to this literature. Among other settings, these authors consider the model $\bY = (\lambda/n)\bX\bX^{\sT}+\bW$,
where $\bX\in\reals^{n\times r}$ is a signal with fixed rank $r$ and i.i.d. rows $(\bX_{\cdot,i})_{i\le n}$ with known distribution
$p_\bX$ over $\reals^r$. Both Bayes optimal error and the performance of optimal message passing algorithms have been characterized in
this context. The goal of the present paper is quite different: we want to understand the optimality of an optimization-based method,
which uses the TAP free energy as the cost function. 
As explained in Remark \ref{rmk:OptVsMessagePassing}, such an approach has significant practical  advantages.

Variational methods based on free-energy  approximations have been the object of recent interest within the statistics literature. 
Zhang and Zhou \cite{zhang2017theoretical} study  mean field variational inference applied to the stochastic block model.  
They prove optimal error rates in a regime that is equivalent to the $\lambda \to \infty$ regime for $\Z_2$ synchronization.
In the opposite direction, \cite{ghorbani2018instability} shows that for $\lambda \in (1/2, 1)$,
the minimizer of the mean field free energy does not approximate the Bayes posterior mean. Replacing the mean field with
the TAP free energy solves this problem.

\section{Proof of Theorem \ref{thm:Kac_rice_complexity}}\label{sec:Describe_Prop_Kac_Rice}

In this section we prove Theorem \ref{thm:Kac_rice_complexity}.
i
By symmetry with respect to the sign change for any coordinate
\[x_i \mapsto -x_i, \qquad m_i \mapsto -m_i,\]
we may assume without loss of generality throughout the proof of
Theorem~\ref{thm:Kac_rice_complexity} that $\bx=\ones$.
Let $f_n(\bbm)=n\cF_{\beta, \lambda}(\bbm)$ be the scaled TAP free energy.
Then $f_n$ and its gradient and Hessian are given by
\begin{align}
f_n(\bbm) &= - \sum_{i=1}^n h(m_i) - \frac{\beta \lambda}{2n} \< \ones, \bbm \>^2 -
\frac{\beta}{2} \< \bW, \bbm \bbm^\sT\> - \frac{n \beta^2}{4}
[1-Q(\bbm)]^2,\\
\bg_n(\bbm) &=  \arctanh(\bbm) - \frac{\beta \lambda}{n} \, \<\ones, \bbm\> \ones - \beta \bW \bbm + \beta^2 [1 - Q(\bbm)] \bbm,\label{eq:gn}\\
\bH_n(\bbm)&= \diag(1 / (1 - \bbm^2)) - \frac{\beta \lambda}{n} \ones \ones^\sT - \beta \bW + \beta^2 [1 - Q(\bbm)] \id_n - \frac{2 \beta^2}{n} \bbm
\bbm^\sT.\label{eq:Hn}
\end{align}

\subsection{Kac-Rice formula}

We prove the following statement in Appendix~\ref{sec:Proof_Basic_Kac_Rice},
using a version of the Kac-Rice formula.
\begin{proposition}\label{PROP:KACRICE}
Fix any $\beta>0$ and $\lambda \ge 0$.
Denote by $p_{\bbm}(\bzero)$ the Lebesgue density of $\bg_n (\bbm)$ at $\bzero$. Then for any $\delta \equiv \delta_n > 0$ and any $U \subseteq  [\delta, 1] \times \R^3 $ and $V \subseteq [-1 + \delta, 1 - \delta]$, we have
\begin{align}
\E[\Crit_n(U, V)] &\le \int_{V^n} \ones_{\{ (Q(\bbm), M(\bbm), A(\bbm), E(\bbm))
\in U \}}\nonumber\\
&\hspace{0.4in}\E\big[\vert \det(\bH_n(\bbm)) \vert  \big\vert \bg_n (\bbm) =
\bzero\big] p_{\bbm}(\bzero) \d \bbm.\label{eq:KRupperbound}
\end{align}
\end{proposition}

\subsection{Determinant of Hessian and density of gradient}

We next analyze the expectation on the right side of
(\ref{eq:KRupperbound}) using techniques from random matrix theory.
The proof of the following result is deferred to Section \ref{sec:Hessian}.

\begin{proposition}\label{PROP:LOGDETBOUND}
Fix $\beta>0$, $\lambda\ge 0$, and $\eta,b \in (0,1)$. Then there are constants
$n_0 \equiv n_0(\beta, \lambda,\eta,b)>0$ and $C_0 \equiv C_0(\beta, \lambda, \eta, b)$, such that the following holds for all
$n \geq n_0$.
Let $\bbm$ satisfy $\|\bbm\|_\infty \le 1 - e^{-n^b}$ and
$Q(\bbm) \in [\eta, 1]$. Define
\begin{equation}\label{eq:Lm}
L(\bbm)=\frac{\beta^2[1-Q(\bbm)]^2}{2}+\frac{1}{n}
\sum_{i=1}^n \log\left(\frac{1}{1-m_i^2}\right).
\end{equation}
Then
\[
\E\big[\l \det (\bH_n(\bbm)) \l \big| \bg_n(\bbm) = \bzero\big] \leq \exp(n
\cdot L(\bbm)+ C_0  \cdot n^{\max(0.9,b)}).
\]
\end{proposition}

Define
\begin{align*}
\bu(\bbm) &= [u_1(\bbm), \ldots, u_n(\bbm)]^\sT\\
&= \arctanh(\bbm) - \beta\lambda/n \cdot \<\ones, \bbm \> \cdot \ones +
\beta^2[1 - Q(\bbm)]\bbm.
\end{align*}
The quantity $p_{\bbm}(\bzero)$ in (\ref{eq:KRupperbound}) may also be
evaluated as follows.

\begin{lemma}\label{lem:density_of_gradient}
\begin{align*}
p_{\bbm}(\bzero) &= \int_\R \frac{1}{\{ 2 \pi \beta^2 Q(\bbm) \}^{n/2}  }
\exp\Big\{ -\sum_{i=1}^n \frac{(u_i(\bbm) - y m_i)^2}{2  \beta^2  Q(\bbm)}
\Big\}\\
&\hspace{0.5in}\cdot \sqrt{\frac{n}{2\pi \beta^2}}\exp\Big\{ - \frac{n
y^2}{2\beta^2} \Big\} \d y.
\end{align*}
\end{lemma}

\begin{proof}

First, we observe that 
\[\E[ (\bW\bbm)_i (\bW\bbm)_j ] = Q(\bbm) \cdot \ones\{i=j\} + m_i m_j/n.\]
Therefore, we have 
\[
\bW\bbm ~\stackrel{d}{=} ~ Q(\bbm)^{1/2} \bv + Y \cdot \bbm,
\]
where $\bv \sim \cN(0, \id_n)$ and $Y \sim \cN(0,1/n)$ are independent. As a
consequence, by (\ref{eq:gn}),
\begin{align*}
\bg_n(\bbm) ~\stackrel{d}{=} & ~\bu(\bbm) - \beta (Q(\bbm)^{1/2} \bv +
Y \cdot \bbm ),
\end{align*}
and the form of $p_{\bbm}(\bzero)$ follows by first writing the
density of $\bg_n(\bbm)$ conditional on $\beta Y$, and then integrating over
the law of $\beta Y$.
\end{proof}

Combining Proposition \ref{PROP:KACRICE}, Proposition
\ref{PROP:LOGDETBOUND}, and Lemma \ref{lem:density_of_gradient} above,
we have the following immediate corollary.

\begin{corollary}\label{cor:kacrice}
Fix $\beta>0$, $\lambda \ge 0$, and $\eta,b \in (0, 1)$. There exist $n_0 \equiv
n_0(\beta, \lambda,\eta,b)$ and $C_0 \equiv C_0(\beta, \lambda, \eta, b)$ such that for all $n \ge n_0$, any $U \subseteq  [\eta, 1]
\times \R^3 $, and $V_n = [-1 + e^{-n^b}, 1 - e^{-n^b}]$, we have
\[
\E[\Crit_n(U, V_n)]
\leq T(U,V_n)\]
where
\begin{align*}
T(U,V_n)&=\sqrt{\frac{n}{2\pi \beta^2}}\int_{(-1,1)^n \times \R}
\exp\left(nJ(\bbm,y)-\frac{ny^2}{2\beta^2}
+C_0 \cdot n^{\max(0.9,b)}\right)\\
&\hspace{1.5in}\cdot \ones_{(Q(\bbm), M(\bbm), A(\bbm), E(\bbm))
\in U}\,\d \bbm\, \d y,
\end{align*}
\begin{align*}
J(\bbm,y)&=\frac{\beta^2(1-Q(\bbm))^2}{2}-\frac{1}{2}\log (2\pi \beta^2
Q(\bbm))\\
&\hspace{1in}+\frac{1}{n}\sum_{i=1}^n \log g(m_i;M(\bbm),Q(\bbm),y),
\end{align*}
\begin{equation}\label{eq:gx}
g(x;\vphi,q,y)=\frac{1}{1-x^2} \cdot \exp\left(-\frac{(\arctanh(x)-\beta\lambda
\vphi +\beta^2(1-q)x-yx)^2}{2\beta^2q}\right).
\end{equation}
\end{corollary}

\subsection{Variational upper bound}
We next analyze the integral $T(U,V_n)$
in Corollary \ref{cor:kacrice} using large deviations techniques.
Let $\cP$ be the space of Borel probability measures on $(-1,1)$.
For $\rho \in \cP$, define
\[A(\rho)=\int_{-1}^1 x\arctanh(x)\,\rho(\d x),\;\;
M(\rho)=\int_{-1}^1 x\,\rho(\d x),\;\; Q(\rho)=\int_{-1}^1 x^2\,\rho(\d x),\]
\[E(\rho)= - \int_{-1}^1 \big[ h(x) + \frac{1}{2} x \arctanh(x) \big] \rho(\d x)-\frac{\beta^2}{4}(1-Q(\rho)^2)\, .
\]
Here,
$A(\rho)$ may equal $\infty$ and $E(\rho)$ may equal $-\infty$.
For $\rho \in \cP$ such that $Q(\rho)>0$ and for $y \in \R$,
define the extended real-valued functional
\[J(\rho,y)=\frac{\beta^2(1-Q(\rho))^2}{2}-\frac{1}{2}\log(2\pi \beta^2
Q(\rho))+\int_{-1}^1 \log
g(x;M(\rho),Q(\rho),y)\,\rho(\d x).\]
Note that $\log [1/(1-x^2)]$ grows at a slower rate than $\arctanh(x)^2$ as
$x \to \pm 1$, so $J(\rho,y)<\infty$ always,
and $J(\rho,y)>-\infty$ when $\int_{-1}^1
\arctanh(x)^2\,\rho(\d x)<\infty$.
These functionals extend the previous definitions upon identifying $\bbm \in
(-1,1)^n$ with the empirical distribution of its coordinates.
Further recall the definition of relative entropy between two Borel probability measures $\mu$ and $\pi$
on $\reals$:
\begin{align*}
  H(\mu|\pi) := \begin{dcases*}
   \int \log \left(\dfrac{\de \mu}{\de\pi}(x)\right)\; \mu(\de x) &\;\; \mbox{if $\mu\ll \pi$,}\\
    +\infty &\;\; \mbox{otherwise.}
    \end{dcases*}
\end{align*}
We establish the following upper bound in Appendix
\ref{sec:Proof_Kac_Rice_Abstract}.

\begin{proposition}\label{PROP:ABSTRACT}
Fix $\beta>0$, $\lambda\ge 0$,
$\eta,b \in (0,1)$, and any closed set $U \subseteq [\eta,1]
\times \R^3$. Let $V_n$ and $T(U,V_n)$ be as in Corollary \ref{cor:kacrice},
let $\pi_0$ be the uniform distribution on
$(-1,1)$. Then
\begin{align*}
&\limsup_n \frac{1}{n}\log T(U, V_n)\\
&\quad\leq \sup_{\rho \in \cP:\,
(Q(\rho), M(\rho), A(\rho), E(\rho)) \in U}\;\sup_{y \in \R}\;
\Big\{J(\rho,y)-H(\rho\l \pi_0)-\frac{y^2}{2\beta^2}+\log 2\Big\}\, .
\end{align*}
\end{proposition}

\subsection{Proof of Theorem \ref{thm:Kac_rice_complexity}}

Finally, we conclude the proof of Theorem \ref{thm:Kac_rice_complexity} by 
analyzing the optimization problem in Proposition \ref{PROP:ABSTRACT} over $y$ and $\rho$. 

\begin{proof}
Fix $\rho \in \cP$ such that $\sup_{y \in \R} [J(\rho,y)-H(\rho\l \pi_0)]>-\infty$.
This implies $J(\rho,y)>-\infty$ for some $y$, so $\int_{-1}^1
\arctanh(x)^2\,\rho(\d x)<\infty$. Furthermore $H(\rho\l \pi_0)<\infty$, so in
particular $\rho(\d x)=f(x)\d x$ for some density function $f(x)$ on $(-1,1)$.
Let $(q,\vphi,a,e)=(Q(\rho),M(\rho),A(\rho),E(\rho))$. Then
\[\int \log g(x;\vphi,q,y)\,\rho(\d x)=\chi_1 y^2+\chi_2 y+\chi_3,\]
where
\begin{align*}
\chi_1&=-\int \frac{x^2}{2\beta^2 q}\rho(\d x)=-\frac{1}{2\beta^2},\\
\chi_2&=\int \frac{x(\arctanh(x)-\beta\lambda \vphi+\beta^2(1-q)x)}{\beta^2
q}\rho(\d x)=\frac{a}{\beta^2q}-\frac{\lambda}{\beta}\frac{\vphi^2}{q}+1-q,\\
\chi_3&=\int \left(\log \frac{1}{1-x^2}-\frac{\arctanh(x)^2}{2\beta^2 q}
+\frac{\lambda \vphi\arctanh(x)}{\beta q}\right)\rho(\d x)\\
&\hspace{0.5in}+\frac{(2\beta\lambda (1-q)  - \lambda^2)\vphi^2-2a(1-q)}{2q}-\frac{\beta^2(1-q)^2}{2}.
\end{align*}
Let $\rho_{\vphi,q}$ be the law of $\tanh(z)$ where $z \sim
\cN(\beta\lambda \vphi,\beta^2q)$. Note that
\begin{align*}
&\int_{-1}^1 \frac{1}{1-x^2}\exp\left(-\frac{\arctanh(x)^2}{2\beta^2 q}
+\frac{\lambda \vphi\arctanh(x)}{\beta q}\right)\d x\\
&=\int_\R e^{-y^2/(2\beta^2 q)+ \lambda \vphi y/ (\beta q)}dy
=e^{\lambda^2 \vphi^2/(2q)}(2\pi \beta^2 q)^{1/2},
\end{align*}
and
\[h_{\vphi,q}(x)=e^{- \lambda^2 \vphi^2/(2q)}(2\pi \beta^2 q)^{-1/2}
\frac{1}{1-x^2}\exp\left(-\frac{\arctanh(x)^2}{2\beta^2 q}
+\frac{\lambda \vphi\arctanh(x)}{\beta q}\right)\]
is the density function of the law $\rho_{\vphi,q}$. Then
$\chi_3=\int \log h_{\vphi,q}(x)\rho(\d x)+\frac{1}{2}\log (2\pi \beta^2
q)+\chi_4$ where
\[\chi_4=\frac{(\beta \lambda \vphi^2-a)(1-q)}{q}
-\frac{\beta^2(1-q)^2}{2}.\]
Note further that
\[\sup_{y \in \R} \left[\left(\chi_1-\frac{1}{2\beta^2}\right) y^2+\chi_2 y\right]
=\sup_{y \in \R} \left[-\frac{1}{\beta^2} y^2+\chi_2 y\right] = \frac{\beta^2\chi_2^2}{4},\]
and $H(\rho\l \pi_0)=\log 2+\int_{-1}^1 \log f(x)\,\rho(\d x)$, where $f(x)$ is the
density of $\rho$. We hence obtain
\begin{align*}
&\sup_{y \in \R} \left[J(\rho,y)-H(\rho\l \pi_0)-y^2/(2\beta^2)+\log 2\right]\\
&=\frac{\beta^2(1-q)^2}{2}+\frac{\beta^2\chi_2^2}{4}+\chi_4
-\int_{-1}^1 \log \frac{f(x)}{h_{\vphi,q}(x)}\rho(\d x)\\
&= \frac{1}{4\beta^2}\left(\beta^2(1-q)+\frac{\beta\lambda \vphi^2}{q}-\frac{a}{q}\right)^2
-H(\rho\l \rho_{\vphi,q}),
\end{align*}
where the second line simplifies
$\beta^2(1-q)^2/2+\beta^2\chi_2^2/4+\chi_4$. Then, combining this with
Corollary \ref{cor:kacrice} and Proposition \ref{PROP:ABSTRACT},
\begin{align*}
&\limsup_n n^{-1} \log \E[\Crit_n(U, V_n)]\\
&\leq \sup_{(q, \vphi, a, e) \in U}
\left[\frac{1}{4\beta^2} \left(\beta^2(1-q)+\frac{\beta \lambda \vphi^2}{q}-\frac{a}{ q}\right)^2
-\inf_{\rho \in \cP(q,\vphi,a,e)} H(\rho\l \rho_{\vphi,q})\right]
\end{align*}
where $\cP(q,\vphi,a,e)=\{\rho \in
\cP:(Q(\rho),M(\rho),A(\rho),E(\rho))=(q,\vphi,a,e)\}$.

Recalling $u(q,a)$ from (\ref{eq:uqa}) and
introducing Lagrange multipliers for these constraints, we have
\begin{align*}
\inf_{\rho \in \cP(q,\vphi,a,e)} H(\rho\l \rho_{\vphi,q})
&=\inf_{\rho \in \cP} \sup_{\mu,\nu,\tau,\gamma}
H(\rho\l \rho_{\vphi,q})-\mu(Q(\rho)-q)-\nu(M(\rho)-\vphi)\\
&\hspace{0.3in}-\tau(A(\rho)-a)-\gamma[u(Q(\rho),A(\rho))-E(\rho)-u(q,a)+e]\\
&\ge \sup_{\mu,\nu,\tau,\gamma}
\mu q+\nu \vphi+\tau a+\gamma[u(q,a)-e]+\inf_{\rho \in \cP} 
F(\rho;\mu,\nu,\tau,\gamma),
\end{align*}
where 
\begin{align*}
&F(\rho;\mu,\nu,\tau,\gamma)\\
&=H(\rho\l \rho_{\vphi,q})-\mu Q(\rho)-\nu M(\rho)-\tau A(\rho)
-\gamma[u(Q(\rho),A(\rho))-E(\rho)]\\
&=\int_{-1}^1 \left(\log \tfrac{f(x)}{h_{\vphi,q}(x)}-\mu x^2
-\nu x-\tau x\arctanh(x)
+\tfrac{\gamma}{2}{ \left(\log
\tfrac{1+x}{2}+\log \tfrac{1-x}{2}\right)}\right)\rho(\d x)\\
&=H(\rho|\rho_{\mu,\nu,\tau,\gamma})-\log I.
\end{align*}
Here,
\begin{align*}
\rho_{\mu,\nu,\tau,\gamma}(\d x)&=I^{-1}\exp\Big(\mu x^2+
\nu x+\tau x\arctanh(x)\\
&\hspace{1in}+\frac{\gamma}{2}{ \Big(\log
\frac{1+x}{2}+\log \frac{1-x}{2}\Big)}\Big)h_{\vphi,q}(x)\d x
\end{align*}
is a suitably defined exponential family density with base measure
$h_{\vphi,q}(x)\d x$, and 
$I$ as defined in (\ref{eq:I}) is the normalizing
constant for this density. Thus
\[\inf_{\rho \in \cP} F(\rho;\mu,\nu,\tau,\gamma)=-\log I\]
with the infimum achieved at $\rho=\rho_{\mu,\nu,\tau,\gamma}$,
and Theorem \ref{thm:Kac_rice_complexity} follows.
\end{proof}

\section{Proof of Proposition \ref{PROP:LOGDETBOUND}}\label{sec:Hessian}

In this section, we prove Proposition \ref{PROP:LOGDETBOUND} which analyzes
$\det \bH_n(\bbm)$.
The following lemma is standard, and can be found for example in \cite[Lemma 11]{bayati2011dynamics}.
\begin{lemma}\label{lem:conditional_GOE}
Let $\bW \sim \textup{GOE}(n)$, and $\bx, \by \in \R^n$. Denote by
$\Proj_\bx^\perp=\id - \bx \bx^\sT / \| \bx \|_2^2$ the projection orthogonal
to $\bx$. Then we have
\[
\begin{aligned}
\bW \vert_{\bW \bx = \by} \stackrel{d}{=} \Proj_\bx^\perp \bW \Proj_\bx^\perp +
(\bx \by^\sT + \by \bx^\sT)/ \Vert \bx\Vert_2^2 - \bx \bx^\sT \< \bx, \by\> /
\Vert \bx \Vert_2^4.
\end{aligned}
\]
\end{lemma}

Denoting $\bu \equiv \bu(\bbm)
= \arctanh(\bbm) - \beta \lambda/n \cdot \<\ones, \bbm\> \ones +
\beta^2 [1 - Q(\bbm)] \bbm$, the condition $\bg_n(\bbm)=\bzero$ is equivalent
to $\bW\bbm=\beta^{-1}\bu$. Then  by Eqs.\ (\ref{eq:gn}--\ref{eq:Hn}), conditional on $\bg_n(\bbm)=\bzero$,
the Hessian $\bH_n(\bbm)$ is equal in law to 
\[\bZ(\bbm)= \bD(\bbm)-\beta \bW +\Delta(\bbm)\]
where
\begin{equation}\label{eq:di}
\bD(\bbm)=\diag(d_1,\ldots,d_n),
\qquad d_i=\frac{1}{1-m_i^2}+\beta^2[1-Q(\bbm)],
\end{equation}
and
\begin{equation}\label{eq:Delta}
\begin{aligned}
&\Delta(\bbm)= \beta \bW \bbm \bbm^\sT / \| \bbm \|_2^2 + \beta \bbm \bbm^\sT \bW/ \| \bbm \|_2^2 - \beta \bbm \bbm^\sT \<\bbm, \bW \bbm\> / \| \bbm \|_2^4 \\
&- \frac{\beta\lambda}{n} \cdot \ones\ones^\sT - 2 \beta^2/n \cdot \bbm \bbm^\sT - (\bbm \bu^\sT + \bu \bbm^\sT)/ \Vert \bbm \Vert_2^2 + \bbm \bbm^\sT \< \bbm, \bu\> / \Vert \bbm \Vert_2^4. 
\end{aligned}
\end{equation}

Let $\sigma_\beta$ be the semicircle law with support
$[-2\beta,2\beta]$, let $\mu_{\bD(\bbm)}$ be the empirical spectral
distribution of $\bD(\bbm)$, and define their additive free convolution \cite{voiculescu1991limit}
\[\nu_\bbm=\mu_{\bD(\bbm)} \boxplus \sigma_\beta.\]
We expect $\nu_\bbm$ to approximate the bulk spectral distribution of
$\bZ(\bbm)$ for large $n$ \cite{pastur}.
Denote $\ii=\sqrt{-1}$ and let $\log$ denote
the complex logarithm with branch cut on the negative real axis.
For $\eps>0$, let
\[l_\eps(x)=\log \l x-\ii \eps \l =\Re \log(x-\ii \eps).\]
Denote
\[\Tr l_\eps(\bZ)=\sum_{i=1}^n l_\eps(\lambda_i(\bZ))\]
where $\lambda_1(\bZ),\ldots,\lambda_n(\bZ)$ are the eigenvalues of $\bZ$.

Recall $L(\bbm)$ as defined in (\ref{eq:Lm}).
We prove Proposition \ref{PROP:LOGDETBOUND} via the following three lemmas.
We defer the proofs of the first two to Sections \ref{sec:Lemma_C2} and
\ref{sec:Lemma_C3}.

\begin{lemma}\label{lemma:expectationbound_new}
Fix $\beta>0$, $\lambda \ge 0$, $\eta,b \in (0,1)$, and $a \in (0,1/9)$.
Let $\bbm$ satisfy $\| \bbm \|_\infty \le 1 - e^{-n^b}$ and $Q(\bbm) \in [\eta,
1]$, and set $\eps=n^{-a}$. Then there exist
constants $C \equiv C(\beta, \lambda,\eta,b,a)>0$ and
$n_0 \equiv n_0(\beta, \lambda,\eta,b,a)>0$ such that for all
$n \geq n_0$,
\[\E[\Tr l_\eps(\bZ(\bbm))]
\leq n\int l_\eps(x)\nu_\bbm(\d x)+C(\eps^{-5}+n^b+\log n).\]
\end{lemma}

\begin{lemma}\label{lemma:integratelog}
For any $\beta>0$ and $\bbm \in (-1,1)^n$,
there exists a constant $C \equiv C(\beta)>0$
depending only on $\beta$ such that for any $\eps>0$,
\[\int l_\eps(x)\nu_\bbm(\d x) \leq L(\bbm)+C\eps.\]
\end{lemma}

\begin{lemma}\label{lemma:concentration_new}
For any $\eps>0$, any $t>0$, any $\bbm \in (-1,1)^n$, and all $n \geq 1$,
\[\P[\l \Tr l_\eps(\bZ(\bbm))-\E \Tr l_\eps(\bZ(\bbm))\l \geq nt]
\leq 2\exp(-n^2\eps^2t^2/(8\beta^2)).\]
\end{lemma}
\begin{proof}[Proof of Lemma \ref{lemma:concentration_new}]
By \cite[Lemma 2.3.1]{AGZbook}, if $f:\R \to \R$ is $L$-Lipschitz, then
the map $\bZ \mapsto \Tr f(\bZ)$ is $L\sqrt{2n}$-Lipschitz with respect to the
Frobenius norm of $\bZ$. Then for each fixed $\bbm$, the quantity
$\Tr f(\bZ(\bbm))$ is a $2L\beta$-Lipschitz function of the $n(n-1)/2$
standard Gaussian variables which parametrize $\Proj_\bbm^\perp \bW \Proj_\bbm^\perp$. By Gaussian concentration
of measure, for any $t>0$, we have
\[\P\left[ \l \Tr f(\bZ(\bbm))-\E \Tr f(\bZ(\bbm))\l \geq nt\right]
\leq 2\exp\left(-\frac{n^2t^2}{8L^2 \beta^2}\right).\]
The result follows from observing that $l_\eps$ is differentiable, with
derivative at each $x \in \R$ satisfying
\[\left\l\frac{\d}{\d x}l_\eps(x)\right\l=\left\l
\frac{\d}{\d x}\Re \log(x-\ii \eps)\right\l
=\left\l\Re \frac{1}{x-\ii \eps}\right\l \leq \frac{1}{\eps}.\]
\end{proof}

\begin{proof}[Proof of Proposition \ref{PROP:LOGDETBOUND}]
Taking $\eps=n^{-0.11}$, we obtain
from Lemmas \ref{lemma:expectationbound_new},
\ref{lemma:integratelog}, and \ref{lemma:concentration_new}, for some constants
$C,c,n_0>0$, all $n \geq n_0$, and all $t>0$,
\[\P\Big[\Tr l_\eps(\bZ(\bbm)) \geq nL(\bbm)+Cn^{\max(0.89,b)}+nt\Big]
\leq 2\exp(-cn^{1.78}t^2).\]
Then, setting $c_n=nL(\bbm)+Cn^{\max(0.89,b)}$,
\begin{align*}
\E[\l \det \bZ(\bbm)\l ] &\leq \E[\exp(\Tr l_\eps(\bZ(\bbm)))]\\
&\leq \exp(c_n+n^{0.89})+\int_{\exp(c_n+n^{0.89})}^\infty
\P[\exp(\Tr l_\eps(\bZ(\bbm))) \geq t]\,\d t\\
&= \exp(c_n+n^{0.89})+e^{c_n}\int_{n^{0.89}}^\infty e^s\,
\P[\Tr l_\eps(\bZ(\bbm)) \geq c_n+s]\,\d s\\
&\leq \exp(c_n+n^{0.89})+e^{c_n}\int_{n^{0.89}}^\infty
2\exp(s-cs^2/n^{0.22}) \d s\\
&<\exp(nL(\bbm)+C_0 \cdot n^{\max(0.9,b)}),
\end{align*}
the last line holding for sufficiently large $n_0$ and $C_0$. 
\end{proof}

In the remainder of this section, we prove Lemmas
\ref{lemma:expectationbound_new} and \ref{lemma:integratelog}.

\subsection{Proof of Lemma \ref{lemma:expectationbound_new}}\label{sec:Lemma_C2}
We first establish an analogue of Lemma \ref{lemma:expectationbound_new} for a
matrix of the form $\bD-\beta \bW$.
\begin{lemma}\label{lemma:expectationbound}
Fix $\beta>0$ and $a \in (0,1/9)$, and let $\eps=n^{-a}$.
Then there exist constants $C \equiv C(\beta,a)>0$ and $n_0 \equiv
n_0(\beta,a)>0$ such that for all $n \geq n_0$, the following holds:
Let $\bD \in \R^{n \times n}$ be any real symmetric matrix, let
$\bH=\bD-\beta \bW$ where $\bW \sim \GOE(n)$, 
and let $\nu=\mu_{\bD} \boxplus \sigma_\beta$. Then
\[\E[\Tr l_\eps(\bH)] \leq n\int l_\eps(x)\nu(\d x)
+C(\eps^{-5}+\log (\|\bD\|_\op+1)+\log n).\]
\end{lemma}

By the results of \cite{pastur}, the Stieltjes transform
\begin{equation}\label{eq:stieltjesdef}
g(z)=\int \frac{1}{x-z}\nu(\d x)
\end{equation}
of the above measure $\nu$ satisfies, for each $z \in \C^+$, the fixed-point
equation 
\begin{equation}\label{eq:fixedpoint}
g(z)=\frac{1}{n}\sum_{i=1}^n \frac{1}{d_i-z-\beta^2 g(z)},
\end{equation}
with $(d_i)_{i\le n}$ denoting the eigenvalues of $\bD$. We first provide a quantitative estimate of the approximation of the spectral
law of $\bH$ by $\nu$, in terms of their Stieltjes transforms at points
$z \in \C^+$ where $\Im z \gtrsim n^{-1/9}$.

\begin{lemma}\label{lemma:Egz}
Fix $\beta>0$, let $\bD \in \R^{n \times n}$ be any real symmetric
matrix, and let $\bH$ and $\nu$ be as in Lemma \ref{lemma:expectationbound}.
For $z \in \C^+$, denote
the Stieltjes transforms of $\bH$ and $\nu$ as
\[\hg(z)=\frac{1}{n}\Tr(\bH-z \id)^{-1},\qquad
g(z)=\int \frac{1}{x-z}\nu(\d x).\]
If $n \geq 40\max(\beta/\Im z,1)^9$, then
\[\l\E\hg(z)-g(z)\l<\frac{12}{n \Im z}\max(\beta/ \Im z,1)^5.\]
\end{lemma}
\begin{proof}
We follow the approach of \cite[Theorem 3.1]{pasturreview}, applying integration
by parts and the Poincar\'e inequality.  Fix $z \in \C^+$ and denote
$\eta=\Im z$.
Denote the resolvents of $\bH$ and $\bD$ as
\[G_\bH(z)=(\bH-z\id)^{-1},\qquad G_\bD(z)=(\bD-z\id )^{-1}.\]
Applying $\bA^{-1}-\bB^{-1}=\bA^{-1}(\bB- \bA)\bB^{-1}$, we obtain
\[G_\bH(z)-G_\bD(z)=\beta G_\bH(z)\bW G_\bD(z).\]
As $G_\bD(z)$ is deterministic, this yields
\begin{equation}\label{eq:EGz}
\E\,G_\bH(z)=G_\bD(z)+\beta\,\E[G_\bH(z)\bW]G_\bD(z).
\end{equation}

We use integration by parts on the term $\E[G_\bH(z)\bW]$. For any differentiable
bounded function $f:\R \to \R$, when $\xi \sim \cN(0,\sigma^2)$,
\begin{equation}\label{eq:integrationbyparts}
\E[\xi f(\xi)]=\sigma^2\E[f'(\xi)].
\end{equation}
Consider indices $i,j,k \in [n]$. If $j<k$, then
setting $\bX$ as the matrix with $(j,k)$ and $(k,j)$ entries equal to 1 and
remaining entries 0, we have
\begin{align}
\frac{\partial G_\bH(z)_{ij}}{\partial W_{jk}}
&=\lim_{\delta \to 0} \delta^{-1}\Big((\bH-z\id -\beta \delta
\bX)^{-1}-(\bH-z\id )^{-1}\Big)_{ij}\nonumber\\
&=\beta (G_\bH(z)\bX G_\bH(z))_{ij}=\beta(G_\bH(z)_{ij}G_\bH(z)_{kj}+G_\bH(z)_{ik}
G_\bH(z)_{jj}).\label{eq:Gderiv}
\end{align}
For $z \in \C^+$ with $\Im z=\eta$, we have $\l G_\bH(z)_{ij}\l \leq 1/\eta$. Then
(\ref{eq:integrationbyparts}) yields
\[\E[G_\bH(z)_{ij}W_{jk}]=\frac{1}{n}
\E\left[\frac{\partial G_\bH(z)_{ij}}{\partial W_{jk}}\right]
=\frac{\beta}{n}\E[G_\bH(z)_{ij}G_\bH(z)_{kj}+G_\bH(z)_{ik}G_\bH(z)_{jj}].\]
Similar arguments yield the same identity when $j=k$. Then applying
$G_\bH(z)_{kj}=G_\bH(z)_{jk}$ and summing over $j$, we obtain
\[\E[G_\bH(z)\bW]=\frac{\beta}{n}\,\E[G_\bH(z)^2]+\beta\,\E[\hg(z)G_\bH(z)].\]

Denoting $\delta(z)=\hg(z)-\E \hg(z)$ and substituting the above
into (\ref{eq:EGz}),
\[(\E\,G_\bH(z))
[\id -\beta^2(\E \hg(z))G_\bD(z)]=\left(\id +\frac{\beta^2}{n}\,\E[G_\bH(z)^2]
+\beta^2\,\E[\delta(z)G_\bH(z)]\right)G_\bD(z).\]
Multiplying on the right by $G_\bD(z)^{-1}=\bD-z\id $,
\[(\E\,G_\bH(z))[\bD-z\id -\beta^2(\E \hg(z))\id ]=\id +\frac{\beta^2}{n}\,\E[G_\bH(z)^2]
+\beta^2\,\E[\delta(z)G_\bH(z)].\]
Now multiplying on the right by
$[\bD-z\id -(\beta^2 \E \hg(z))\id ]^{-1}=G_\bD(z+\beta^2 \E \hg(z))$,
taking $n^{-1}\Tr$, and rearranging,
\begin{equation}\label{eq:approxfixedpoint}
\E\hg(z)-\frac{1}{n}\sum_{i=1}^n \frac{1}{d_i-z-\beta^2 \E \hg(z)}=r(z)
\end{equation}
where 
\begin{align*}
r(z)&=\frac{\beta^2}{n^2}\E \Tr\Big(G_\bH(z)^2G_\bD(z+\beta^2\E \hg(z))\Big)\\
&\hspace{0.5in}
+\frac{\beta^2}{n} \E\left[\delta(z)\Tr \Big(G_\bH(z)G_\bD(z+\beta^2 \E
\hg(z))\Big)\right]=:\mathrm{I}+\mathrm{II}.
\end{align*}

Let us bound this remainder $r(z)$. Noting $\|(\bX -z\id )^{-1}\|_{\op} \leq 1/\Im z$ for
any real-symmetric $\bX$, observing $\Im z+ \beta^2 \E \hg(z) \geq \Im z$, and applying
$\l\Tr \bX \l \leq n\| \bX\|_{\op}$, we obtain
\[\l\mathrm{I}\l \leq \frac{\beta^2}{n\eta^3}.\]
For $\mathrm{II}$, applying these bounds and Cauchy-Schwarz,
\[\l\mathrm{II}\l \leq \frac{\beta^2}{n}\E\big[\l\delta(z)\l^2\big]^{1/2}
\frac{n}{\eta^2}=\frac{\beta^2}{\eta^2}\Var(\hg(z))^{1/2},\]
where $\Var(\hg)=\E[\l\hg-\E\hg\l^2]=\E[\l\hg\l^2]-\l\E \hg\l^2$.
We apply the Poincar\'e inequality to bound $\Var(\hg(z))$. For i.i.d.\
$\cN(0,1)$ variables $\xi_1,\ldots,\xi_k$ and $f:\R^k \to \C$ with bounded
partial derivatives,
\begin{equation}\label{eq:poincare}
\Var(f(\xi_1,\ldots,\xi_k)) \leq \E[\|\nabla f(\xi_1,\ldots,\xi_k)\|_{2}^2].
\end{equation}
For $j<k$, a computation similar to (\ref{eq:Gderiv}) yields
\begin{align*}
\frac{1}{\sqrt{n}}\frac{\partial \hg(z)}{\partial W_{jk}}&=
\frac{1}{n^{3/2}}\sum_{i=1}^n \frac{\partial G_\bH(z)_{ii}}{\partial W_{jk}}
=\frac{2\beta}{n^{3/2}}[G_\bH(z)^2]_{jk},\\
\frac{1}{\sqrt{n/2}}\frac{\partial \hg(z)}{\partial W_{jj}}&=
\frac{\sqrt{2}\beta}{n^{3/2}}[G_\bH(z)^2]_{jj}.
\end{align*}
Noting $(\sqrt{n/2}\,W_{jj}:j \in [n]) \cup (\sqrt{n}\,W_{jk}:1 \leq j<k \leq
n)$ are i.i.d.\ $\cN(0,1)$ variables, (\ref{eq:poincare}) yields
\begin{align*}
\Var(\hg(z)) &\leq \E\left[\sum_{j=1}^n \frac{2\beta^2}{n^3}
\l[G_\bH(z)^2]_{jj}\l^2+\sum_{j<k}
\frac{4\beta^2}{n^3}\l[G_\bH(z)^2]_{jk}\l^2\right]\\
&=\frac{2\beta^2}{n^3}\E \Tr\big( G_\bH(z)^2 \overline{(G_\bH(z)^2)}^\sT\big)
\leq \frac{2\beta^2}{n^2\eta^4}.
\end{align*}
Combining the above,
\begin{equation}\label{eq:rzbound}
\l r(z)\l \leq \frac{\beta^2}{n\eta^4}(\eta+\sqrt{2}\beta).
\end{equation}

Finally, we compare (\ref{eq:approxfixedpoint}) with the fixed-point equation
(\ref{eq:fixedpoint}) satisfied by $g(z)$ to obtain a bound on $\E\hg(z)-g(z)$:
Denoting $(x)_+=\max(x,0)$, we have
\begin{align}
\l r(z) \l
&=\left\l\E \hg(z)-g(z)-\frac{1}{n}\sum_{i=1}^n \left(\frac{1}{d_i-z-\beta^2
\E \hg(z)}-\frac{1}{d_i-z-\beta^2 g(z)}\right)\right\l\nonumber\\
&=\big\l\E \hg(z)-g(z)\big\l\left\l1-\frac{1}{n}\sum_{i=1}^n \frac{\beta^2}
{(d_i-z-\beta^2 \E \hg(z))(d_i-z-\beta^2 g(z))}\right\l\nonumber\\
&\geq \big\l\E \hg(z)-g(z)\big\l
\left(1-\sqrt{\frac{1}{n}\sum_{i=1}^n
\frac{\beta^2}{\l d_i-z-\beta^2 \E \hg(z)\l^2}}
\sqrt{\frac{1}{n}\sum_{i=1}^n \frac{\beta^2}{\l d_i-z-\beta^2 g(z) \l^2}}
\right)_+,\label{eq:stability1}
\end{align}
the last step applying $\l1-x\l \geq (1-\l x\l )_+$ and Cauchy-Schwarz.
Taking the imaginary part of (\ref{eq:fixedpoint}) and rearranging,
\[\frac{\beta^2 \Im g(z)}{\eta+\beta^2 \Im g(z)}=\frac{1}{n}\sum_{i=1}^n
\frac{\beta^2}{\l d_i-z-\beta^2 g(z)\l ^2}.\]
We have
\[\Im g(z)=\Im \int \frac{1}{x-z}\nu(\d x)=\int \frac{\eta}{\l x-z\l ^2}\nu(\d x)
\leq 1/\eta,\]
and hence
\[\frac{1}{n}\sum_{i=1}^n \frac{\beta^2}{\l d_i-z-\beta^2 g(z)\l ^2}
\leq \frac{\beta^2}{\eta^2+\beta^2}.\]
Similarly, taking imaginary parts of (\ref{eq:approxfixedpoint}) and
rearranging,
\[\frac{1}{n}\sum_{i=1}^n \frac{\beta^2}{\l d_i-z-\beta^2 \E \hg(z)\l ^2}=
\frac{\beta^2\Im \E \hg(z)-\beta^2\Im r(z)}{\eta+\beta^2 \Im \E \hg(z)}
\leq \frac{\beta^2}{\eta^2+\beta^2}+\frac{\beta^2\l r(z)\l }{\eta}.\]
Applying $\sqrt{a+b} \leq \sqrt{a}+\sqrt{b}$ and the above bounds to
(\ref{eq:stability1}), we obtain
\begin{align*}
\l r(z)\l &\geq \big\l \E\hg(z)-g(z)\big\l 
\left(1-\frac{\beta^2}{\eta^2+\beta^2}-\sqrt{\frac{\beta^2}
{\eta^2+\beta^2}\cdot \frac{\beta^2\l r(z)\l }{\eta}}\right)_+\\
&\geq  \big\l \E\hg(z)-g(z)\big\l  \left(\frac{\eta^2}{\eta^2+\beta^2}
-\sqrt{\frac{\beta^2\l r(z)\l }{\eta}}\right)_+.
\end{align*}

When
\[n \geq 40\max\left(\frac{\beta}{\eta},1\right)^9
\geq 20\left(\frac{\beta^4}{\eta^4}+\frac{\beta^9}{\eta^9}\right)>
\frac{4\beta^4(\eta+\sqrt{2}\beta)(\eta^2+\beta^2)^2}{\eta^9},\]
we verify from (\ref{eq:rzbound}) that
$\sqrt{\beta^2\l r(z)\l/\eta} \leq \eta^2/(2(\eta^2+\beta^2))$.
Then the above implies
\begin{align*}
\l \E \hg(z)-g(z)\l  &\leq \frac{2\l r(z)\l (\eta^2+\beta^2)}{\eta^2}
\leq \frac{2\beta^2(\eta+\sqrt{2}\beta)(\eta^2+\beta^2)}{n\eta^6}\\
&<\frac{6}{n\eta}\left(\frac{\beta^2}{\eta^2}+\frac{\beta^5}{\eta^5}\right)
\leq \frac{12}{n\eta}\max\left(\frac{\beta}{\eta},1\right)^5.
\end{align*}
\end{proof}

\begin{proof}[Proof of Lemma \ref{lemma:expectationbound}]
Denote by $\hat{\nu}$ the empirical
spectral measure of $\bH$, and by $\hg(z)$ its Stieltjes transform. Let
$t_n = n(\|\bD\|_\op+\beta+1)$. Then we have
\begin{align*}
\Tr l_\eps(\bH)&=n\int \Re \log(x-\ii\eps)\,\hat{\nu}(\d x)\\
&=n\int \Re \left(\log(x-\ii t_n)-\int_\eps^{t_n} \frac{-\ii}{x-\ii
t}\d t\right)\hat{\nu}(\d x)\\
&=n\int \log\l x-\ii t_n\l \,\hat{\nu}(\d x)
+n\Re\left(\ii \int_\eps^{t_n} \hg(\ii t)\d t\right).
\end{align*}
For $\eps=n^{-a}$ where $a \in (0,1/9)$,
we may apply Lemma \ref{lemma:Egz} to get
\begin{align*}
\left\l \int_\eps^{t_n} \E[\hg(\ii t)]\d t-\int_\eps^{t_n} g(\ii t)\d t\right\l 
&\leq \frac{12}{n}\left(\int_\eps^\beta \frac{\beta^5}{t^6}\d t
+\int_\beta^{t_n} \frac{1}{t}\d t\right)\\
&\leq C\left(\frac{\eps^{-5}+\log (\|\bD\|_\op+1)+\log n}{n}\right)
\end{align*}
for some $(\beta,a)$-dependent constants $C,n_0>0$ and all $n \geq n_0$.
Reversing the above steps, we may write
\[\Re\left(\ii \int_\eps^{t_n} g(\ii t)\d t\right)
=\int\Big(-\log\l x-\ii t_n\l +\log\l x-\ii \eps\l \Big)\nu(\d x),\]
so combining the above yields
\begin{equation}\label{eq:Elogxbound}
\E \Tr l_\eps(\bH) \leq n\int l_\eps(x)\nu(\d x)+r_1+r_2,
\end{equation}
where
\begin{align*}
r_1&=\E\left[n\int \log\l x-\ii t_n\l \,\hat{\nu}(\d x)\right]
-n\int \log\l x-\ii t_n\l \,\nu(\d x),\\
r_2&=C(\eps^{-5}+\log (\|\bD\|_\op+1)+\log n).
\end{align*}

To bound $r_1$, note that since $\nu = \mu_{\bD}  \boxplus \sigma_{\beta}$, we
have $\sup\{ \l x\l: x \in \supp(\nu)\} \le \|\bD\|_\op+2\beta$. Applying
\[\log \l x-\ii t_n\l -\log t_n=\Re \log(x-\ii t_n)-\Re \log(-\ii t_n)
=\Re \frac{x}{\tilde{x}-\ii t_n}\]
for some $\tilde{x} \in \R$ with $\l \tilde{x}\l  \leq \l x\l $, we obtain
\begin{align}
\left\l \int \log \l x-\ii t_n\l \,\nu(\d x)-\log t_n\right\l  \leq
\frac{\|\bD\|_\op+2\beta}{t_n} \leq \frac{2}{n}.\label{eq:logxt}
\end{align}
A similar argument holds for $\hat{\nu}$. Indeed, for all $t>0$, we have
\begin{equation}\label{eq:normconcentration}
\P[\|\bH\|_{\op}>\|\bD\|_{\op}+2\beta+t] \leq \P[\beta \|\bW\|_{\op}>2\beta+t] \leq
2\exp\left(-\frac{nt^2}{4\beta^2}\right),
\end{equation}
see e.g.\ \cite[Theorem II.11]{davidsonszarek}. 
Let $\mathcal{E}$ be the event where $\|\bH\|_{\op} \le \|\bD\|_{\op}+2\beta+ 1$. Then we have
\begin{align*}S_1 &:= \left\l \E\Big[ \int \log \l x-\ii t_n\l \,\hat{\nu}(\d x) \Big\vert
\mathcal{E} \Big] -\log t_n\right\l  \leq
\frac{\|\bD\|_\op+2\beta+1}{t_n} \leq \frac{2}{n},\\
S_2 &:= \E\Big[ \int \log \l x-\ii t_n\l \,\hat{\nu}(\d x) \Big \vert
\mathcal{E}^c \Big] \le \E [\log(\| \bH \|_{\op}) \vert \mathcal{E}^c ] +
\log(t_n).
\end{align*}
Note that, by (\ref{eq:normconcentration}),
\begin{align*}
&\E[\log (\|\bH\|_\op)  \mid \mathcal{E}^c]\P[\mathcal{E}^c]\\
&= \E[\log (\|\bH\|_\op) \ones\{\mathcal{E}^c\}]\\
&=\int_0^\infty \P\Big[\|\bH\|_\op \geq \max(e^t,\|\bD\|_\op+2\beta+1)\Big]
\d t\\
& \leq \log(\|\bD\|_\op+2\beta+1) \cdot 2e^{-n/(4\beta^2)}
+\int_{\log(\|\bD\|_\op+2\beta+1)}^\infty
\P[\|\bH\|_\op>e^t]\,\d t\\
&<[\log(\|\bD\|_\op+1)+\log n] \cdot e^{-cn}
\end{align*}
for a constant $c \equiv c(\beta)>0$ and $n \geq n_0(\beta)>0$. Then
\begin{align*}
\left\l  \E \int \log \l x-\ii t_n\l \,\hat{\nu}(\d x) - \log t_n \right\l &\le
 S_1 \cdot \P[\mathcal{E}] + [S_2 + \log(t_n)] \cdot \P[\mathcal{E}^c]\\
&\leq \frac{2}{n}+[\log(\|\bD\|_\op+1)+\log n] \cdot e^{-cn}.
\end{align*}
Combining this bound with Eq.~(\ref{eq:logxt}), we obtain $|r_1| \leq \log(\|\bD\|_\op+1)+\log n$
for sufficiently large $n_0$ and all $n\ge n_0$. Then the lemma follows from (\ref{eq:Elogxbound}).
\end{proof}


\begin{proof}[Proof of Lemma \ref{lemma:expectationbound_new}]

Fix $\bbm$ and
denote $\bH=\bD-\beta \bW$, so that
$\bZ=\bH+\Delta$. Lemma \ref{lemma:expectationbound} gives
\begin{equation}\label{eq:TrlepsH}
\E[\Tr l_\eps(\bH)] \leq n\int l_\eps(x)\nu(\d
x)+C(\eps^{-5}+\log(\|\bD\|_\op+1)+\log n).
\end{equation}
From (\ref{eq:Delta}), the operator norm of $\Delta$ can be bounded by
\[
\begin{aligned}
\| \Delta \|_{\op} &= \| \beta \bW \bbm \bbm^\sT / \| \bbm \|_2^2 \|_{\op} + \|
\beta \bbm \bbm^\sT \bW/ \| \bbm \|_2^2  \|_{\op}\\
&\hspace{0.2in}+ \| \beta \bbm \bbm^\sT \<\bbm, \bW \bbm\> / \| \bbm \|_2^4 \|_{\op}
+\| \beta \lambda/n \cdot \ones \ones^\sT  \|_{\op} + \| 2 \beta^2/n \cdot \bbm
\bbm^\sT \|_{\op}\\
&\hspace{0.2in}+ \| (\bbm \bu^\sT + \bu \bbm^\sT)/ \Vert \bbm \Vert_2^2 \|_{F} + \| \bbm \bbm^\sT \< \bbm, \bu\> / \Vert \bbm \Vert_2^4 \|_{F}\\
&\le 3 \beta \| \bW \|_{\op} + 2 \beta^2 + \beta\lambda + 3 \| \bu \|_2 / \| \bbm \|_2. 
\end{aligned}
\]
For $Q(\bbm) \in [\eta, 1]$ and $\bbm \in [-1 + e^{- n^b}, 1 - e^{-n^b}]$, we have 
\[
\begin{aligned}
\| \bu \|_2 \le& \| \arctanh(\bbm)\|_2 + \| \beta \lambda /n \cdot \<\ones, \bbm\> \ones \|_2 + \| \beta^2 (1 - Q(\bbm)) \bbm \|_2 \\
\le& \sqrt{n} \| \arctanh(\bbm)\|_\infty + (\beta^2 + \beta \lambda) \sqrt n \le \sqrt n (n^b + (\beta^2 + \beta \lambda)),
\end{aligned}
\]
and $\| \bbm \|_2 \ge \sqrt{n \eta}$. Accordingly, we have
\begin{equation}\label{eq:Deltabound}
\begin{aligned}
\| \Delta \|_{\op} \le 3 \lambda \| \bW \|_{\op} + (2 \beta^2 + \beta \lambda) + 3(n^b + (\beta^2 + \beta \lambda))/\sqrt \eta. 
\end{aligned}
\end{equation}

Note that $\Delta$ is given in Eq. (\ref{eq:Delta}), which has rank at most 8. Suppose it has $r_+$
positive eigenvalues and $r_-$ negative eigenvalues, where $r_+ + r_- \le 8$ (it is possible that $r_+$ or $r_-$ is zero).
Denote the eigenvalues of $\bH$ as $\lambda_1(\bH) \ge
\lambda_{2}(\bH) \ge \cdots \ge \lambda_k(\bH) \ge 0 > \lambda_{k+1}(\bH) \ge
\cdots \ge \lambda_{n}(\bH)$, and those of $\bZ$ as
$\lambda_1(\bZ) \ge \lambda_2(\bZ) \ge \cdots \ge \lambda_n(\bZ)$. We apply the
following bounds (we use the convention that, if $k = 0$, the set $\{1, 2, \ldots, k \} = \emptyset$):
\begin{itemize}
\item For $i \in S \equiv \{1, \ldots, r_+\} \cup \{n + 1 - r_-, \ldots, n\}
\cup \{k - r_- + 1, \ldots, k + r_+\}$, we have
\[
\l \lambda_i(\bZ) \l \le \| \bH \|_{\op} + \| \Delta \|_{\op}. 
\]
\item 
The rest of the eigenvalues of $\bZ$ satisfy, by Weyl's eigenvalue interlacing,
\[
\begin{aligned}
\lambda_{i - r_+}(\bH) &\ge \lambda_i(\bZ) \ge 0,~~~~ i \in \{r_+ + 1, r_+ + 2, \ldots, k - r_-\}, \\
\lambda_{i + r_-}(\bH) &\le \lambda_i(\bZ) \le 0,~~~~ i \in \{k + r_+ + 1, k + r_+ + 2, \ldots, n - r_-\}. 
\end{aligned}
\]
\end{itemize}
Then
\begin{align*}
\Tr l_\eps(\bZ) &\le \Tr l_\eps(\bH) -
\sum_{i \in S} l_\eps(\lambda_i(\bH)) + 16\log( \| \bH \|_{\op} + \| \Delta \|_{\op})\\
&\le \Tr l_\eps(\bH) - 16\log \eps+ 16 \log( \| \bD \|_{\op} + \lambda
\| \bW \|_{\op}+\| \Delta \|_{\op}). 
\end{align*}
Note that for $\bbm$ with $\| \bbm \|_\infty \le 1 - e^{-n^b}$, we have
$\|\bD\|_\op \leq e^{2 n^b} + \beta^2$. Then, taking expectations of the
above and applying Jensen's inequality,
the result follows from Eq. (\ref{eq:TrlepsH}) and (\ref{eq:Deltabound}), and $\E[\| \bW \|_{\op}] \le 3$.
\end{proof}

\subsection{Proof of Lemma \ref{lemma:integratelog}}\label{sec:Lemma_C3}
Fix $\lambda>0$ and $\bbm \in (-1,1)^n$, and write as shorthand $\nu \equiv
\nu_\bbm$.
Lemma \ref{lemma:integratelog} follows from several properties of $\nu$
and its support:
\begin{lemma}\label{lem:density}
The Stieltjes transform $g(z)$ of $\nu$ extends continuously from $\C^+$
to $\C^+ \cup \R$. Denoting this extension also by $g(z)$,
the measure $\nu$ admits a continuous density given by
\[f(z)=\frac{1}{\pi} \Im g(z),\qquad z \in \R.\]
At each $z \in \R$, letting $z_0 \in \R$ denote the closest point where
$f(z_0)=0$, we have
\begin{equation}\label{eq:edgedecay}
f(z)\leq \left(\frac{3}{4\pi^3\beta}\l z-z_0\l \right)^{1/3}.
\end{equation}
\end{lemma}
\begin{proof}
See \cite[Corollaries 1, 2, 5]{biane1997free}.
\end{proof}

\begin{lemma}\label{lem:support}
The support of $\nu$ is contained in $[0,d_{\max}+2\beta]$, where $d_{\max} = \max_{i\le n}d_i$.
\end{lemma}
\begin{proof}
We note that (\ref{eq:fixedpoint}) holds also for $z \in \R$ by continuity,
where $g(z)$ for $z \in \R$ is the continuous extension defined in Lemma
\ref{lem:density}.
Fix any $z \leq 0$, and consider the function
\[F(g)=\frac{1}{n}\sum_{i=1}^n \frac{1}{d_i-z-\beta^2g}\]
defining the right side of (\ref{eq:fixedpoint}). Let
$p_1<\ldots<p_k$ be the distinct values among
\[\{(d_i-z)/\beta^2:i=1,\ldots,n\}.\]
Then $F(g)$ is a rational function with poles $p_1,\ldots,p_k$.

The equation $g=F(g)$ is rearranged as a polynomial equation of degree $k+1$,
and hence it as most $k+1$ complex roots counting multiplicity. From the graph
of $F(g)$, there is at least one real root between $p_j$ and $p_{j+1}$ for each
$j=1,\ldots,k-1$. Furthermore, when $z \leq 0$, we have $p_1>1-Q(\bbm)$ and
\begin{equation}\label{eq:1minusQ}
F(1-Q(\bbm))=\frac{1}{n}\sum_{i=1}^n \frac{1}{(1-m_i^2)^{-1}-z} \leq
\frac{1}{n}\sum_{i=1}^n (1-m_i^2)=1-Q(\bbm).
\end{equation}
Then from the graph of $F(g)$ on $g \in (0,p_1)$, the equation $g=F(g)$
has at least two real roots in this interval, counting multiplicity. So
all $k+1$ roots of $g=F(g)$ are real, there is exactly one root in each interval
$(p_j,p_{j+1})$, and exactly two roots in $(0,p_1)$. In particular,
$g(z)$ is real. Then by Lemma \ref{lem:density},
the density of $\nu$ is 0 for all $z \leq 0$, so $\nu$ is supported on
$[0,\infty)$.  The upper bound $d_{\max}+2\beta$ follows from $\sup\{|x|:x
\in \supp(\nu)\} \leq \|\bD\|_{\op}+2\beta$, as $\nu=\mu_{\bD} \boxplus
\sigma_\beta$.
\end{proof}

\begin{lemma}\label{lemma:logidentity}
At $z=0$, the Stieltjes transform $g(0)$ is the
smallest real root of (\ref{eq:fixedpoint}) and satisfies $0<g(0) \leq 1-Q(\bbm)$.
Furthermore, for $L(\bbm)$ defined in (\ref{eq:Lm}), we have
\[\int \log(x)\nu(\d x) \leq L(\bbm).\]
\end{lemma}
\begin{proof}
From the proof of Lemma \ref{lem:support}, the equation $g=F(g)$ has
exactly one root $g \in (p_j,p_{j+1})$ for each $j=1,\ldots,k-1$ and exactly
two roots $g \in (0,p_1)$ counting multiplicity.
From the graph of $F$, this implies $F'(g)>1$ for each root
$g \in (p_j,p_{j+1})$, and also the two roots $g_1 \leq g_2$ in $(0,p_1)$
satisfy $F'(g_1) \leq 1$ and $F'(g_2) \geq 1$, with $F'(g_2)=1$ if and only if
$g_1=g_2$. At any $z<0$, since $z \notin \supp(\nu)$, $g(z)$ is analytic in
a neighborhood of $z$ and is given by (\ref{eq:stieltjesdef}).
Differentiating (\ref{eq:stieltjesdef}) in $z$, we verify $g'(z)>0$. On the
other hand, differentiating (\ref{eq:fixedpoint}) in $z$, we obtain
\[g'(z)=\frac{1}{n}\sum_{i=1}^n \frac{1+\beta^2g'(z)}{(d_i-z-\beta^2
g(z))^2},\]
so the condition $g'(z)>0$ implies
\[1-\frac{1}{n}\sum_{i=1}^n \frac{\beta^2}{(d_i-z-\beta^2 g(z))^2}>0.\]
Thus $F'(g(z))<1$. This holds for all $z<0$, so we must have $F'(g(0))
\leq 1$ at $z=0$ by continuity.
Then $g(0)$ is the smallest real root to $g=F(g)$ at $z=0$. Since $1-Q(\bbm)$
is one such root by (\ref{eq:1minusQ}), we have $0<g(0) \leq 1-Q(\bbm)$.

To prove the bound on $\int \log(x)\nu(\d x)$, we write for any $T>0$
\begin{align*}
\int \log(x)\nu(\d x)&=\int \left(\log(x-\ii T)-\int_0^T \frac{-\ii}{x-\ii t}\d t
\right)\nu(\d x)\\
&=\int \log(x-\ii T)\nu(\d x)+\ii \int_0^T g(\ii t)\d t.
\end{align*}
For each $t \in (0,T)$, applying (\ref{eq:fixedpoint}), we have
\begin{align*}
g(\ii t)&=g(\ii t)(1+\beta^2 g'(\ii t))-\beta^2 g(\ii t)g'(\ii t)\\
&=\frac{1}{n}\sum_{i=1}^n \frac{1+\beta^2g'(\ii t)}{d_i-\ii t-\beta^2
g(\ii t)}-\beta^2 g(\ii t)g'(\ii t).
\end{align*}
Then $-\ii g(\ii t)=B'(t)$ for
\[B(t):=\frac{1}{n}\sum_{i=1}^n
\log(d_i-\ii t-\beta^2 g(\ii t))+\frac{\beta^2}{2}g(\ii t)^2,\]
so $\ii \int_0^T g(\ii t)\d t=B(0)-B(T)$. We next take $T \to \infty$, using
$\l g(\ii T)\l \le \sup_x(1/|x-iT|)\le 1/T$ from Eq.~(\ref{eq:stieltjesdef}).
Applying a Taylor expansion of
$\log(x-\ii T)$, we have
\[\int \log(x-\ii T)\nu(\d x)=\log(-\ii T)+O(1/T),\qquad B(T)=\log(-\ii T)
+O(1/T).\]
(In the first equation, we use the fact that $\nu$ has bounded support, cf. Lemma \ref{lem:support}, and in the second we use the definition of $B$,
together with $|g(\ii T)| \le 1/T$.)
Combining the above, we obtain $\int \log(x)\nu(\d x)=B(0)$.

Finally, note that for $F$ as defined in the proof of Lemma \ref{lem:support},
the equation $g=F(g)$ is equivalent to $0=R'(g)$ where
\[R(g)=\frac{\beta^2g^2}{2}+\frac{1}{n}\sum_{i=1}^n
\log (d_i-\beta^2g).\]
Recall that either $g(0)=1-Q(\bbm)$,
or $g(0)<1-Q(\bbm)$ and these are the two roots of this equation $0=R'(g)$
in $(-\infty,p_1)$. 
Since $R(g) \to \infty$ when $g \to -\infty$, we obtain in both cases $
B(0)= R(g(0))
\leq R(1-Q(\bbm))=L(\bbm)$.
\end{proof}

\begin{proof}[Proof of Lemma \ref{lemma:integratelog}]
For any $x>0$,
\begin{align*}
l_\eps(x)-\log x &=\Re \log(x-\ii \eps)-\Re \log(x)\\
&= \log \sqrt{x^2+\eps^2}-\log x = \frac{ \sqrt{x^2+\eps^2}-x}{x+\tilde{\eps}}
\end{align*}
for some $\tilde{\eps} \in [0,\sqrt{x^2+\eps^2}-x]$. Then
$\l l_\eps(x)-\log x\l  \leq \eps/x$, so
\[\int l_\eps(x)\nu(\d x) \leq \int \log(x)\nu(\d x)+\eps\int (1/x)\nu(\d x).\]
The result follows from Lemma \ref{lemma:logidentity},
Lemma \ref{lem:support}, and the density decay
condition (\ref{eq:edgedecay}) at the smallest edge of $\nu$  in the case
where $0 \in \supp(\nu)$.
\end{proof}

\section{Proof of Theorem \ref{thm:TAPBayes}}\label{sec:proof_thm_TAPBayes}

We now prove Theorem \ref{thm:TAPBayes} on the Bayes risk achieved by minimizing
the TAP free energy in the $\Z_2$ synchronization model, corresponding to
$\beta=\lambda$.

Throughout the proof, we write $\cF$ as shorthand for $\cF_{\lambda, \lambda}$. Let $q_\star$ be the largest real solution of the equation
\begin{align}\label{eq:qstar0}
q_\star = \E_G[\tanh(\lambda^2 q_\star + \sqrt{\lambda^2 q_\star} G)^2],
\end{align}
for $G \sim \cN(0, 1)$. Note that when $\lambda > 1$, we have $q_\star>0$ \cite[Lemma 4.2]{deshpande2016asymptotic}. Denote
\begin{equation}\label{eq:estar}
\vphi_\star=q_\star,\;\;  a_\star=\lambda^2 q_\star,\;\;
e_\star = -\frac{\lambda^2}{4}(1 - 2 q_\star -
q_\star^2) - \E_G\{\log [2\cosh(\lambda^2 q_\star + \sqrt{\lambda^2 q_\star }
G)]\},
\end{equation}
and recall $Q(\bbm) = \| \bbm \|_2^2/n$ and $\magn(\bbm)=\<\bx, \bbm\>/n$.
Recall $\cC_{\lambda, n}$ the set of critical points of $\cF$ with energy below
$-\lambda^2/3$,  and define
\[\cC_{\lambda, n}^+ \equiv \{\bbm \in [-1, 1]^n: \nabla \cF(\bbm) = \bzero, ~
\cF(\bbm) \leq -\lambda^2/3,~ M(\bbm) \geq 0\}\]
so that $\cC_{\lambda,n}=\cC_{\lambda,n}^+ \cup -\cC_{\lambda,n}^+$.

\subsection{Localization of critical points}

Denote $\Gamma_\eps \subseteq [-1, 1]^n$ to be the pre-image of $(Q, M, \cF)$ at a neighborhood of $(q_\star, \vphi_\star, e_\star)$:
\begin{align}\label{eqn:pre-image_set}
\Gamma_\eps \equiv& \{\bbm \in [-1, 1]^n : \l Q(\bbm) - q_\star \l \le \eps,~ \l M(\bbm) - \vphi_\star \l \le \eps,~ \l \cF(\bbm) -  e_\star \l < \eps\}.
\end{align}

We first prove that $\emptyset \neq { \cC_{\lambda, n}^+} \subseteq \Gamma_\eps$ with high probability:  
\begin{proposition}\label{prop:mainprop}
There exists $\lambda_0 < \infty$ such that for any $\lambda \ge \lambda_0$ and
$\eps > 0$, for some constants $n_0 \equiv n_0(\lambda, \eps)$ and $\const_0
\equiv \const_0(\lambda,
\eps)$, with probability at least $1 - e^{-\const_0 n}$
for all $n \ge n_0$, we have: ${ \cC_{\lambda, n}^+} \subseteq \Gamma_\eps$, and some global minimizer $\bbm_\star$ of $\cF$ satisfies $\bbm_\star \in { \cC_{\lambda, n}^+} \subseteq \Gamma_\eps$. 
\end{proposition}


Recall that in Theorem \ref{thm:Kac_rice_complexity}, for $\lambda > 0$,
choosing $\beta = \lambda$, taking $\eta,b \in (0, 1)$, and taking closed sets $U \subseteq  [\eta, 1] \times \R^3$, $V_n = [-1 + e^{-n^b}, 1 - e^{-n^b}]$, we have
\[
\limsup_{n \to \infty} n^{-1} \log \E[\Crit_n(U, V_n)] \le \sup_{(q, \vphi, a,
e) \in U} S_\star(q, \vphi, a, e).
\]
%
To prove Proposition \ref{prop:mainprop}, the following first
bounds the function $S_\star(q,\vphi,a,e)$ for $(q,\vphi)$ in a small neighborhood of
$(q_\star,\vphi_\star)$. We defer its proof to Section
\ref{sec:Proof_Prop_Negative_Complexity}.

\begin{proposition}\label{PROP:COMPNEG}
For any $\lambda>0$, we have
\[
S_\star(q_\star, \vphi_\star, a_\star, e_\star) = 0. 
\]
Furthermore,
there exist constants $\const_0,\lambda_0>0$ and a function $c(\lambda)>0$ such
that for any $\lambda \ge \lambda_0$ and any $(q,\vphi)$ satisfying
$\l q - q_\star \l, \l \vphi - \vphi_\star \l \le \const_0/ \lambda^2$, we have 
\begin{equation}\label{eqn:quantitative_complexity_bound}
S_\star(q, \vphi, a, e) \le -c(\lambda) [(a - a_\star)^2 + (q - q_\star)^2 +
(\vphi - \vphi_\star)^2 + (e - e_\star)^2].
\end{equation}
\end{proposition}

Note that for $c_0,\lambda_0$ as above and $\lambda \geq \lambda_0$,
this implies 
\[\sup \{ S_\star(q, \vphi, a, e):  \l q - q_\star \l, \l \vphi - \vphi_\star \l
\le \const_0/ \lambda^2, (a, e) \in \R^2 \} \le 0,\]
with the supremum uniquely attained at $(q_\star, \vphi_\star, a_\star,
e_\star)$.
The following then localizes any critical point of $\cF$ satisfying Theorem
\ref{prop:mainprop} to the above neighborhood of $(q_\star,\vphi_\star)$.
Its proof is in Appendix~\ref{sec:Proof_Localization}.

\begin{proposition}\label{PROP:LOCALIZATION}
Fix any positive integer $k$. Then there exist $\lambda_0,\Const_0>0$ and
functions $\const_0(\lambda),n_0(\lambda)>0$ such that
for all $\lambda \geq \lambda_0$ and $n \geq n_0(\lambda)$,
with probability at least $1 - e^{-\const_0(\lambda) n}$,
any $\bbm \in { \cC_{\lambda, n}^+}$ also satisfies $Q(\bbm), M(\bbm) \ge 1 - \Const_0 / \lambda^k$.  
\end{proposition}

Finally, the following shows that with high probability, there are no
critical points with any coordinate outside the set $V_n$ in Theorem \ref{thm:Kac_rice_complexity}.
\begin{lemma}\label{lem:infinity_norm_high_probability_bound}
There exist constants $n_0$, and $\const_0 > 0$ so that for $n \ge n_0$, with probability at least $1 -
e^{-\const_0 n}$, all points $\bbm \in (-1, 1)^n$ satisfying $\bg_n(\bbm) =
\bzero$ also satisfy
\begin{align}
\| \bbm \|_\infty \le 1 - e^{ - 4 \lambda^2 - 6 \lambda \sqrt n}.
\end{align}
\end{lemma}

\begin{proof}

Note that when $\bg_n(\bbm) = \bzero$, we have
\begin{align*}
\| \arctanh(\bbm) \|_\infty &= \Big\| - \lambda^2/n \cdot \<\ones, \bbm\> \ones -
\lambda \cdot \bW \bbm + \lambda^2 [1 - Q(\bbm)] \bbm \Big\|_\infty\\
&\le 2 \lambda^2 + \lambda \| \bW \|_{\op} \sqrt n.
\end{align*}
With probability at least $1 - e^{-\const_0 n}$ for $n \ge n_0$ and some $n_0 >
0$ and $\const_0 > 0$, we have $\| \bW \|_{\op} \le 3$, and hence $\| \arctanh
(\bbm) \|_\infty \le 2 \lambda^2 + 3 \lambda \sqrt n$. Applying $\tanh(x) \leq 1-e^{-2x}$ for $x \geq 0$, with probability at least $1- e^{-\const_0 n}$, we have
\begin{align}
\| \bbm \|_\infty \le \tanh(\|\arctanh(\bbm) \|_\infty) \le 1 - e^{- 4 \lambda^2
- 6 \lambda \sqrt n}. 
\end{align}
\end{proof}

\begin{proof}[Proof of Proposition \ref{prop:mainprop}]

According to Lemma \ref{lem:infinity_norm_high_probability_bound} and
Proposition \ref{PROP:LOCALIZATION}, there exist constants
$\lambda_0,\Const_0$ and functions $\const_0(\lambda),n_0(\lambda)$ such that
for $\lambda \ge \lambda_0$ and $n \ge n_0(\lambda)$, the following happens with
probability at least $1 - e^{-\const_0 n }$: For any $\bbm \in { \cC_{\lambda, n}^+}$, we have $\| \bbm \|_\infty \le 1 - e^{- 4 \lambda^2
- 6 \lambda \sqrt n} \le 1 - e^{-\lambda^2 \sqrt n}$ and $Q(\bbm), M(\bbm) \ge 1 - \Const_0/\lambda^3$. 

Define
\begin{align*}
U_\eps &= \{(q, \vphi, a, e): (q, \vphi) \in [1 - \Const_0/\lambda^3, 1]^2, (a, e)
\in \R^2 \} \\
&\hspace{1in}\cap \{ (q, \vphi, a, e): \max\{ \l q - q_\star\l, \l \vphi -
\vphi_\star\l, \l e - e_\star \l \} \ge \eps  \},
\end{align*}
and $V_n = [-1 + e^{-\lambda^2 \sqrt n}, 1 - e^{-\lambda^2 \sqrt n}]$. According to Theorem \ref{thm:Kac_rice_complexity}, we have
\[\limsup_{n \to \infty} n^{-1} \log \E[\Crit_n(U_\eps, V_n)] \le \sup_{(q,
\vphi, a, e) \in U_\eps} S_\star(q, \vphi, a, e). \]
According to Proposition \ref{PROP:COMPNEG}, for $\lambda \ge \lambda_0$ with $\lambda_0$ large enough, we have 
\[0 > \sup_{(q, \vphi, a, e) \in U_\eps} S_\star(q, \vphi, a, e) \equiv -
s_0(\lambda, \eps).\]
Therefore, there exists $n_0(\lambda,\eps)$ large enough so that for $n \ge
n_0(\lambda,\eps)$, 
\[n^{-1} \log \E[\Crit_n(U_\eps, V_n)] \le - s_0(\lambda, \eps) /2.\]
Accordingly, by Markov's inequality, we have
\[\P(\Crit_n(U_\eps, V_n) \ge 1) \le \E[\Crit_n(U_\eps, V_n)] \le \exp\{ - s_0
n/2 \}.\]
Combining the above statements concludes that ${ \cC_{\lambda, n}^+} \subseteq \Gamma_\eps$. 

Next we show that a global minimizer $\bbm_\star \in { \cC_{\lambda, n}^+}$.
The gradient of $\cF$ diverges at the boundary of $[-1, 1]^n$ and points outside
$[-1, 1]^n$ because of the entropy term. Hence, any global minimizer
$\bbm_\star$ belongs to $(-1, 1)^n$ and satisfies $\nabla \cF(\bbm_\star) =
\bzero$. Furthermore, note $\cF(\bx) = -\lambda^2/2 - \lambda \<\bx, \bW \bx \>
/(2n) \le -\lambda^2/2 + \lambda \| \bW \|_{\op}/2$, and $\P(\|\bW \|_\op \le 3)
\ge 1 - e^{- c_0 n}$ for $n \ge n_0$ with some constant $c_0$ and $n_0$.
Accordingly, we have $\cF(\bbm_\star) \le \cF(\bx) \le -\lambda^2 /2 + 3 \lambda
/2 \le - \lambda^2/3$ for $\lambda \ge 9$ and $n \ge n_0$ with probability at
least $1 - e^{-c_0 n}$. Taking $-\bbm_\star$ in place of $\bbm_\star$ if
necessary, we ensure $M(\bbm_\star) \geq 0$.
\end{proof}

\subsection{Proof of Theorem \ref{thm:TAPBayes}}

In this regime, \cite{deshpande2016asymptotic} showed that $\MMSE_n(\lambda)$ converges
to a limiting value which admits the following characterization.
\begin{proposition}[\cite{deshpande2016asymptotic}]\label{prop:asymptoticMMSE}
Let $\lambda>0$ be a fixed constant. Let $q_\star$ be the largest nonnegative
solution to the equation
\begin{equation}\label{eq:qstar}
q_\star=\E_{G \sim \cN(0,1)}
\Big[\tanh(\lambda^2 q_\star+\sqrt{\lambda^2q_\star}G)^2\Big].
\end{equation}
Then $\lim_{n \to \infty} \MMSE_n(\lambda)=1-q_\star^2$.
The value $q_\star$ is monotonically increasing in $\lambda$, with
$q_\star=0$ when $\lambda \leq 1$ and $q_\star \in (0,1)$ when $\lambda>1$.
\end{proposition}
\begin{proof}[Proof of Proposition \ref{prop:asymptoticMMSE}]
See \cite[Eqs.\ (167) and (143) and Lemma 4.2]{deshpande2016asymptotic}.
(The notational identification with \cite{deshpande2016asymptotic} is
$\gamma_*/\lambda \leftrightarrow q_\star$ and $\lambda \leftrightarrow \lambda^2$.
\cite{deshpande2016asymptotic} establishes this result with additional direct observations under a
binary erasure
channel with erasure probability $1-\eps$, and the
statement for $\eps=0$ follows from continuity in $\eps$.)
\end{proof}



\begin{proof}[Proof of Theorem \ref{thm:TAPBayes}]
Note that $\cC_{\lambda, n}$ is compact. Hence, when $\cC_{\lambda, n} \neq \emptyset$, the maximum value of
$\|\bbm\bbm^\sT-\widehat{\bX}_{\mathrm{Bayes}}\|_F^2$ over $\cC_{\lambda, n}$ can be attained by some
$\bbm \in [-1,1]^n$. We denote $\bbm_\star \in \R^n$ to be
a random element within the set $\argmax_{\bbm \in \cC(\lambda,
n)}\|\bbm\bbm^\sT-\widehat{\bX}_{\mathrm{Bayes}}\|_F^2$ if $\cC_{\lambda, n}
\neq \emptyset$, or $\bbm_\star=\bzero$ if $\cC_{\lambda, n} = \emptyset$.
Since $\cC_{\lambda,n}=\cC_{\lambda,n}(\bY)$ is uniquely defined in terms of $\cF_{\lambda,\beta}$, which itself depends on
a measurable way on $\bY$ (see Eq.~\eqref{eqn:TAP_free_energy}), it follows that
$\bbm_{\star}$ can be taken to be a measurable function of $\bY$. (Measurability follows from  \cite[Theorem 1]{brown1973measurable} using the
fact that $\{(\bm,\bY): \bm\in \cC_{\lambda,n}(\bY)\}$ is Borel.)
In particular, $\bbm_\star$ is independent of $\bx$ conditional on $\bY$. 

By Proposition \ref{prop:mainprop}, for any constants
$\eps,\delta>0$, for sufficiently large $n$, with probability at
least $1-\delta$, we have $\cC_{\lambda, n} \neq \emptyset$, and
\begin{equation}\label{eq:normandalignment}
n^{-1}\|\bbm_\star\|^2 \in [q_\star-\eps,q_\star +\eps], \qquad
n^{-1}|\<\bbm_\star ,\bx\>| \in [q_\star -\eps,q_\star +\eps].
\end{equation}
Denoting the above good event to be $\cE$. Then we have
\begin{align*}
n^{-2}\E\Big[\|\bbm_\star \bbm_\star ^\sT-\bx\bx^\sT\|_F^2 \Big \vert  \cE \Big] =&n^{-2}\E\Big[\|\bx\|^4+\|\bbm_\star \|^4-2|\<\bbm_\star ,\bx\>|^2 \vert \Big \vert \cE \Big] \\
\leq& \Big(1+(q_\star +\eps)^2-2(q_\star -\eps)^2\Big) \leq 1-q_\star ^2+6\eps,
\end{align*}
so that
\begin{align*}
n^{-2}\E\Big[\|\bbm_\star \bbm_\star ^\sT-\bx\bx^\sT\|_F^2 \Big] &\le
n^{-2}\E\Big[\|\bbm_\star \bbm_\star ^\sT-\bx\bx^\sT\|_F^2 \Big \vert  \cE \Big]
\P(\cE) + 2 \cdot \P(\cE^c)\\
&\le 1-q_\star ^2+6\eps +  2 \delta.
\end{align*}
As $\bbm_\star$ is a measurable function $\bY$, orthogonality of
$(\bx\bx^\sT-\widehat{\bX}_{\mathrm{Bayes}})=(\bx\bx^\sT-\E[\bx\bx^{\sT}|\bY])$
with respect to any function measurable on $\bY$, implies
\begin{align*}
n^{-2}\E\left[\|\bbm_\star \bbm_\star ^\sT-\bx\bx^\sT\|_F^2\right]
&=n^{-2}\E\left[\|\bbm_\star \bbm_\star ^\sT-\widehat{\bX}_{\mathrm{Bayes}}\|_F^2\right]
+\MMSE_n(\lambda).
\end{align*}
Then Proposition \ref{prop:asymptoticMMSE} implies, for all $n \geq
n_0(\lambda,\eps,\delta)$,
\[n^{-2}\E\left[\|\bbm_\star \bbm_\star ^\sT-\widehat{\bX}_{\mathrm{Bayes}}\|_F^2\right]
<7\eps+ 2 \delta.\]
Then by the definition of $\bbm_\star$, we have 
\[
 \Big(n^{-2} \sup_{\bbm \in \cC_{\lambda, n}} \|\bbm
\bbm^\sT-\widehat{\bX}_{\mathrm{Bayes}}\|_F^2\Big)\ones\{\cC_{\lambda,n} \neq
\emptyset\}\le n^{-2} \|\bbm_\star \bbm_\star ^\sT-\widehat{\bX}_{\mathrm{Bayes}}\|_F^2, 
\]
so that (\ref{eq:nearBayes}) holds. 
%
\end{proof}

\section{Proof of Proposition \ref{PROP:COMPNEG}}\label{sec:Proof_Prop_Negative_Complexity}

Denote $\xi \sim \cN(\lambda^2 \vphi_\star,\lambda^2q_\star)$. Note that $\varphi_\star = q_\star$ implies $\E \{f(-\xi)\} =\E\{e^{-2\xi}f(\xi)\}$ whence
$\E\{\tanh(\xi)^2\} = \E\{\tanh(\xi)\}$. Applying Gaussian
integration by parts, we may verify from (\ref{eq:qstar}) and (\ref{eq:estar})
that
\begin{equation}\label{eq:starredquantities}
\vphi_\star=\E[\tanh(\xi)],
\qquad q_\star=\E[\tanh^2(\xi)],
\qquad a_\star=\E[\xi \tanh(\xi)],
\end{equation}
\[u(q_\star,a_\star)-e_\star=\E[\log 2\cosh(\xi)].\]
Note that
\[S(q_\star,\vphi_\star,a_\star,e_\star;\mu,\nu,\tau,\gamma)
=-q_\star \mu-\vphi_\star \nu-a_\star\tau-[u(q_\star,a_\star)-e_\star]\gamma+\log I,\]
where
\[I=\E\Big[\exp\Big(\mu \tanh^2(\xi)+\nu\tanh(\xi)+\tau
\xi\tanh(\xi)+\gamma \log 2\cosh(\xi)\Big)\Big].\]
Then $S(q_\star,\vphi_\star,a_\star,e_\star;\mu,\nu,\tau,\gamma)$ is a convex
function of $(\mu,\nu,\tau,\gamma)$, with derivative in
$(\mu,\nu,\tau,\gamma)$ equal to 0 at $(\mu,\nu,\tau,\gamma)=0$. Thus
$S_\star(q_\star,\vphi_\star,a_\star,e_\star)=0$ follows.

In the remainder of this section, we establish
Eq. (\ref{eqn:quantitative_complexity_bound}). Note that $q_\star \to 1$ as $\lambda\to\infty$ (see, e.g. \cite[Lemma 3.2]{deshpande2016asymptotic}).
Throughout the proof, we assume $\lambda_0$ is large enough such that the
conditions of the proposition guarantee
$(q, \vphi) \in [1/(1 + \eps), 1]^2$ for $\eps \equiv 0.01$. 
Denote
\begin{align*}
S_+(q, \vphi, a)&= \frac{1}{4\lambda^2 q^2}[a - \lambda^2 \vphi^2 -
\lambda^2(1-q)q]^2\\
S_-(q, \vphi, a, e; \mu, \nu, \tau, \gamma)&= - q \mu - \vphi \nu - a \tau -
[u(q, a) - e] \gamma + \log I,
\end{align*}
so that $S=S_++S_-$ and 
\[S_\star(q, \vphi, a, e) = S_+(q, \vphi, a) + \inf_{\mu, \nu, \tau,
\gamma}S_-(q, \vphi, a, e;\mu, \nu, \tau, \gamma).\]
Denote $Q(x) = \tanh^2 x$, $M(x) = \tanh x$, $A(x) = x \tanh x$, and $U(x) =
\log [2 \cosh (x)]$.\\

\noindent
{\bf Step 1. Upper bound $S_+$. }
Noting $a_\star = \lambda^2 \vphi_\star^2 + \lambda^2 q_\star(1 - q_\star)$, we have
\[S_+(q, \vphi, a) = \frac{1}{4 \lambda^2 q^2}\{(a - a_\star) - (\lambda^2 \vphi^2 - \lambda^2
\vphi_\star^2) - [\lambda^2 q (1 - q) - \lambda^2 q_\star (1 - q_\star)]\}^2.\]
By the inequality
\[(x_1 + x_2 + x_3)^2 \le (1 + \eps)^2 x_1^2 + (2 + 2/\eps)^{2} x_2^2 + (2 +
2/\eps)^{2} x_3^2,\]
and the conditions $q, \vphi \in [1/(1+\eps), 1]$ and $\eps = 0.01$, we have
\begin{equation}\label{eqn:Bound_for_S_plus}
\begin{aligned}
S_+(q, \vphi, a)
&\le (1 + \eps)^2 (a - a_\star)^2/(4 \lambda^2 q^2)\\
&\hspace{0.1in}+ (2 + 2/\eps)^2 \lambda^4
\{(\vphi^2 - \vphi_\star^2)^2
+ [q(1 - q) - q_\star(1 - q_\star)]^2\}/(4 \lambda^2 q^2)\\
&\le (1 + \eps)^4 (a - a_\star)^2/(4 \lambda^2) + (10/\eps^2) \lambda^2  (\vphi - \vphi_\star)^2 + (10/\eps^2) \lambda^2 (q - q_\star)^2. 
\end{aligned}
\end{equation}

\noindent
{\bf Step 2. Bound the moment generating function of $Q$, $M$, $A$, and $U$. }
Let $x \sim \cN(\lambda^2 \vphi,\lambda^2 q)$. Let $\E$ and $\Var$ be taken with respect to the randomness of $x$. We bound $\E[e^{\mu (Q - \E Q)}]$, $\E[e^{\nu (M - \E M)}]$, $\E[e^{\tau (A - \E A)}]$, and $\E[e^{\gamma (U - \E U)}]$. 

Since $|Q| \leq 1$, we have for $|\mu|<1$
\begin{equation}\label{eqn:MGF_bound_for_Q}
\begin{aligned}
&\E[\exp\{\mu(Q - \E Q)\}]\\
&\leq 1 + \frac{1}{2}\mu^2 \Var(Q)+  \sum^{\infty}_{k=3} \frac{1}{k!}\l
\mu\l^k \E[\l Q - \E Q\l^k] \\
&\le 1 +  \frac{1}{2}\mu^2\Var(Q) \sum_{k=2}^\infty  \l \mu\l ^{k-2}
\le \exp\{\mu^2 \Var(Q)/[2(1 -  \l \mu\l )]\}.
\end{aligned}
\end{equation}
Similarly, we have for $|\nu|<1$
\begin{equation}\label{eqn:MGF_bound_for_M}
\begin{aligned}
\E[\exp\{\nu(M - \E M)\}] &\le \exp\{\nu^2 \Var(M)/[2(1 - \l \nu\l )]\}. 
\end{aligned}
\end{equation}

To bound $\E[e^{\tau (A - \E A)}]$, note that $A(x) = x \tanh(x)$ is
$L_a$-Lipschitz. Indeed, simple calculus shows that $\sup_tA'(t) = A'(A^{-1}(1))= A^{-1}(1)$, whence $L_a = A^{-1}(1) \le 1.2$. Applying Gaussian concentration
of measure, see e.g.\ \cite[Theorem 5.5]{BLM}, we have 
\begin{align}\label{eqn:MGF_bound_for_A}
\E[\exp\{\tau(A - \E A)\}] \le \exp\{ \lambda^2 q \tau^2 L_a^2 /2\} \le \exp\{ \lambda^2 \tau^2 L_a^2 / 2 \}. 
\end{align}

Similarly, to bound $\E[e^{\gamma (U - \E U)}]$, note $U(x) = \log [2 \cosh(x)]$ is $1$-Lipschitz. Hence we have
\begin{align}\label{eqn:MGF_bound_for_U}
\E[\exp\{\gamma(U - \E U)\}] \le \exp\{ \lambda^2 q \gamma^2 /2\} \le \exp\{\lambda^2 \gamma^2/2 \}. 
\end{align}

\noindent
{\bf Step 3. Upper bound $S_-$. }
Set $\alpha = 1 + \eps$ and $\kappa = (3 + 3\eps)/\eps$, where $\eps=0.01$ as above.
By Holder's inequality
\begin{align*}
&\E[e^{\mu(Q - \E Q) + \nu(M - \E M) + \tau(A - \E A) + \gamma(U - \E U)}]\\
 &\le \E[e^{\alpha \tau(A - \E A)}]^{1/\alpha} \E[e^{\kappa \mu(Q - \E
Q)}]^{1/\kappa} \E[e^{\kappa \nu(M - \E M)}]^{1/\kappa} \E[e^{\kappa \gamma (U - \E U)}]^{1/\kappa}. 
\end{align*}
Given estimates (\ref{eqn:MGF_bound_for_Q}), (\ref{eqn:MGF_bound_for_M}),
(\ref{eqn:MGF_bound_for_A}), and (\ref{eqn:MGF_bound_for_U}), we have for
$|\mu|,|\nu|<1/\kappa$
\[
\begin{aligned}
&\E[e^{\mu(Q - \E Q) + \nu(M - \E M) + \tau(A - \E A) + \gamma(U - \E U)}]\\
&\le \exp\Big\{ \frac{ \alpha \tau^2 \lambda^2 L_a^2}{2} +  \frac{\kappa \mu^2 \Var(Q)}{2(1 - \kappa \l \mu\l )} + \frac{\kappa \nu^2 \Var(M)}{2(1 - \kappa \l \nu\l )} + \frac{\kappa \gamma^2 \lambda^2 }{2} \Big\}, \\
\end{aligned}
\]
and hence
\begin{equation}\label{eqn:Bound_for_S_-_1}
\begin{aligned}
&S_-(q, \vphi, a, e; \mu, \nu, \tau, \gamma)\\
&\le  - \tau(a - \E[A]) + \frac{ \alpha \tau^2 \lambda^2 L_a^2}{2}  - \mu(q - \E[Q])  +  \frac{\kappa \mu^2 \Var(Q)}{2(1 - \kappa \l \mu\l )} \\
&\hspace{0.2in}- \nu(\vphi - \E[M]) + \frac{\kappa \nu^2 \Var(M)}{2(1 - \kappa \l \nu\l )} - \gamma (u(q, a) - e - \E[U]) + \frac{\kappa \gamma^2 \lambda^2}{2} .
\end{aligned}
\end{equation}

To bound $\Var(Q)$ and $\Var(M)$, note that $\tanh(x) \geq 1-e^{-x}$. Then
for $(q, \vphi) \in [1/(1+\eps), 1]^2$ with $\eps=0.01$,
there exists a universal constant $\lambda_0$ such that when $\lambda \ge \lambda_0$, we have (denoting $\phi(x) = \P(\l G \l \ge x)$ for $G \sim \cN(0,1)$)
\begin{equation}
\begin{aligned}\label{eqn:Bound_for_1_minus_tanh}
\E[1 - \tanh x] =& \E[(1 - \tanh x) \ones\{ x \ge \lambda^2/2 \}] + \E[(1 - \tanh x) \ones\{ x < \lambda^2/2  \} ]\\
\le& e^{- \lambda^2/2} + 2 \P( x < \lambda^2/2 ) \le e^{- \lambda^2/2} + \phi(\lambda(\vphi - 1/2)) \le e^{- \lambda^2/5},
\end{aligned}
\end{equation}
and hence
\[\Var(M) = \E[\tanh^2 x] - \E[\tanh x]^2 \le (1 + \E[\tanh x]) (1 - \E[\tanh
x]) \le e^{-\lambda^2 / 10}.\]
Similarly, $\Var(Q) \le e^{-\lambda^2 / 10}$.

Now we take $L_q^2 \ge e^{-\lambda^2/10}$ to be determined, and we take $\tau, \mu, \nu, \gamma$ to be 
\[\tau  = \frac{ a - \E[A] }{ \alpha \lambda^2 L_a^2}, ~~~~ \mu = \frac{q -
\E[Q] }{2 \kappa L_q^2}, ~~~~ \nu = \frac{ \vphi - \E[M]}{2 \kappa L_q^2}, ~~~~
\gamma = \frac{u(q, a) - e- \E[U]}{\kappa \lambda^2} \]
Then, as long as
\begin{equation}\label{eqn:region_a_q_phi_1}
\begin{aligned}
\max\{ \l q - \E[Q]\l, \l \vphi - \E[M]\l \} \le& L_q^2, 
\end{aligned}
\end{equation}
we have $\l\mu\l, \l \nu \l \le 1/(2 \kappa)$, and according to Eq. (\ref{eqn:Bound_for_S_-_1}), we have
\begin{align}
&\inf_{\mu, \nu, \tau, \gamma} S_-(q, \vphi, a, e; \mu, \nu, \tau,
\gamma)\nonumber\\
&\le - \frac{(a - \E[A])^2}{2 \alpha \lambda^2 L_a^2  }  - \frac{(q - \E [Q])^2}{4 \kappa L_q^2 } - \frac{(\vphi - \E [M])^2}{4 \kappa L_q^2} - \frac{(u(q, a) - e - \E [U])^2}{2 \kappa \lambda^2}. \label{eq:InfSminus}
\end{align}

\noindent
{\bf Step 4. Bound $\E A - a_\star$, $\E Q - q_\star$, $\E M - \vphi_\star$, and
$\E U - (u(q_\star, a_\star) - e_\star)$. }

{ Recall that the expectations in (\ref{eq:InfSminus}) are with respect to
$x \sim \cN(\lambda^2\vphi,\lambda^2q)$ while the quantities of
(\ref{eq:starredquantities}) are defined with $\xi \sim
\cN(\lambda^2\vphi_\star,\lambda^2q_\star)$. To bound this difference,}
define $D_{F, \#} = \sup_{q, \vphi \in [0.9,1]^2} \l (\d / \d \#) \E_G [F
(\lambda^2 \vphi + \sqrt{\lambda^2 q} G)]) \l $ for $F = A, Q, M, U$ and $\# =
q, \vphi$, where $G \sim \cN(0,1)$. Now we bound $D_{F, \#}$. We denote $x =
\lambda^2 \vphi + \sqrt{\lambda^2 q} G$, where (to simplify notation) in each line below $\vphi,q\ge 0.9$ are chosen to maximize the corresponding expression. 
Applying
(\ref{eqn:Bound_for_1_minus_tanh}) and Gaussian integration by parts,
we obtain
\begin{align*}
D_{A, \vphi} &\le \l (\d/\d \vphi) \E[ x \tanh x] \l =\lambda^2 \l \E[\tanh x +
x(1 - \tanh^2 x)] \l \le 2 \lambda^2,\\
D_{M, \vphi} &\le \l (\d/\d \vphi) \E[ \tanh x] \l \le \lambda^2 \E[1 - \tanh^2
x] \le 2 \lambda^2 e^{-\lambda^2/5}.
\end{align*}
Similar arguments, omitted for brevity, show
\[D_{A, q},D_{U,\varphi} \le \lambda^2,\;
D_{M, q} \le 2 \lambda^2 e^{-\lambda^2/5},\;
D_{Q, \vphi},D_{Q,q} \le 4 \lambda^2 e^{-\lambda^2/5},\;
D_{U, q} \le \lambda^2 e^{-\lambda^2/5}.\]
Moreover, denoting $D_{u, q}$ as the Lipschitz constant of $u(q, a)$ with respect to $q$, we have
\[
D_{u, q} \le \sup_{q\in [0.9,1]} \l (\d/\d q) u(q, a) \l \le \sup_{q\in [0.9,1]} \lambda^2 q / 2 \le \lambda^2/2.
\]

For $(f, F) = (q, Q)$ or $(\vphi, M)$, we have
\[
\begin{aligned}
\l f - \E[F(x)] \l \le& \l f - f_\star\l + D_{F, q} \l  q - q_\star\l + D_{F, \vphi} \l  \vphi - \vphi_\star \l. \\
\end{aligned}
\]
For $\lambda \ge 10$, we have $D_{Q, q} \vee D_{Q, \vphi} \vee D_{M, q} \vee D_{M, \vphi} \le 1/2$, and hence
\begin{align}\label{eqn:region_a_q_phi_2}
\max\{ \l q - q_\star \l, \l \vphi - \vphi_\star \l \} \le  L_q^2/2
\end{align}
implies (\ref{eqn:region_a_q_phi_1}). 

Moreover, for $(f, F) = (q, Q)$, $(\vphi, M)$, or $(a, A)$, we have
\[
\begin{aligned}
(f - \E[F(x)])^2 \ge& (\l f - f_\star\l - D_{F, q} \l q - q_\star\l - D_{F, \vphi} \l \vphi - \vphi_\star\l )_+^2\\
\ge& ( f - f_\star)^2 - 2 \l f - f_\star\l \cdot ( D_{F, q} \l q - q_\star \l + D_{F, \vphi} \l \vphi - \vphi_\star\l )\\
\ge& (1 - \eps)(f - f_\star)^2 - (1/\eps) ( D_{F, q} \l q  - q_\star\l + D_{F, \vphi} \l \vphi - \vphi_\star\l )^2\\
\ge& (1 - \eps)(f - f_\star)^2 - (2/\eps) D_{F, q}^2 (q  - q_\star)^2 - (2/\eps) D_{F, \vphi}^2 (\vphi - \vphi_\star)^2,
\end{aligned}
\]
and 
\[
\begin{aligned}
& (u(q, a) - e - \E[U])^2\\
&\ge (\l (e - e_\star) - (a - a_\star)/2\l - (D_{U, q} + D_{u, q}) \l q - q_\star\l - D_{U, \vphi} \l \vphi - \vphi_\star \l  )_+^2\\
&\ge (1 - \eps)[ (e - e_\star) - (a - a_\star)/2]^2 - (2/\eps)(D_{U, q} + D_{u, q})^2(q - q_\star)^2 - (2/\eps) D_{U, \vphi}^2 (\vphi - \vphi_\star)^2. \\
\end{aligned}
\]

Accordingly, for $\lambda_0$ sufficiently large, using Eq.~(\ref{eq:InfSminus}), we obtain
\begin{equation}\label{eqn:Bound_for_S_minus}
\begin{aligned}
&\inf_{\mu, \nu, \tau, \gamma} S_-(q, \vphi, a, e; \mu, \nu, \tau, \gamma)\\ 
\le& - \frac{1 - \eps}{2 \alpha L_a^2 \lambda^2 } (a - a_\star)^2  \\
&- \Big[\frac{1 - \eps}{4 \kappa L_q^2} - \frac{2 D_{A, q}^2}{2 \eps \alpha L_a^2 \lambda^2} - \frac{2 D_{Q, q}^2}{4 \eps \kappa L_q^2}  - \frac{2 D_{M, q}^2}{4 \eps \kappa L_q^2} - \frac{2(D_{U, q} + D_{u, q})^2}{2 \eps \kappa \lambda^2}\Big] (q - q_\star)^2  \\
& - \Big[ \frac{1 - \eps}{4 \kappa L_q^2}  - \frac{2 D_{A, \vphi}^2}{2\eps \alpha L_a^2 \lambda^2} - \frac{2 D_{Q, \vphi}^2}{4 \eps \kappa L_q^2}  - \frac{2 D_{M, \vphi}^2}{4 \eps \kappa L_q^2}  -\frac{2 D_{U, \vphi}^2}{2 \eps \kappa \lambda^2} \Big] (\vphi - \vphi_\star)^2 \\
& - \frac{1 - \eps}{2 \kappa \lambda^2} [(e - e_\star) - (a - a_\star)/2]^2. 
\end{aligned}
\end{equation}
as long as $(q, \vphi)$ satisfies (\ref{eqn:region_a_q_phi_2}).\\

\noindent
{\bf Step 5. Finish the proof. }
Let us apply
\[S_\star(q, \vphi, a, e)
= S_+(q, \vphi, a) + \inf_{\mu, \nu, \tau, \gamma} S_-(q, \vphi, a, e; \mu,
\nu, \tau, \gamma)\]
and add the upper bounds from
(\ref{eqn:Bound_for_S_plus}) and (\ref{eqn:Bound_for_S_minus}).
As $\eps = 0.01$, $\alpha = 1 + \eps = 1.01$, $\kappa = (3 + 3 \eps) / \eps =
303$, and $L_a \le 1.2$, the coefficient for the $(a-a_\star)^2$ term
\[-\frac{1-\eps}{2\alpha L_a^2\lambda^2}+\frac{(1+\eps)^4}{4\lambda^2}\]
is a negative function of $\lambda$. Now take $L_q^2 = 2 \const_0/\lambda^{2}$ for some small constant
$\const_0>0$. For $\lambda_0$ large enough, this implies $L_q^2 \ge
\exp\{-\lambda^2/10\}$.

Note that $D_{Q, q}, D_{Q, \vphi}, D_{M, q}, D_{M, \vphi}$
are exponentially small in $\lambda$, while $D_{A, q}, D_{A, \vphi}, D_{U, q},
D_{u, q}, D_{U, \vphi} \le 2\lambda^2$. Hence for $\const_0$ small enough and
$\lambda_0$ large enough,  the coefficients for the $(q - q_\star)^2$ and $(\vphi
- \vphi_\star)^2$ terms are also negative functions of $\lambda$; moreover, $\max\{ \l
q - q_\star \l, \l \vphi - \vphi_\star\l\} \le \const_0/\lambda^{2}$ implies
(\ref{eqn:region_a_q_phi_2}). Finally, observe that for any
$c_1(\lambda),c_2(\lambda)>0$, there exists $c(\lambda)>0$ such that
\[-c_1(\lambda)(a-a_\star)^2-c_2(\lambda)[(e-e_\star)-(a-a_\star)/2]^2
\leq -c(\lambda)(a-a_\star)^2-c(\lambda)(e-e_\star)^2.\]
This concludes the proof of (\ref{eqn:quantitative_complexity_bound}).

\appendix

\section{Proof of Proposition \ref{PROP:KACRICE}}\label{sec:Proof_Basic_Kac_Rice}

Fix $n$ and write as shorthand $\bg \equiv \bg_n$ and $\bH \equiv \bH_n$.
For any measurable $T \subseteq (-1,1)^n$, define
\[\Crit(T)=\sum_{\bbm: \bg(\bbm) = \bzero} \ones\{ \bbm \in T \}.\]
We wish to apply the Kac-Rice formula \cite[Theorem 11.2.1]{adlertaylor} for $\E[\Crit(T)]$. The statement of
\cite[Theorem 11.2.1]{adlertaylor} does not directly apply in our setting, as
$\bg(\bbm)$ and $\bH(\bbm)=\nabla \bg(\bbm)$ do not admit a joint density
on $\R^n \times \R^{n(n+1)/2}$. (Conditional on $\bH(\bbm)$, the
gradient $\bg(\bbm)$ is deterministic.) We will instead adapt the proof
presented in \cite{adlertaylor} to handle this technicality. For simplicity, we
prove only the upper bound, which is all that we require for our application.

Let $\ball(\delta)$ be the open ball of radius $\delta$ around $\bzero$.
Proposition \ref{PROP:KACRICE} is an immediate consequence of the following lemma. 

\begin{lemma}\label{lem:Kac_Rice}
Let $p_{\bbm}$ be the Lebesgue density of $\bg(\bbm)$.
Let $\delta \equiv \delta_n > 0$. Then for any Borel measurable set
$T \subseteq V_\delta \equiv [-1 + \delta, 1 - \delta]^n \setminus \ball(\delta)$,
\[\E[\Crit(T)] \leq \int \ones\{\bbm \in T\}\,
\E\Big[\big|\det \bH(\bbm)\big|  \Big| \bg(\bbm)=\bzero\Big]
p_{\bbm}(\bzero) \d\bbm.\]
\end{lemma}

To prove this lemma,
we apply the following result from \cite{adlertaylor}: Define the
smoothed delta function
\[\delta_\eps(\bbm)=\begin{cases}
\operatorname{Vol}(\ball(\eps))^{-1} & \bbm \in \ball(\eps) \\
0 & \bbm \notin \ball(\eps)\end{cases}\]
which integrates to 1 over $\R^n$.

\begin{lemma}\label{lemma:KRdeterministic}
Suppose $\bg:(-1,1)^n \to \R^n$ is deterministic and continuously
differentiable, and $T \subset (-1,1)^n$ is compact. Suppose furthermore that
there are no points $\bbm \in T$ satisfying both $\bg(\bbm)=\bzero$ and
$\det \nabla \bg(\bbm)=\bzero$, and also no points $\bbm \in \partial T$
satisfying $\bg(\bbm)=\bzero$. Then
\[\Crit(T)=\lim_{\eps \to 0} \int
\ones\{\bbm \in T \} \delta_\eps(\bg(\bbm))\,|\det \nabla \bg(\bbm)|\d\bbm.\]
\end{lemma}
\begin{proof}
See \cite[Theorem 11.2.3]{adlertaylor}.
\end{proof}

The below verifies that the conditions required for Lemma
\ref{lemma:KRdeterministic} hold almost surely.

\begin{lemma}\label{lemma:morse}
Let $\bg \equiv \bg_n$
be defined by (\ref{eq:gn}), and let $T \subset (-1,1)^n$ be compact
with $\bzero \notin T$ and $\partial T$ having Lebesgue measure $0$. Then the conditions of Lemma \ref{lemma:KRdeterministic} hold with
probability 1.
\end{lemma}
\begin{proof}

Let us first verify that with probability 1, no point $\bbm \in \partial T$
satisfies $\bg(\bbm)=\bzero$. Fix $C_0>0$ and consider the event
$\mathcal{E}  = \{ \|\bW\|_\op<C_0 \}$. 
As $T$ is compact and does not contain $\bzero$, it belongs to
$K_{2\delta} \equiv [-1+2\delta,1-2\delta]^n \setminus (-2\delta,2\delta)^n$
for some ($n,T$)-dependent quantity $\delta>0$. Since $\partial T$ has Lebesgue
outer measure $0$, for any $\eps > 0$ there exists a countable collection of
balls $\{\ball(\bbm_i, r_i):i \in \mathcal{I}\}$ such that
\[\sum_{i \in \mathcal{I}} r_i^{n}<\eps, \qquad
\partial T \subset \bigcup_{i \in \mathcal{I}} \ball(\bbm_i, r_i),\]
where each $\ball(\bbm_i, r_i)$ is the open ball of radius $r_i$ around
$\bbm_i$. Taking $\eps<\delta^n$ so that $r_i<\delta$ for each $i$,
we may assume without loss of generality that each center
$\bbm_i$ belongs to $K_\delta$ (because otherwise the ball has empty intersection with $K_{2\delta}$ and therefore with $T$). 
On the event $\mathcal{E}$, $\bg(\bbm)$ is
$L$-Lipschitz over $K_\delta$ for some $(n, \delta,C_0)$-dependent quantity $L>0$. Hence
\[\P\left[\;\mathcal{E} \cap
\{\text{there exists } \bbm \in \partial T:\bg(\bbm)=0\}\right]
\leq \sum_{i \in \mathcal{I}}\P\left[\|\bg(\bbm_i)\|_2<Lr_i\right].\]
Observe that for each fixed $\bbm_i$, the vector $\bg(\bbm_i)$ has a
multivariate normal distribution with covariance
\[\beta^2\left(n^{-1}\|\bbm_i\|_2^2\id+n^{-1}\bbm_i\bbm_i^\sT
\right) \succeq \beta^2 \delta^2 \id.\]
Then the density of $\bg(\bbm_i)$ is bounded as $\varphi_{\bm_i}(\bx) \le (2\pi\beta^2\delta^2)^{-n/2}$, and for some $C=C(\beta)>0$ not depending on $\eps,\delta$, we have
\begin{align*}
\P\left[\|\bg(\bbm_i)\|_2<Lr_i\right] \leq \left(\frac{CLr_i}{\delta}\right)^n.
\end{align*}
Hence
\begin{align*}
\P\left[\;\mathcal{E} \cap
\{\text{there exists } \bbm \in \partial T:\bg(\bbm)=0\}\right]
\leq\sum_{i \in \mathcal{I}} \left(\frac{CLr_i}{\delta}\right)^n\le  \left(\frac{C L}{\delta}\right)^n \eps.
\end{align*}
As $\eps>0$ is arbitrary, the above probability must be 0. Then
\[\P\left[\text{there exists } \bbm \in \partial T:\bg(\bbm)=0\right]
\leq 1-\P[\mathcal{E}].\]
Now taking $C_0 \to \infty$,
\[\P\left[\text{there exists } \bbm \in \partial T:\bg(\bbm)=0\right]=0.\]

Next, let us verify that with probability 1, no point $\bbm \in T$ satisfies
both $\bg(\bbm)=0$ and $\det \bH(\bbm)=0$. Define the set
\[S=\left\{\bbm \in (-1,1)^n:\frac{m_i}{1-m_i^2}
-\frac{2\beta^2}{n}\|\bbm\|_2^2\cdot m_i-\arctanh(m_i)=0 \quad \forall i \in [n]\right\},\]
and suppose first that $T \cap S=\emptyset$. Note that for any $r>0$, we may
construct a maximal packing $\{\ball(\bbm_i, r/2):i \in \mathcal{I}\}$ where
$\bbm_i \in T$ for each $i$. Namely these balls do not intersect, and no additional
ball $\ball(\bbm,r/2)$ with $\bbm \in T$ may be added to the packing. Then
$\{\ball(\bbm_i, r):i \in \mathcal{I}\}$
is a finite cover of $T$, and a volume argument shows $|\mathcal{I}| \leq
C/r^n$ for an $n$-dependent constant $C>0$. As $T$ is compact
and $S$ is closed, $T$ is included in the set
\[U_\delta=\{\bbm \in [-1+\delta,1-\delta]^n:
\operatorname{dist}(\bbm,S) \geq \delta,\;\|\bbm\|_2 \geq \delta\}\]
for some $(n,T)$-dependent $\delta>0$. Fixing $C_0>0$ and defining the event
$\mathcal{E} = \{ \|\bW\|_\op<C_0\}$, we have on $\mathcal{E}$
that $\bg(\cdot)$ and $\det(\bH(\cdot))$ are both $L$-Lipschitz over $\bbm \in U_{\delta}$ for some $L = L(n, \delta, C_0) >0$. Hence
\begin{align*}
&\P\left[\;\mathcal{E} \cap
\{\text{there exists } \bbm \in T:\bg(\bbm)=0,\;\det \bH(\bbm)=0\}\right]\\
&\leq \sum_{i \in \mathcal{I}} \P\left[\mathcal{E},\;\|\bg(\bbm_i)\|_2<Lr,
\;|\det \bH(\bbm_i)|<Lr\right].
\end{align*}
Consider a fixed index $i$, and let $\bv$ and $\bA$ be such that
\[\bg(\bbm_i)=-\beta(\bW\bbm_i+\bv),\qquad
\bH(\bbm_i)=-\beta(\bW+\bA)\]
for $\bg$, $\bH$ as defined in (\ref{eq:gn}) and (\ref{eq:Hn}).
Since $\bbm_i \in U_{\delta}$, the conditions of Lemma \ref{lemma:KRhelper} below
are satisfied for some quantities $c_0 \equiv c_0(n,\delta,\lambda, \beta)>0$
and $C_0 \equiv C_0(n,\delta,\lambda, \beta)>0$.
Then we obtain, for all $r$ sufficiently small and some $C>0$ independent of
$r$,
\[\P\left[\mathcal{E},\;
\|\bg(\bbm_i)\|_2<Lr, \;|\det \bH(\bbm_i)|<Lr\right]<Cr^{n+1/(3n)}.\]
Applying this and $|\mathcal{I}| \leq C/r^n$ to the above, and taking $r \to
0$ followed by $C_0 \to \infty$, we obtain
\[\P\left[\text{there exists } \bbm \in T:\bg(\bbm)=0,\;\det
\bH(\bbm)=0\right]=0.\]

If $T \cap S \neq \emptyset$, this argument holds for the compact set
$T \setminus \{\bbm:\operatorname{dist}(\bbm,S)<\delta\}$
and any $\delta>0$. Taking a union bound over a countable sequence $\delta \to 0$,
and noting that $T \cap S$ is closed, we obtain
\[\P\left[\text{there exists } \bbm \in T \setminus S:\bg(\bbm)=0,\;\det
\bH(\bbm)=0\right]=0.\]

Finally, note that $S$ is the zero set of a non-trivial real analytic function, and thus $S$ has Lebesgue measure $0$ \cite{mityagin2015zero}. The same argument as for
$\partial T$ shows
$\P\left[\text{there exists } \bbm \in T \cap S:\bg(\bbm)=0\right]=0$,
and combining the above,
\[\P\left[\text{there exists } \bbm \in T:\bg(\bbm)=0,\;\det
\bH(\bbm)=0\right]=0.\]
\end{proof}

\begin{lemma}\label{lemma:KRhelper}
For any ($n$-dependent) quantities $c_0,C_0>0$, there exist
$C \equiv C(n,c_0,C_0)>0$ and $\eps_0 \equiv \eps_0(n,c_0,C_0)>0$ such
that the following holds: Let $\bv \in \R^n$, $\bbm \in (-1,1)^n$,
and $\bA \in \R^{n \times n}$ be any deterministic vectors/matrices
such that $\bA$ is symmetric, $\|\bA\|_\op<C_0$,
$\|\bA\bbm-\bv\|_2>c_0$, and $\|\bbm\|_2>c_0$. Let
$\bW \sim \GOE(n)$. Then for all $\eps \in (0,\eps_0)$,
\[\P\Big[\|\bW\bbm+\bv\|_2<\eps \text{ and } |\det(\bW+\bA)|<\eps \text{ and }
\|\bW\|_\op<C_0 \Big]<C\eps^{n+1/(3n)}.\]
\end{lemma}
\begin{proof}
Throughout the proof, $C$ and $c>0$ denote arbitrary
$(n, c_0, C_0)$-dependent constants that may change from line to line.

Note that $\bW\bbm+\bv$ has a multivariate normal distribution with covariance
\[n^{-1}\|\bbm\|_2^2\id+n^{-1}\bbm\bbm^\sT \succeq n^{-1}c_0^2 \id.\]
Then the density of $\bW\bbm+\bv$ is bounded by an $n$-dependent constant,
so $\P[\|\bW\bbm+\bv\|_2<\eps]<C\eps^n$. Hence it suffices to show
\begin{equation}\label{eq:smalldetprob}
\P\Big[|\det(\bW+\bA)|<\eps,\;\|\bW\|_\op<C_0 \Big| \bW\bbm=\bw
\Big]<C\eps^{1/(3n)}
\end{equation}
for any  deterministic $\bw$ satisfying $\|\bw+\bv\|_2<\eps<\eps_0$.

For this, define a deterministic orthogonal matrix $\bO \in \R^{n \times n}$ such that its
first column
is $\bbm/\|\bbm\|_2$, and the span of its first two columns contains $\bbm$ and
$\bA\bbm+\bw$. Set
\[\check{\bW}=\bO^\sT\bW\bO, \qquad
\check{\bw}=\bO^\sT\bw, \qquad \check{\bA}=\bO^\sT\bA\bO.\]
Then, rotating coordinates and denoting by $\be_i$ the $i$th standard basis
vector,
\begin{align*}
&\P\Big[|\det(\bW+\bA)|<\eps,\;\|\bW\|_\op<C_0 \Big|\bW\bbm=\bw\Big]\\
&=\P\Big[|\det(\check{\bW}+\check{\bA})|<\eps,\;\|\check{\bW}\|_\op<C_0 \Big|
\|\bbm\|_2\check{\bW}\be_1=\check{\bw}\Big].
\end{align*}
Conditional on $\|\bbm\|_2\check{\bW}\be_1=\check{\bw}$, the first column (and
also first row) of $\check{\bW}+\check{\bA}$ is deterministic and given by
\[(\check{\bW}+\check{\bA})\be_1
=\frac{\check{\bw}}{\|\bbm\|_2}+\check{\bA}\be_1
=\frac{1}{\|\bbm\|_2}\bO^\sT(\bw+\bA\bbm)=\alpha_1\be_1+\alpha_2 \be_2,\]
where the last equality holds
by construction of $\bO$ for some scalars $\alpha_1,\alpha_2$ which satisfy
$\alpha_1^2+\alpha_2^2=\|\bA\bbm+\bw\|_2^2/\|\bbm\|_2^2$. Then, denoting by
$\bH^{(1)}=\check{\bW}^{(1)}+\check{\bA}^{(1)}$ and
$\bH^{(12)}=\check{\bW}^{(12)}+\check{\bA}^{(12)}$ the lower-right
$(n-1) \times (n-1)$ and
$(n-2) \times (n-2)$ submatrices of $\bW+\bA$, and expanding the determinant
along the first column,
\[\det(\check{\bW}+\check{\bA})
=\alpha_1 \det \bH^{(1)}-\alpha_2^2 \det \bH^{(12)}.\]

For sufficiently small $\eps_0$, the given conditions and
$\|\bw+\bv\|_2<\eps_0$ imply $\alpha_1^2+\alpha_2^2>c$. We consider two cases:
\begin{itemize}
\item Case 1: $|\alpha_1|<\eps^{1/3}$. Then $\alpha_2^2>c$ for a constant
$c>0$. For $\| \bW \|_\op \le C_0$ and $\| \bA \|_\op \le C_0$, we have $\| \bH^{(1)} \|_\op \le 2 C_0$ and hence $\l \det(\bH^{(1)})\l \le (2C_0)^{n-1} $. Combining with $|\det(\check{\bW}+\check{\bA})|<\eps$, we have
$|\det \bH^{(12)}|<C\eps^{1/3}$ for some constant $C>0$. Also,
$\| \bW \|_\op<C_0$ and $\|\bA\|_\op \le C_0$ imply $\|\bH^{(12)}\|<2C_0$. Hence
\begin{align}
&\P\Big[|\det(\check{\bW}+\check{\bA})|<\eps,\;\|\check{\bW}\|_\op<C_0 \Big|
\|\bbm\|_2\check{\bW}\be_1=\check{\bw}\Big]\nonumber\\
&\leq \P\Big[|\det \bH^{(12)}|<C\eps^{1/3},\;\|\bH^{(12)}\|_\op<2C_0\Big]\label{eq:case1}
\end{align}
for a constant $C>0$ and sufficiently small $\eps$. Writing the spectral decomposition
$\bH^{(12)}=\bU\bLambda\bU^\sT$ where
$\Lambda=\diag(\lambda_1,\ldots,\lambda_{n-2})$,
and applying the change of variables
$\d\bH^{(12)}=(1/Z')\prod_{i<j} |\lambda_i-\lambda_j|\,\d\Lambda\,\d\bU$ \cite{AGZbook},
the joint density of ordered eigenvalues $\lambda_1,\ldots,\lambda_{n-2}$
of $\bH^{(12)}$ is given by
\begin{align*}
&(1/Z)\ones\{\lambda_1 \leq \ldots \leq \lambda_{n-2}\}
\prod_{i<j} |\lambda_j-\lambda_i|\\
&\hspace{0.5in}\cdot \int
\exp\left(-\frac{n}{4}\Tr \big[(\bU\bLambda\bU^\sT-\check{\bA}^{(12)})^2\big]\right)
\d\bU,\end{align*}
where the integral is over the orthogonal group of dimension $n-2$. This density
is bounded by an $n$-dependent constant over the set
\[\left\{\lambda_1 \leq \ldots \leq \lambda_{n-2}:\prod_i |\lambda_i|<
C\eps^{1/3}, \;\max_i |\lambda_i|<2C_0\right\},\]
and the Lebesgue volume of this set is bounded by $C\eps^{1/(3n)}$ since at
least one coordinate $|\lambda_i|$ is less than $C\eps^{1/(3n)}$.
Then the right side of (\ref{eq:case1}) is at most $C\eps^{1/(3n)}$ for a
constant $C>0$. 
\item Case 2: $|\alpha_1| \geq \eps^{1/3}$. Considering separately the events
$|\det \bH^{(12)}|<\eps^{1/3}$ and $|\det \bH^{(12)}| \geq \eps^{1/3}$, and
handling the first event by the argument of Case 1 above, we obtain
\begin{align*}
&\P\Big[|\det(\check{\bW}+\check{\bA})|<\eps,\;\|\check{\bW}\|_\op<C_0 \Big|
\|\bbm\|_2\check{\bW}\be_1=\check{\bw}\Big]\\
&\leq C\eps^{1/(3n)}+\P\Big[|\det(\check{\bW}+\check{\bA})|<\eps,\;
|\det \bH^{(12)}| \geq \eps^{1/3} \Big|\|\bbm\|_2\check{\bW}\be_1=\check{\bw}
\Big]\, .
\end{align*}
For the probability on the right side, let us further
condition on all entries of $\check{\bW}$ except $\check{W}_{22}$. We have
\[\det(\check{\bW}+\check{\bA})=\alpha_1 \check{W}_{22}\det
\bH^{(12)}+\text{const}\]
where $\text{const}$ denotes a quantity that does not depend on $\check{W}_{22}$.
Then, when $|\alpha_1| \geq \eps^{1/3}$ and
$|\det \bH^{(12)}| \geq \eps^{1/3}$, the quantity
$\det(\check{\bW}+\check{\bA})/\eps$ is conditionally
normally distributed with variance at least $2\eps^{-2/3}/n$. This normal
distribution has density upper bounded by $C\eps^{1/3}$, so
\[\P\Big[|\det(\check{\bW}+\check{\bA})|<\eps,\;
|\det \bH^{(12)}| \geq \eps^{1/3} \Big|\|\bbm\|_2\check{\bW}\be_1=\check{\bw}
\Big] \leq C\eps^{1/3}.\]
\end{itemize}
Combining these cases yields (\ref{eq:smalldetprob}) as desired.
\end{proof}

\begin{proof}[Proof of Lemma \ref{lem:Kac_Rice}]

First consider $T \subseteq V_\delta \equiv [-1 + \delta, 1 - \delta]^n
\setminus \ball(\delta)$ to be a closed hyperrectangle. By Lemma \ref{lemma:KRdeterministic}, Lemma \ref{lemma:morse},
Fatou's Lemma, and Fubini's Theorem,
\[\E[\Crit(T)] \leq \liminf_{\eps \to 0}
\int \ones\{\bbm \in T\} \E\Big[\delta_\eps(\bg(\bbm))\,|\det
\bH(\bbm)|\Big]\d\bbm.\]
Denoting by $p_{\bbm}$ the density of $\bg(\bbm)$, we have
\begin{align*}
&\E\Big[\delta_\eps(\bg(\bbm))\,|\det \bH(\bbm)|\Big]\\
&=\operatorname{Vol}(\ball(\eps))^{-1}
\int_{\bu \in \ball(\eps)}\E\Big[|\det \bH(\bbm)|\Big|\bg(\bbm)=\bu\Big]
p_{\bbm}(\bu)\d\bu.
\end{align*}
Define 
\[D(\bu, \bbm) \equiv \E\Big[|\det \bH(\bbm)|\Big|\bg(\bbm)=\bu\Big]p_{\bbm}(\bu).\]
For any fixed $\bbm \in V_\delta$, the vector $\bg(\bbm)$ is a Gaussian random
vector, and $\E[\bg(\bbm)]$ and $\E[\bg(\bbm) \bg(\bbm)^\sT]$ are continuous
functions in $\bbm$. Hence, the density $p_{\bbm}(\bu)$ is a continuous
function of $(\bu, \bbm) \in \overline{\ball(\eps)} \times T$. Moreover, by Lemma \ref{lem:conditional_GOE}, we have
$[\bH(\bbm) \l \bg(\bbm) = \bu] \stackrel{d}{=} -\beta \Proj_\bbm^\perp \bW
\Proj_{\bbm}^\perp + A(\bu, \bbm)$, where $\bW \sim \GOE(n)$, $\Proj_\bbm^\perp$
is the projection orthogonal to $\bbm$, and $A(\bu, \bbm)$ is continuous in $(\bu, \bbm) \in \overline{\ball(\eps)} \times T$. Therefore, $D(\bu, \bbm)$ is continuous in $(\bu, \bbm) \in \overline{\ball(\eps)} \times T$, and hence it is bounded on $\overline{\ball(\eps)} \times T$. The bounded convergence theorem and the continuity of $D(\bu, \bbm)$ yield
\begin{align*}
\E[\Crit(T)] &\leq \int \ones\{\bbm \in T\}
\liminf_{\eps \to 0} \E\Big[\delta_\eps(\bg(\bbm))\,|\det
\bH(\bbm)|\Big]\d\bbm\\
&=\int \ones\{\bbm \in T\} D(\bzero, \bbm)\d\bbm. 
\end{align*}

Finally, we generalize the above inequality to a general Lebesgue measurable set $E \subseteq V_\delta$. This follows by standard measure theory machinery, and we sketch the proof in the following.

For any Lebesgue measurable set $E \subseteq V_\delta$, define $\nu_0(E) = \E[\Crit(E)]$ and
$\nu(E) = \int \ones\{\bbm \in E\} D(\bzero, \bbm)\d\bbm$. Define
\[\cT \equiv \{ T \subseteq V_\delta: T \text{ is closed hyperrectangle} \}.\]
Let $\nu_\star$
be the outer measure generated by the set function $\nu: \cT \rightarrow \R$, i.e. for any $E \subset V_\delta$, 
\begin{align}
\nu_\star(E) \equiv \inf\left\{\sum_{i=1}^\infty \nu(T_i):\; E \subseteq \bigcup_i T_i,\,\, T_i \in \cT \right\}.  
\end{align}

Note that we have shown $\nu_0(T) \le
\nu(T)$ for any closed hyperrectangle $T \in \cT$. Then for
any Lebesgue measurable set $E \subseteq V_\delta$,
\begin{align*}
\nu_0(E) &\stackrel{(i)}{\le} \inf\left\{\sum_{i=1}^\infty \nu_0 (T_i):\; E
\subseteq \bigcup_i T_i,\,\, T_i \in \cT \right\}\\
&\stackrel{(ii)}{\le} \inf\left\{\sum_{i=1}^\infty \nu(T_i):\; E \subseteq \bigcup_i T_i,\,\, T_i \in \cT \right\} \stackrel{(iii)}{=} \nu_\star(E).
\end{align*}
The inequality $(i)$ is given by the definition $\nu_0(E) = \E[\Crit(E)]$, the nonnegativity and additivity of $\Crit(\, \cdot \,)$, and the linearity of expectation. The inequality $(ii)$ is given by $\nu_0(T_i) \le \nu(T_i)$ for each $T_i \in \cT$. The equality $(iii)$ is given by the definition $\nu_\star$. 

Since $D(\bzero, \bbm)$ is bounded over $\bbm
\in V_\delta$, we have that $\nu$ is absolutely continuous with respect to
Lebesgue measure,
and $\nu_\star(E)=\nu(E)$ for any Lebesgue measurable set $E \subseteq
V_\delta$ by a standard argument. This concludes the proof.
\end{proof}

\section{Proof of Proposition \ref{PROP:ABSTRACT}}\label{sec:Proof_Kac_Rice_Abstract}

The proof applies the following version of Varadhan Lemma's upper bound, which
is a small extension of \cite[Lemma 4.3.6]{dembozeitouni}.

\begin{lemma}\label{lemma:varadhan}
Let $\{Z_n\}_{n=1}^\infty$ be random variables taking values in a
regular topological space $\mathcal{X}$.
Let $I:\mathcal{X} \to [0,\infty]$ be a lower semi-continuous rate function
with compact level sets $\Psi(\alpha) \equiv \{x:I(x) \leq \alpha\}$.
For all closed sets $A \subseteq \mathcal{X}$,
suppose $Z_n$ satisfies the large deviations upper bound
\begin{equation}\label{eq:LDPupper}
\limsup_{n \to \infty} \frac{1}{n}\log \P[Z_n \in A] \leq -\inf_{x \in A}
I(x).
\end{equation}
Consider an upper semi-continuous function $\phi:\mathcal{X} \to [-\infty,\infty)$
and a Borel set $A \subseteq \mathcal{X}$ such that
$A \cap \{x:\phi(x) \geq -\alpha\}$ is closed for each $\alpha<\infty$, and
\begin{equation}\label{eq:LDPtailcondition}
C(\gamma) \equiv \limsup_{n \to \infty} \frac{1}{n} \log \E[e^{\gamma n\phi(Z_n)}\ones\{Z_n \in
A\}]<\infty
\end{equation}
for some $\gamma>1$. Then
\[\limsup_{n \to \infty} \frac{1}{n} \log \E[e^{n\phi(Z_n)}\ones\{Z_n \in A\}] \leq
\sup_{x \in A} (\phi(x)-I(x)).\]
\end{lemma}
\begin{proof}
The proof follows \cite[Lemmas 4.3.6 and 4.3.8]{dembozeitouni} with minor
modifications. Fixing $\alpha<\infty$ and $\delta>0$,
by compactness and semi-continuity,
we may take a finite open cover $\{A_{x_1},\ldots,A_{x_N}\}$ 
of $A \cap \{x:\phi(x) \geq -\alpha\} \cap \{x:I(x) \leq \alpha\}$ such that
\begin{equation}\label{eq:cover}
\inf_{y \in \overline{A_{x_i}}} I(y) \geq I(x_i)-\delta, \qquad
\sup_{y \in \overline{A_{x_i}}} \phi(y) \leq \phi(x_i)+\delta.
\end{equation}
Then
\[\limsup_{n \to \infty} \frac{1}{n}\log \E[e^{n\phi(Z_n)}\ones\{Z_n \in A\}]
\leq \max(E_1,\ldots,E_N,R_1,R_2)\]
where
\begin{align*}
E_i&=\limsup_{n \to \infty} \frac{1}{n}\log \E[e^{n\phi(Z_n)}\ones\{Z_n \in
\overline{A_{x_i}}\}]\\
R_1&=\limsup_{n \to \infty} \frac{1}{n}\log
\E[e^{n\phi(Z_n)}\ones\{Z_n \in \{x:\phi(x)<-\alpha\}\}]\\
R_2&=\limsup_{n \to \infty} \frac{1}{n}\log
\E[e^{n\phi(Z_n)}\ones\{Z_n \in A \cap \{x:\phi(x) \geq -\alpha\} \cap
A_{x_1}^c \cap \ldots \cap A_{x_N}^c\}].
\end{align*}
Together, (\ref{eq:LDPupper}) and (\ref{eq:cover}) imply
$E_i \leq \phi(x_i)-I(x_i)+2\delta$. Also by definition, $R_1 \leq -\alpha$.
For $R_2$, denote $B=A \cap \{x:\phi(x) \geq -\alpha\} \cap
A_{x_1}^c \cap \ldots \cap A_{x_N}^c$. Note that $B$ is closed and $I(x)>\alpha$
on $B$. Then by Holder's inequality and (\ref{eq:LDPupper}) and
(\ref{eq:LDPtailcondition}),
\begin{align*}
R_2 &\leq \limsup_{n \to \infty} \frac{1}{n}\log
\left(\E\left[e^{\gamma n\phi(Z_n)}\ones\{Z_n \in A\}\right]^{1/\gamma}
\P\left[Z_n \in B\right]^{1-1/\gamma}\right)\\
&\leq \frac{1}{\gamma}C(\gamma)+\left(1-\frac{1}{\gamma}\right)\alpha.
\end{align*}
The result follows from combining these bounds and then
taking $\alpha \to \infty$ and $\delta \to 0$.
\end{proof}

We now apply this to our setting. We equip $\cP$ with
the topology of weak convergence and the corresponding Borel $\sigma$-algebra.
The following lemmas verify the conditions needed in Lemma \ref{lemma:varadhan}.

\begin{lemma}\label{lemma:Jproperties}
\hspace{1in}
\begin{enumerate}
\item[(a)] For any constants $\eta>0$ and $\beta > 0$, there exists a constant
$C \equiv C(\beta,\eta)<\infty$ such that $J(\rho,y) \leq
C +y^2/(4\beta^2)$ for all $\rho \in \cP$ with $Q(\rho)  \geq \eta$.
\item[(b)] $J:\{\rho \in \cP:Q(\rho)>0\} \times \R \to [-\infty,\infty)$
is upper semi-continuous.
\end{enumerate}
\end{lemma}
\begin{proof}
For $(a)$, we apply the inequality $(c+1)a^2+(c^{-1}+1)b^2 \geq
(a+b)^2$ with $c=2$, $a=\arctanh(x)- \beta \lambda \vphi+\beta^2(1-q)x-yx$, and
$b=yx$. Recalling $g(x;\vphi,q,y)$ from Eq.\ (\ref{eq:gx}), this yields
\[\log g(x;\vphi,q,y) \leq \log \frac{1}{1-x^2}
-\frac{(\arctanh(x)-\beta \lambda\vphi+\beta^2(1-q)x)^2}{6\beta^2q}
+\frac{(yx)^2}{4\beta^2q}.\]
By the boundedness of $M(\rho)$ and $Q(\rho)$ and the comparison of
$\log 1/(1-x^2)$ and $\arctanh(x)^2$ as $x \to \pm 1$, we see that for
$\vphi =M(\rho)$ and $q=Q(\rho)$, the first two terms above are together
upper-bounded by an $\eta$-dependent constant. Then
\[J(\rho,y) \leq C +\int \frac{(yx)^2}{4\beta^2Q(\rho)} \rho(\d x)
=C +\frac{y^2}{4\beta^2}\]
for a constant $C \equiv C(\beta,\eta)<\infty$.

For $(b)$, fix $y,y_1,y_2,\ldots \in \R$ such that $y_i \to
y$ and $\rho,\rho_1,\rho_2,\ldots \in \cP$ such that $\rho_i \to \rho$ weakly,
$Q(\rho)>0$, and $Q(\rho_i)>0$ for all $i$. Then $Q(\rho_i) \to Q(\rho)$, so
there is a lower bound $\eta>0$ on all $Q(\rho_i)$, as well as a finite upper
bound on all $y_i^2$. Fix a constant $\alpha \in \R$ and define
\[f_{i,\alpha}(x)=\max(\alpha,\log g(x;M(\rho_i),Q(\rho_i),y_i)),\]
\[f_\alpha(x)=\max(\alpha,\log g(x;M(\rho),Q(\rho),y)).\]
Then
$f_{i,\alpha}$ and $f_\alpha$ are uniformly bounded above and below over all
$i$. Furthermore, there is a value $\delta>0$ such that
$f_{i,\alpha}(x)=f_\alpha(x)=\alpha$ for all $x<-1+\delta$ and $x>1-\delta$ and 
all $i$. As $f_{i,\alpha}(x) \to f_\alpha(x)$ uniformly over
$x \in [-1+\delta,1-\delta]$, this implies $f_{i,\alpha}(x) \to f_\alpha(x)$
uniformly over $x \in (-1,1)$. Then
\begin{align*}
&\left\l \int f_{i,\alpha}(x)\rho_i(\d x)-\int f_\alpha(x)\rho(\d x)\right\l \\
&\leq \int \l f_{i,\alpha}(x)-f_\alpha(x)\l \rho_i(\d x)
+\left\l \int f_\alpha(x)(\rho_i-\rho)(\d x)\right\l  \to 0.
\end{align*}
Hence
\begin{align*}
\limsup_i J(\rho_i,y_i) &\leq
\lim_i \left(\frac{\beta^2(1-Q(\rho_i))^2}{2}-\frac{1}{2}\log (2\pi
\beta^2Q(\rho_i))+\int f_{i,\alpha}(x)
\rho_i(\d x)\right)\\
&=\frac{\beta^2(1-Q(\rho))^2}{2}-\frac{1}{2}\log (2\pi \beta^2Q(\rho))
+\int f_\alpha(x)\rho(\d x).
\end{align*}
The left side is independent of $\alpha$. As $x \mapsto g(x;M(\rho),Q(\rho),y)$
is bounded above, the monotone convergence theorem yields for the right side
\[\lim_{\alpha \to -\infty} \int f_\alpha(x)\rho(\d x)
=\int \log g(x;M(\rho),Q(\rho),y)\,\rho(\d x).\]
Then 
$\limsup_i J(\rho_i,y_i) \leq J(\rho,y)$, so $J$ is upper semi-continuous.
\end{proof}

\begin{lemma}\label{lemma:Econtinuous}
Fix any $\alpha \in (0,\infty)$, and suppose $\rho,\rho_1,\rho_2,\ldots \in \cP$
are such that $\rho_i \to \rho$ weakly and
\[\int_{-1}^1 \arctanh(x)^2\,\rho(\d x) \leq
\alpha, \qquad \int_{-1}^1 \arctanh(x)^2\,\rho_i(\d x) \leq \alpha\]
for all $i$. Then $A(\rho_i) \to A(\rho)$ and $E(\rho_i) \to E(\rho)$.
\end{lemma}
\begin{proof}
As $x^2$ and $h(x)$ are continuous and bounded over $(-1,1)$, it suffices to
show
\[\int x \arctanh(x)\,\rho_i(\d x) \to \int x \arctanh(x)\,\rho(\d x).\]
For any $\delta \in (0,1)$ such that $-1+\delta$ and $1-\delta$ are continuity
points of $\rho$, we have by weak convergence
\[\int_{[-1+\delta,1-\delta]} x\arctanh(x)\,\rho_i(\d x) \to
\int_{[-1+\delta,1-\delta]} x\arctanh(x)\,\rho(\d x).\]
Denoting $I_\delta=(-1,-1+\delta) \cup (1-\delta,1)$, by Cauchy-Schwarz
\[\left(\int_{I_\delta} x\arctanh(x)\,\rho(\d x)\right)^2
\leq \int_{I_\delta} x^2\rho(\d x) \cdot \int_{I_\delta}
\arctanh(x)^2\,\rho(\d x) \leq \alpha\int_{I_\delta} x^2\rho(\d x),\]
and similarly for $\rho_i$. As
$\int_{I_\delta} x^2\rho_i(\d x) \to \int_{I_\delta} x^2\rho(\d x)$
also by weak convergence, this yields
\[\limsup_i \left\l \int x\arctanh(x)\rho_i(\d x)
-\int x\arctanh(x) \rho(\d x)\right\l  \leq 2\alpha \int_{I_\delta}
x^2\rho(\d x).\]
Taking $\delta \to 0$ yields the claim.
\end{proof}

\begin{lemma}\label{lemma:Aclosed}
For each $\alpha<\infty$, the set
\[\Xi(\alpha) \equiv
\{(\rho,y) \in \cP \times \R:\,(Q(\rho),M(\rho),A(\rho),E(\rho)) \in
U,\; J(\rho,y) \geq -\alpha\}\]
is closed.
\end{lemma}
\begin{proof}
Suppose $(\rho,y),(\rho_1,y_1),(\rho_2,y_2),\ldots$ are such that
$(\rho_i,y_i) \to (\rho,y)$ and $(\rho_i,y_i) \in \Xi(\alpha)$ for all $i$.
Then $M(\rho_i) \to M(\rho)$ and $Q(\rho_i) \to Q(\rho)$ by weak convergence. In
particular $Q(\rho) \geq \eta$, so $J$ is well-defined at $\rho$ and
$J(\rho,y) \geq \limsup_i J(\rho_i,y_i) \geq -\alpha$ by the upper
semi-continuity established in Lemma \ref{lemma:Jproperties}(b).
Applying $(a-b)^2 \geq (a^2/2)-b^2$ and $\log
1/(1-x^2) \leq \arctanh(x)^2/(8\beta^2)+c$ for a constant
$c \equiv c(\beta)<\infty$, we have
\[\log g(x;\vphi,q,y) \leq c+\arctanh(x)^2\left(\frac{1}{8\beta^2}
-\frac{1}{4\beta^2q}\right)
+\frac{(\beta \lambda \vphi -\beta^2(1-q)x+yx)^2}{2\beta^2q}.\]
Since $y_i \to y<\infty$, the above bound and the
conditions $J(\rho_i,y_i) \geq -\alpha$,
$J(\rho,y) \geq -\alpha$, $Q(\rho_i) \in [\eta,1]$, and $Q(\rho) \in
[\eta,1]$ imply
$\int_{-1}^1 \arctanh(x)^2\,\rho_i(\d x) \leq \kappa$ and
$\int_{-1}^1 \arctanh(x)^2\,\rho(\d x) \leq \kappa$ for all $i$ and some
$\kappa<\infty$. Then Lemma \ref{lemma:Econtinuous} implies
$A(\rho_i) \to A(\rho)$ and $E(\rho_i) \to E(\rho)$. As $U$ is closed, this implies
$(Q(\rho),M(\rho),A(\rho),E(\rho)) \in U$, so $\Xi(\alpha)$ is closed as desired.
\end{proof}

\begin{proof}[Proof of Proposition \ref{PROP:ABSTRACT}]
Let $\rho_n$ denote the empirical measure of $m_1,\ldots,m_n \overset{iid}{\sim}
\pi_0$, where $\pi_0$ is the uniform distribution on $[-1, 1]$, and let $Y_n \sim \cN(0,\beta^2/n)$ be independent of $\rho_n$. Then
\begin{align}
T(U,V_n)&=2^n\cdot
\E\bigg[\exp(nJ(\rho_n,Y_n)+ C_0 n^{\max(0.9,b)})\nonumber\\
&\hspace{1in}\ones\{(Q(\rho_n), M(\rho_n), A(\rho_n), E(\rho_n)) \in U\}\bigg].\label{eq:kacricebound}
\end{align}
We apply Lemma \ref{lemma:varadhan} with $\mathcal{P} \times \R$ as
$\mathcal{X}$, $J(\rho,y)$ as $\phi(x)$, and
$\{(\rho,y) \in \cP \times \R:\,(Q(\rho),M(\rho),A(\rho),E(\rho)) \in U\}$ as $A$.
By Sanov's Theorem, $\rho_n$ satisfies a large deviation
principle (LDP) with rate $n$ and rate function $\rho \mapsto H(\rho\l \pi_0)$.
Then combining with the form of the normal density for $Y_n$,
$(\rho_n,Y_n)$ satisfies the LDP with rate function
$I(\rho,y)=H(\rho\l \pi_0)+y^2/(2\beta^2)$. This rate function is
lower semi-continuous with compact level sets.
Lemmas \ref{lemma:Jproperties}(b) and \ref{lemma:Aclosed} verify the
semi-continuity and closedness conditions required for $\phi(x)$ and $A$ in Lemma
\ref{lemma:varadhan}. Note further that by Lemma \ref{lemma:Jproperties}(a), for any
$\gamma \in (1,2)$ and some $C \equiv C(\beta,\eta)<\infty$,
\begin{align*}
&\limsup_{n \to \infty} \frac{1}{n}\log \E[e^{\gamma n
J(\rho_n,Y_n)}\ones\{Q(\rho_n) \geq \eta\}\\
&\leq \limsup_{n \to \infty}
\frac{1}{n}\log \E\left[e^{\gamma n(C+Y_n^2/(4\beta^2))}\right]=C.
\end{align*}
This verifies (\ref{eq:LDPtailcondition}), so the result follows by Lemma
\ref{lemma:varadhan}.
\end{proof}

\section{Proof of Proposition \ref{PROP:LOCALIZATION}}\label{sec:Proof_Localization}

\begin{lemma}\label{lem:initial_global_bound}
Fix any $\alpha \in (1/2,1)$.
Then there exist constants $\lambda_0$, $n_0$, and $\const_0 > 0$ so that for all
$\lambda \ge \lambda_0$ and $n \ge n_0$, with probability
at least $1 - e^{-\const_0 n}$, every $\bbm \in (-1, 1)^n$ with $\cF(\bbm) \le -
\alpha \lambda^2/2$ satisfies
\[
\begin{aligned}
Q(\bbm) \ge& (2 \alpha - 1 - 6/\lambda - 4 / \lambda^2)^{1/2}, \\
{ \l M(\bbm)\l} \ge& (2 \alpha - 1 - 6/\lambda - 4 / \lambda^2)^{1/4}.
\end{aligned}
\]
\end{lemma}

\begin{proof}

Note that we have
\[
\begin{aligned}
\cF(\bbm) \ge& -1   - \lambda/(2n) \cdot \< \bbm, \bW \bbm \> - (\lambda^2 / 2) M(\bbm)^2 - (\lambda^2/4)(1 - Q(\bbm))^2\\
\ge & -1 - \lambda \Vert \bW \Vert_{\op}/2 - (\lambda^2/2) M(\bbm)^2 - (\lambda^2 / 4) (1 - Q(\bbm))^2.
\end{aligned}
\]
With probability at least $1-e^{-c_0n}$ for $n \geq n_0$ and some $n_0,c_0>0$,
we have $\|\bW\|_\op<3$. On this event, for
any $\bbm$ such that $\cF(\bbm) \le -\alpha \lambda^2/2$, we have
\[
M(\bbm)^2 + 1 / 2 \cdot (1 - Q(\bbm))^2 \ge \alpha - 2/\lambda^2 - \Vert \bW
\Vert_{\op} / \lambda>\alpha - \eps,
\]
where $\eps = 2/\lambda^2 + 3 / \lambda $. By
Cauchy-Schwarz, $Q(\bbm) \ge M(\bbm)^2$. Then
\[
Q(\bbm) + 1/2 \cdot (1 - Q(\bbm))^2 \ge \alpha - \eps.
\]
As $2\alpha-1-2\eps>0$ for any fixed $\alpha>1/2$ and for
sufficiently large $\lambda_0$, we obtain
$Q(\bbm) \ge (2 \alpha - 1 - 2 \eps)^{1/2}$.
Then also $(1-Q(\bbm))^2 \leq (1-(2\alpha-1-2\eps)^{1/2})^2$, so
\[
M(\bbm)^2 \ge \alpha - \eps - 1/2 \cdot (1 - Q(\bbm))^2 \ge (2 \alpha -1 - 2 \eps)^{1/2}. 
\]
\end{proof}

\begin{lemma}\label{lem:Fraction_bound}
There exist constants $\lambda_0,n_0,\Const_0,\const_0>0$ such that
for any $\lambda \ge \lambda_0$ and $n \ge n_0$, we have
\begin{align}\label{eqn:Fraction_bound_final_rescaled}
\P\Big(\sup_{\bu \in \ball^n(\bzero,1)} \Big[\frac{1}{n}\sum_{i=1}^n \ones\{\l
\< \bg_i, \bu \>\l \ge \lambda \}  \Big] \ge \Const_0 /\lambda^2 \Big) \le \exp\{-\const_0 n\},
\end{align}
where $\{\bg_i\}_{i\in[n]} \simiid \cN(0, \id_n)$ and $\ball^n(\bzero,1)$ is
the $n$-dimensional unit ball centered at $\bzero$.
\end{lemma}

\begin{proof}
Let $N(\eps) = \{\bv_1, \ldots, \bv_{|N(\eps)|} \}$ be an $\eps$-net of
$\ball^n(\bzero, 1)$, with $|N(\eps)| \le (3/\eps)^n$. That is, for any $\bv \in \ball^n(\bzero, 1)$, there exists $\bv_\star \in N(\eps)$, such that $\| \bv - \bv_\star \|_2 \le \eps$. Then, for any $\bu \in \ball^n(\bzero, 1)$, there exists a sequence $\{ \bu_j \}_{j \ge 0} \subseteq N(\eps)$, such that $\bu = \sum_{j=0}^\infty \eps^j \bu_j$. As a consequence, for any vector $\bg$, we have
\begin{align}
\ones\{ \l \< \bg, \bu \> \l \ge \lambda \} \le \sum_{j = 0}^\infty \ones\{ \eps^j \l \< \bg, \bu_j \>\l \ge \lambda/2^{j+1} \},
\end{align}
and hence 
\begin{align}\label{eqn:Fraction_bound_Chaining}
\sup_{\bu \in \ball^n(\bzero,1)} \frac{1}{n}\sum_{i=1}^n \ones\{\l \< \bg_i, \bu \>\l \ge \lambda \}  \le \sum_{j = 0}^\infty \sup_{\bu \in N(\eps)} \frac{1}{n}\sum_{i=1}^n \ones\{\l \< \bg_i, \bu \>\l \ge \lambda/[2(2 \eps)^j] \}.
\end{align}
We fix $\eps = 1/3$ throughout the proof. 

First, we show that for any $\chi \ge 4$ and $\delta \le 1$, we have (for $\eps = 1/3$)
\begin{align}\label{eqn:Fraction_bound_on_eps_net}
\P\Big( \sup_{\bu \in N(\eps)} \Big[\frac{1}{n}\sum_{i=1}^n \ones\{\l \< \bg_i, \bu \>\l \ge \chi \} \Big] \ge 16 / (\delta \chi^2) \Big) \le \exp\{ - n/\delta \}. 
\end{align}
For a fixed $\bu$ with $\|\bu\|_2=1$, applying the
Chernoff-Hoeffding inequality and denoting $\phi(\chi) = \P(\l G \l \ge \chi)$
for $G \sim \cN(0,1)$, we have for all $t>\phi(\chi)$
\[\P\Big( \frac{1}{n}\sum_{i=1}^n \ones\{\l \< \bg_i, \bu \>\l \ge \chi \}  \ge
t  \Big) \le \exp\{ - n \sD_{\kl}( t \| \phi(\chi)) \},\]
where for $a, b \in (0, 1)$, we define $\sD_{\kl}(a \| b) = a \log(a/b) + (1-a) \log((1-a)/(1-b))$ to be the relative entropy of Bernoulli distribution with parameters $a$ and $b$. 
Taking union bound over $\bu \in N(\eps)$ and applying $\phi(\chi) \geq
\phi(\chi/\|\bu\|_2)$,
\[\begin{aligned}
&\P\Big(\sup_{\bu \in N(\eps)} \Big[\frac{1}{n}\sum_{i=1}^n \ones\{\l \< \bg_i, \bu \>\l \ge \chi \}  \Big] \ge t  \Big) \\
\le& \sum_{\bu \in N(\eps)} \exp\{ - n \sD_{\kl}( t \| \phi(\chi/\| \bu \|_2))
\} \le \exp\{ n [-\sD_{\kl}( t \| \phi(\chi)) + \log(3/\eps)] \}.
\end{aligned}\]
Applying $\eps = 1/3$ and $\phi(\chi) \le 2 \exp\{-\chi^2/2\}$, we have
\begin{align*}
&-\sD_{\kl}(t \| \phi(\chi)) + \log (3/\eps) \\
&= t \log(\phi(\chi)) - t \log t + (1 - t) \log[1 + (t - \phi(\chi)) / (1-t)] + \log 9\\
&\le t \log (\phi(\chi)) - t\log t + (t - \phi(\chi)) + \log9 \\
&\le t (2 - \chi^2/2) + 3.
\end{align*}
Now take $t = 16/(\delta \chi^2) \ge \phi(\chi)$ for $\chi \ge 4$ and $\delta
\le 1$. Then $-\sD_{\kl}(t \| \phi(\chi)) + \log(3/\eps)  \le - 1/\delta$. This
proves Eq. (\ref{eqn:Fraction_bound_on_eps_net}). Applying Eq. (\ref{eqn:Fraction_bound_on_eps_net}) with $\chi_j = \lambda/[2(2\eps)^j]$ (requiring $\lambda \ge 8$ so that $\chi_j \ge 4$) and $\delta_j = (2 \eps)^j \le 1$ for $j \ge 0$, we have 
\[\P\Big( \sup_{\bu \in N(\eps)} \Big[\frac{1}{n}\sum_{i=1}^n \ones\{\l \<
\bg_i, \bu \>\l \ge \lambda/[2(2\eps)^j] \} \Big] \ge 64 (2 \eps)^j / \lambda^2
\Big) \le \exp\{ - n/(2\eps)^j  \}. \]
Finally, taking a union bound over $j \geq 0$ and applying
\[\sum_{j = 0}^\infty \exp\{ - n/(2\eps)^j \} \le \exp\{ -\const_0 n \},
\qquad
\sum_{j=0}^\infty 64 (2 \eps)^j / \lambda^2 \le \Const_0 / \lambda^2\]
for $n \ge n_0$ with some $\Const_0,\const_0, n_0>0$ concludes the proof.
\end{proof}

\begin{lemma}\label{lem:final_global_bound}
Fix any positive integer $k$. Then
there exist $\lambda_0,\Const_0>0$ and functions
$\const_0(\lambda),n_0(\lambda)>0$ such that for all
$\lambda \geq \lambda_0$ and $n \geq n_0(\lambda)$,
with probability at least $1 - e^{-\const_0(\lambda) n}$,
all points $\bbm \in (-1, 1)^n$ which satisfy $M(\bbm) + Q(\bbm) \ge 1.01$ and
$\nabla \cF(\bbm) = \bzero$ also satisfy $M(\bbm), Q(\bbm) \ge 1 - \Const_0 / \lambda^k$. 
\end{lemma}

\begin{proof}

Since $\bW \sim \GOE(n)$,
we can write $\bW = (\bG + \bG^\sT)/\sqrt{2}$, where $\bG = \{ G_{ij}\}_{i, j
\in [n]}$ with $G_{ij} \simiid \cN(0, 1/n)$. Note for any $\bbm \in (-1, 1)^n$,
we have $\| \bbm - \ones\|_2/\sqrt n \le 2$. According to Lemma
\ref{lem:Fraction_bound}, there exist constants $\Const_0$, $\const_0$,
$\lambda_0$, and $n_0$ such that for any $n \ge n_0$ and $\lambda \ge \lambda_0$, we have
\begin{equation}\label{eqn:Fraction_difference_of_m_ones}
\P\Big( \sup_{ \bbm \in (-1, 1)^n} \Big[ \frac{1}{n} \sum_{i=1}^n \ones\{ \l[
\bG (\bbm - \ones)]_i \l \ge 0.001 \lambda\} \Big] \ge \Const_0 / \lambda^2
\Big) \le \exp\{ - \const_0 n \}.
\end{equation}
Moreover, by a simple Chernoff bound, we have (denoting $\phi(\lambda) = \P(\l G
\l \ge \lambda) \le 2 \exp\{ - \lambda^2/2 \}$ for $G \sim \cN(0,1)$)
\begin{equation}\label{eqn:Infinity_norm_bound_W_ones}
\P\Big( \frac{1}{n} \sum_{i=1}^n \ones\{ \l [ \bG \ones]_i \l \ge 0.001
\lambda\}  \ge 3 \phi(0.001 \lambda) \Big) \le \exp\{ - n \phi(0.001 \lambda)
\}.
\end{equation}
The same bounds hold fir $\bG^\top$ in place of $\bG$.
For any $\bbm \in (-1, 1)^n$, note that $\nabla \cF(\bbm)=\bzero$ implies
\begin{align}\label{eqn:stationery_condition_reformulation}
\bbm = \tanh(\lambda^2 M(\bbm) \ones + \lambda\cdot \bW \ones + \lambda \cdot \bW(\bbm -
\ones) - \lambda^2 [1 - Q(\bbm)] \bbm).
\end{align}
When the bad events in (\ref{eqn:Infinity_norm_bound_W_ones}) and
(\ref{eqn:Fraction_difference_of_m_ones}) do not happen, and $M(\bbm)+Q(\bbm)
\geq 1.01$, then
at least $1 - 2 \Const_0 /\lambda^2 - 6 \phi(0.001 \lambda)$ fraction of
coordinates of $\bbm$ satisfy
\begin{equation}\label{eqn:fraction_of_coordinate_bound}
\begin{aligned}
m_j =& \tanh(\lambda^2 M(\bbm) + \lambda[\bW \ones]_j + \lambda [\bW(\bbm - \ones)]_j - \lambda^2 (1 - Q(\bbm)) m_j)\\
\ge& \tanh((M(\bbm) + Q(\bbm) - 1) \lambda^2 - 0.004 \lambda^2) \ge \tanh(0.006 \lambda^2)\\
\ge& 1 - e^{-0.006 \lambda^2},
\end{aligned}
\end{equation}
and the remaining coordinates of $\bbm$ satisfy $m_j \ge -1$. 
Therefore, for sufficiently large $\lambda_0$, we have 
\[
\begin{aligned}
&M(\bbm)=\< \bbm, \ones\> /n\\
&\ge (1 - 2 \Const_0 /\lambda^2 - 6 \phi(0.001 \lambda)) (1 - e^{-0.006 \lambda^2}) - 2\Const_0/\lambda^2 - 6 \phi(0.001 \lambda)\\
&\ge 1 - 6 \Const_0 /\lambda^2.
\end{aligned}
\]
Hence we also have
\[\| \bbm - \ones \|_2/\sqrt n \le ( 2 - 2 \<\bbm, \ones \>/n)^{1/2} \le (12
\Const_0)^{1/2} / \lambda,\]
with probability at least $1 - e^{-c_1(\lambda)n}$ for all $n \geq
n_1(\lambda)$.

In the following, we apply the above argument recursively to prove the lemma.
Suppose we already know that $\| \bbm - \ones \|_2 /\sqrt n \le K/\lambda^k$
for some constants $K > 0$ and $k \ge 1$, with probability $1-e^{-c_k(\lambda)n}$ for $n \geq n_k(\lambda)$. Applying again Lemma
\ref{lem:Fraction_bound}, for $\lambda_0$ sufficiently large such that
$\eps(\lambda) \equiv K/\lambda^k \le 2$ for all $\lambda \geq \lambda_0$,
we have with probability at most $\exp(-c_0n)$ that
\begin{equation}\label{eqn:Fraction_difference_of_m_ones_new}
\sup_{\| \bbm - \ones\|_2/\sqrt n \le \eps(\lambda)} \Big[ \frac{1}{n}
\sum_{i=1}^n \ones\{ [ \l \bG (\bbm - \ones)]_i \l \ge 0.001 \lambda\} \Big] \ge \Const_0 \eps(\lambda)^2 / \lambda^2,
\end{equation}
and similarly for $\bG^\top$.
When the bad events in (\ref{eqn:Infinity_norm_bound_W_ones}) and
(\ref{eqn:Fraction_difference_of_m_ones_new}) do not happen, for any $\bbm$ such
that $\| \bbm  - \ones\|_2/\sqrt n \le \eps(\lambda)$, by
(\ref{eqn:stationery_condition_reformulation}) at least $1 - 2 \Const_0
\eps(\lambda)^2/\lambda^2 - 6 \phi(0.001 \lambda)$ fraction of coordinates of
$\bbm$ satisfy (\ref{eqn:fraction_of_coordinate_bound}). Therefore, for
sufficiently large $\lambda_0$ (depending on $K$ and $k$), we have
\[M(\bbm)=\< \bbm, \ones\> /n \ge 1 - 6 \Const_0 \eps(\lambda)^2/\lambda^2.\]
Hence we have
\begin{align}
\| \bbm - \ones \|_2/\sqrt n \le ( 2 - 2 \<\bbm, \ones \>/n)^{1/2} \le (12
\Const_0)^{1/2} \eps(\lambda) / \lambda = (12 \Const_0)^{1/2}K/\lambda^{k+1}.
\end{align}
This holds with probability at least $1 - e^{-c_{k+1}(\lambda) n}$ for $n \geq
n_{k+1}(\lambda)$ for some $0 < c_{k+1}(\lambda) \le c_k(\lambda)$ and $n_{k+1}(\lambda) \ge n_k(\lambda)$. 

Thus for any fixed $k$, there are $k$-dependent constants $\lambda_0,C_0>0$ and functions $c_0(\lambda),n_0(\lambda)$ such
that $M(\bbm) \geq 1-C_0/\lambda^k$
with probability at least $1-e^{-c_0(\lambda)n}$ for $\lambda \geq \lambda_0$
and $n \geq n_0(\lambda)$.
Applying $Q(\bbm) \geq M(\bbm)^2$ by Cauchy-Schwarz, we obtain also the
statement for $Q(\bbm)$.
\end{proof}

\begin{proof}[Proof of Proposition \ref{PROP:LOCALIZATION}]

Applying Lemma \ref{lem:initial_global_bound} with $\alpha=2/3$, for any
$\lambda \ge \lambda_0$ and $n \ge n_0$, with probability at least $1 -
e^{-\const_0 n}$, all $\bbm \in (-1, 1)^n$ with $\cF(\bbm) \le -\lambda^2/3$
satisfy $Q(\bbm) \ge M(\bbm)^2 \ge (1/3 - 6/\lambda - 4/\lambda^2)^{1/2}$. For
$\lambda_0$ large enough and $M(\bbm) \geq 0$, this implies $Q(\bbm) + M(\bbm) \ge 1.01$. Hence the proposition holds by Lemma \ref{lem:final_global_bound}. 

\end{proof}

\section{Proof of Theorem \ref{THM:NAIVE_FAIL}}\label{sec:proof_thm_naive_fail}

\begin{proof}[Proof of Theorem \ref{THM:NAIVE_FAIL}]
Let $\bbm_\star $ be any critical point of the TAP free energy $\cF = \cF_{\lambda, \lambda}$, and
let $\bbm$ be any critical point of $\cF_\mathrm{MF}$. By the conditions
$\bzero = \nabla \cF(\bbm_\star )$ and $\bzero = \nabla \cF_\mathrm{MF}(\bbm)$, we have
\begin{align}
\bbm_\star &=\tanh(\lambda \cdot \bY\bbm_\star -\lambda^2[1-Q(\bbm_\star )]\bbm_\star ),
\label{eq:mstarcondition}\\
\bbm&=\tanh(\lambda \cdot \bY\bbm)\label{eq:mcondition}.
\end{align}

For
constants $C_0,t>0$ to be chosen later, consider the event $\mathcal{E}$ where 
\[\mathcal{E} = \Big\{ \sup_{\bu \in \ball^n(\bzero, 1)} \Big[\frac{1}{n}\sum_{i=1}^n
\ones\{|(\bW\bu)_i| \ge t/\sqrt{n} \}  \Big]<C_0/t^2 \Big\},\]
and $B^n(\bzero,1)$ denotes the unit ball around $\bzero$.
Note that $\bW = (\bG + \bG^\sT)/\sqrt{2n}$ with $G_{ij} \simiid \cN(0, 1)$ for $1 \le i, j \le n$, so that
\[
\begin{aligned}
&\sup_{\bu \in \ball^n(\bzero, 1)} \Big[\frac{1}{n}\sum_{i=1}^n
\ones\{|(\bW\bu)_i| \ge \frac{t}{\sqrt{n}} \}  \Big] \\
\le& \sup_{\bu \in \ball^n(\bzero, 1)} \Big[\frac{1}{n}\sum_{i=1}^n
\ones\{|(\bG \bu)_i| \ge \frac{t}{ \sqrt 2}\}  \Big] + \sup_{\bu \in \ball^n(\bzero, 1)} \Big[\frac{1}{n}\sum_{i=1}^n
\ones\{|(\bG^\sT \bu)_i| \ge \frac{t}{\sqrt 2} \}  \Big].
\end{aligned}
\]
Then by Lemma
\ref{lem:Fraction_bound},  for some constant $c_0, C_0>0$ and for any $t$
sufficiently large, $\mathcal{E}$ holds with probability at least
$1-e^{-c_0 n}$. Define the (random) index set
\[\mathcal{I}=\Big\{i \in \{1,\ldots,n\}:
|(\bW\bbm)_i|<t \text{ and } |(\bW\bbm_\star )_i|<t\Big\}.\]
As $\bbm/\sqrt{n} \in \ball^n(\bzero,1)$ and similarly for $\bbm_\star $,
we have on $\mathcal{E}$ that $|\mathcal{I}|/n \geq 1-2C_0/t^2$.
Applying $|(\bY\bbm)_i| \leq |\lambda|+|(\bW\bbm)_i|$ by (\ref{eq:Z2model}),
for $i \in \mathcal{I}$ we have
\begin{align*}
|\lambda \cdot (\bY\bbm_\star )_i-\lambda^2[1-Q(\bbm_\star )]m_{\star,i}|
&\leq 2\lambda^2+\lambda t,\\
|\lambda \cdot (\bY\bbm)_i|&\leq \lambda^2+\lambda t.
\end{align*}
Then taking the difference of (\ref{eq:mstarcondition}) and
(\ref{eq:mcondition}) and applying the lower bound $\tanh'(x) \geq
c(\lambda,t)$ for all $|x| \leq 2\lambda^2+\lambda t$ and a constant
$c(\lambda,t)>0$, we have
\[|m_{\star,i}-m_i| \geq c(\lambda,t)\Big|\lambda \cdot (\bY\bbm_\star -\bY\bbm)_i
-\lambda^2[1-Q(\bbm_\star )]m_{\star,i}\Big|.\]
Then
\[\lambda^4[1-Q(\bbm_\star )]^2m_{\star,i}^2 \leq 2c(\lambda,t)^{-2}|m_{\star,i}-m_i|^2
+2\lambda^2 \cdot (\bY\bbm_\star -\bY \bbm)_i^2.\]
Summing over $i \in \{1,\ldots,n\}$ and using the trivial bound $m_{\star,i}^2 \leq
1$ for $i \notin \mathcal{I}$, we obtain
\begin{align*}
&\lambda^4[1-Q(\bbm_\star )]^2\|\bbm_\star \|_2^2\\
&\leq \lambda^4[1-Q(\bbm_\star )]^2(n-|\mathcal{I}|)
+2c(\lambda,t)^{-2}\|\bbm_\star -\bbm\|_2^2+2\lambda^2\|\bY\|_{\mathrm{op}}
^2\|\bbm_\star -\bbm\|_2^2.
\end{align*}
Suppose $\lambda$ is sufficiently large such that the conclusion
$|Q(\bbm_\star )-q_\star |<q_\star /4$ from \ref{prop:mainprop}
holds with probability $1-\delta/4$.
Choose $t$ large enough such that $2C_0/t^2<q_\star /2$, so $|\mathcal{I}|/n>1-q_\star /2$
on the event $\mathcal{E}$. Note also that with probability $1-e^{-c_0 n}$ for some constant $c_0 >0$, we have $\|\bY\|_\mathrm{op} \leq \lambda+3$. Then on a
combined event of probability $1-\delta/2$, for any critical points $\bbm_\star $
and $\bbm$ as above, we obtain
\[\frac{1}{n}\|\bbm_\star -\bbm\|_2^2 \geq c(\lambda)\]
for some constant $c(\lambda)>0$.

Consider now any global minimizer $\bbm_\star $ of $\cF$. Note that $-\bbm_\star $ is
also a global minimizer of $\cF$, so on this combined event, we also have
$n^{-1}\|\bbm_\star +\bbm\|_2^2 \geq c(\lambda)$. Then 
\begin{align}
\frac{1}{n^2}\|\bbm_\star \bbm_\star ^\sT-\bbm\bbm^\sT\|_F^2
&=\frac{1}{n^2}\left(\|\bbm_\star \|_2^4+\|\bbm\|_2^4-2|\< \bbm,\bbm_\star  \>|^2\right)
\nonumber\\
&\geq \frac{1}{2n^2}\|\bbm_\star -\bbm\|_2^2\|\bbm_\star +\bbm\|_2^2 \geq c(\lambda)^2/2.
\label{eq:mstarmclose}
\end{align}
On the other hand, note that
\[\cF(\bbm_\star ) \leq
\cF(\bx)=-\frac{\lambda^2}{2}-\frac{\lambda}{2n}\<\bW,\bx\bx^\sT \>
\leq -\frac{\lambda^2}{2}+\lambda ,\]
where the last inequality holds with probability at least $1-e^{-n}$ by the fact
that $\<\bW,\bx\bx^\sT\> \sim \cN(0,2n)$. For $\lambda$ sufficiently large, we have $-\lambda^2/2 + \lambda \le -\lambda^2/3$. Then Theorem \ref{thm:TAPBayes}
implies, with probability at least $1-\delta/2$,
\[\frac{1}{n^2}\|\bbm_\star \bbm_\star ^\sT-\widehat{\bX}_\mathrm{Bayes}\|_F^2
<c(\lambda)^2/8.\]
Combining this with (\ref{eq:mstarmclose}) and setting
$\eps(\lambda)=c(\lambda)^2/8$, with probability at least $1-\delta$, for any
critical point $\bbm$ of $\cF_{\mathrm{MF}}$ we obtain
\[\frac{1}{n^2}\|\bbm\bbm^\sT-\widehat{\bX}_\mathrm{Bayes}\|_F^2
\geq \frac{1}{2n^2}\|\bbm_\star \bbm_\star ^\sT-\bbm\bbm^\sT\|_F^2
-\frac{1}{n^2}\|\bbm_\star \bbm_\star ^\sT-\widehat{\bX}_\mathrm{Bayes}\|_F^2
>\eps(\lambda).\]
\end{proof}

\section*{Acknowledgments}
This work was partially supported by grants NSF DMS-1613091, NSF CCF-1714305,
NSF IIS-1741162, ONR N00014-18-1-2729, and NSF DMS-1916198.
ZF was partially supported by a Hertz Foundation Fellowship. SM was partially supported by Office of Technology Licensing Stanford Graduate Fellowship.

\bibliography{references}{}
\bibliographystyle{alpha}

\end{document}